\documentclass{article}
\usepackage{enumitem}   
\usepackage{amsmath}
\usepackage{amssymb}
\usepackage{natbib}
\usepackage{amsthm}
\usepackage{graphicx}
\usepackage{dsfont}
\usepackage{color}
\usepackage{hyperref}
\usepackage{thmtools}

\usepackage{url}
\usepackage[ruled,vlined,linesnumbered,noend]{algorithm2e}
\usepackage{algpseudocode}
\usepackage[a4paper,margin=3cm]{geometry}
\declaretheorem[name=Theorem]{theorem}
\newtheorem{proposition}[theorem]{Proposition}
\newtheorem{lemma}[theorem]{Lemma}
\newtheorem{corollary}[theorem]{Corollary}
\newtheorem{definition}[theorem]{Definition}
\newtheorem{assumption}{Assumption}
\newtheorem{remark}[theorem]{Remark}

\def\gapf{{\Delta_f}}
\def\Kval{{K}}
\def\alphaval{{\alpha}}
\def\alphavalcont{{\alpha_T}}
\newcommand{\alphavalcontf}[1]{\alpha_{#1}}
\def\constalpha{{C_{\alphaval}}}
\def\Aconst{{\mathcal{A}_n}}
\def\Bconst{{\mathcal{B}_n}}
\def\R{{\mathbb{R}}}

\def\N{{\mathbb{N}}}

\def\1{\mathds{1}}
\def\cupp{c_{upp}}
\def\cinf{c_{inf}}

\newcommand{\A}{\mathbb{A}}
\def\Abar{\A_n}

\def\xcont{(x_t)_{t \in \R_+}}
\def\xz{(x_t,z_t)_{t \in \R_+}}

\def\xdisctilde{\{\tilde{x}_k \}_{k\in \N}}
\def\xcont{(x_t)_{t\ge 0}}
\def\mart{(M_t)_{t\ge 0}}
\def\tk{\{ T_k \}_{k\in \N}}

\def\Ccont{\mathcal{C}_n}

\def\Hconst{H(\alphaval)}
\def\xrescaled{x_s^\gamma}
\def\xrescaledprev{x_{s^-}^\gamma}
\def\zrescaled{z_s^\gamma}
\def\Nrescaled{N_s^\gamma}
\def\Mrescaled{M_s^\gamma}

\renewcommand{\P}{\mathbb{P}}
\newcommand{\bpar}[1]{\left(#1\right)}

\newcommand{\tkvar}[1]{\{ T_k \}_{k\in \{1,\dots,#1 \}}}
\newcommand{\E}[1]{\mathbb{E}\left[#1 \right]}

\newcommand{\norm}[1]{\left\lVert#1\right\rVert}
\newcommand{\dotprod}[1]{\left< #1\right>}

\newcommand{\ceil}[1]{\left\lceil {#1} \right\rceil}
\newcommand{\floor}[1]{\left\lfloor {#1} \right\rfloor}

\newcommand{\footremember}[2]{%
    \footnote{#2}
    \newcounter{#1}
    \setcounter{#1}{\value{footnote}}%
}
\newcommand{\footrecall}[1]{%
    \footnotemark[\value{#1}]%
} 
\newenvironment{sproof}{%
  \proof}{\endproof}
\newcommand{\off}[1]{}

\newcommand{\ind}{\perp\!\!\!\!\perp}
\newcommand{\bigO}{\mathcal{O}}

\def\argmin{\textup{argmin}\,}

\title{Continuized Nesterov Momentum Achieves the $\mathcal{O}(\varepsilon^{-7/4})$ Complexity in Smooth Nonconvex Optimization}
\author{Julien Hermant\footremember{corr_author}{corresponding author: julien.hermant@math.u-bordeaux.fr}\footremember{Bdx}{Univ. Bordeaux, Bordeaux INP, CNRS, IMB, UMR 5251, F-33400 Talence, France} \and Jean-François Aujol\footrecall{Bdx} \and Charles Dossal\footremember{Tls1}{IMT, Univ. Toulouse, INSA Toulouse, Toulouse, France} \and Lorick Huang\footrecall{Tls1} \and  Aude Rondepierre\footrecall{Tls1} \and Irène Waldspurger\footremember{irène}{CNRS, Univ. Paris Dauphine, Inria Mokaplan, France} }

\date{}

\begin{document}

\maketitle

\begin{abstract}%
For first-order optimization of non-convex functions with Lipschitz continuous gradient and Hessian, the best known complexity for reaching an 
$\varepsilon$-approximation of a stationary point is $\mathcal{O}(\varepsilon^{-7/4})$. The existing algorithms achieving this bound are based on momentum, but are always complemented with safeguard mechanisms that erase the accumulated momentum if a certain condition is violated. 
Whether such momentum-control mechanisms are fundamentally necessary has remained an open question. We show that randomizing the parameters enables to achieve this complexity in expectation when using momentum without any of such mechanisms, and we improve the numerical constant factor of the bound in the case of a large enough number of iterations. From an analysis perspective, we do so
leveraging the continuized method, that interprets the algorithm as a realization of a continuous-time stochastic differential equation (SDE) involving a Poisson process. We show that this SDE converges in probability to the Heavy Ball ordinary differential equation when the stepsize goes to zero, paralleling the behavior of more classical instances of Nesterov momentum. 
\end{abstract}


\section{Introduction}\label{sec:intro}
Many real-world problems can be formulated as the minimization of a function $f: \R^d\to \R $, that is, as $\min_{x \in \R^d} f(x)$. Yet, without structural assumptions - such as convexity of $f$ - computing a global minimizer is generally intractable \cite[Section 2.2]{danilova2022recent}. In this work, assuming $f$ is lower-bounded and differentiable, our goal is thus to find a stationary point of $f$, \textit{i.e.} approximate $x^\ast \in \R^d$ such that
$$\nabla f(x^\ast) = 0. $$ 
The rising computational cost and expanding data scales in modern applications-ranging from machine learning to signal processing - have made \textit{first-order optimization algorithms} widely used. These algorithms rely on gradient information to compute their iterates with gradient descent being a classic example. Among existing acceleration techniques designed to improve their convergence speed \cite{hinton2012rmsprop,kingma2014adam}, one of the most influential is \textit{momentum}, whose idea dates back to the seminal work of \cite{POLYAK19641}.

\noindent
\textbf{Nesterov momentum \:}
The \textit{Nesterov momentum} \eqref{eq:nm} algorithm \cite{nesterov1983method}, sometimes called \textit{Nesterov accelerated gradient}, is known to achieve optimal convergence bounds among first-order algoritfhms in various settings in convex optimization \cite{nemirovskij1983problem,nesterovbook}, see Appendix~\ref{app:convex_acceleration} for a more detailed discussion. There exist several ways to write Nesterov's momentum, see \citep{defazio2019curved} for a discussion in the strongly-convex case, or see \cite[Appendix B.2]{hermantgradient} in a more general setting.
Among these different formulations, we consider the following 

\begin{align}\label{eq:nm}\tag{NM}\left\{
    \begin{array}{ll}
        \tilde{y}_k &= \tilde{x}_k+ \alpha_k(\tilde{z}_{k} - \tilde{x}_k),  \\
        \tilde{x}_{k+1} &=  \tilde{y}_k  - \gamma \nabla f(\tilde{y}_k),\\
        \tilde{z}_{k+1} &=   \tilde{z}_{k} + \beta_k(\tilde{y}_k - \tilde{z}_{k}) - \gamma' \nabla f(\tilde{y}_k).
    \end{array}
\right.
\end{align} 

We note that $\xdisctilde$ in \eqref{eq:nm} reduces to gradient descent for the choice $\alpha_k = 0$, for all $k \in \N$. Because a wide range of applications is formulated as a non-convex optimization problem \cite{hardt2014understanding,candes2015phase,dauphin2014identifying,bhojanapalli2016global,ge2016matrix,ge2017no,li2018visualizing}, an important question is whether the benefits of momentum extend beyond convex optimization. Empirically, it is often observed that the use of momentum improves the speed of convergence for non-convex minimization tasks such as neural network training \cite{sutskever2013importance,he2016deep}. 
Theoretically, acceleration has been shown in some mild non-convex settings where finding a global minimizer remains tractable, see \textit{e.g.} \cite{hinder2020near,hermant2024study,gupta2024nesterov}. In this work, we only adopt \textit{smoothness} assumptions, which fall outside regimes where guarantees of global minimization can be obtained.

\noindent
\textbf{Lipschitz Gradient \:}
When the objective $f$ has a Lipschitz continuous gradient, finding an $\varepsilon$-stationary point, \textit{i.e.} $x\in \R^d$ such that $\norm{\nabla f(x)} \le \varepsilon$, can be achieved by gradient descent in at most $\bigO(\varepsilon^{-2})$ gradient evaluations \cite[ (1.2.22)]{nesterovbook}. This bound is worst-case optimal among first-order algorithms \cite{lowerboundI}. The momentum mechanism thus does not provably accelerate gradient descent, achieving a similar rate \cite{ghadimi2016accelerated}. This impossibility result suggests that the class of functions with merely Lipschitz-continuous gradients does not provide an adequate framework for understanding the acceleration mechanisms observed in practice. This motivates the study of higher-order smoothness assumptions.

\noindent
\textbf{Lipschitz Hessian \:} 
Assuming in addition that the Hessian is Lipschitz continuous, the aforementioned bound can be improved to $\tilde \bigO(\varepsilon^{-7/4})$ \cite{convexguilty,momentumsaddle}, up to some logarithmic factor further removed by \cite{li2023restarted} to achieve the bound $\bigO(\varepsilon^{-7/4})$.
 This class of functions therefore admits accelerated bounds, while still encompassing important non-convex problems. Examples include the symmetric low-rank matrix factorization problem, as long as the domain is a compact set \cite{jin2017escape}, or squared loss regression of a neural network whose activation functions are with Lipschitz first and second derivatives, \textit{e.g.} if using SoftPlus activations \citep[Proposition 2]{renaud2025provably}. 
 
The existing first-order algorithms achieving the $\bigO(\varepsilon^{-7/4})$ or $\tilde \bigO(\varepsilon^{-7/4})$ complexity rely on the following update as a core mechanism
 \begin{equation}\label{eq:nm_prime}\tag{NM'}
\left\{
\begin{alignedat}{2}
    &\tilde y_k = \tilde x_k + \alpha_k (\tilde x_k - \tilde x_{k-1}),  \\
    &\tilde x_{k+1} = \tilde x_k - \gamma \nabla f(\tilde y_k).
\end{alignedat}
\right.
\end{equation}
 \eqref{eq:nm_prime} can be described as a Nesterov momentum algorithm, with fewer degrees of freedom compared with the $3$-sequence version \eqref{eq:nm}. This can be seen as \eqref{eq:nm} can be written as a $2$-sequence form that reduces to \eqref{eq:nm_prime} for some parameter choices
 \citep[Proposition 11]{hermant2024study}, or by considering that the high-resolution limit-ODE associated to \eqref{eq:nm} and \eqref{eq:nm_prime} can be written in the same form \citep[Proposition 27]{hermant2025continuized}.    
 However, these methods achieving the improved complexity do not rely on \eqref{eq:nm_prime} only; they also need to perform a \textit{safety check} at each iteration. If it fails at iteration $k$,
 the \textbf{momentum built up so far is erased}, which is practically done by setting $\alpha_k =0$.
 At a conceptual level, these procedures can be summarized as follows:
 
 \begin{algorithm}[H]
\caption{Structure of existing algorithms achieving the $\bigO(\varepsilon^{-7/4})$ complexity}
\label{alg:nce}
\DontPrintSemicolon
\If{Safety check succeeds}{
    Run update rule \eqref{eq:nm_prime} 
}
\Else{
Erase accumulated momentum (+ potentially trigger alternative mechanisms)
}
\end{algorithm}
 
 In \cite{convexguilty,momentumsaddle}, the safety check monitors the level of non-convexity between consecutive iterates.  When it exceeds a given threshold:
 \[f(\tilde x_k) \le  f(\tilde y_k) + \dotprod{ \nabla f(\tilde y_k), \tilde x_k - \tilde y_k }- \frac{\ell}{2} \norm{\tilde x_k - \tilde y_k}^2, \quad \ell > 0, \]
accumulated momentum is erased, and a negative-curvature exploitation step is triggered as an alternative mechanism.
An alternative is provided in \cite{li2023restarted}, which employs a safety check based on the length of the trajectory residual. Exceeding a given threshold
 \[
        k \displaystyle\sum_{t=0}^{k-1}
        \|\tilde x_{k+1} - \tilde x_k\|^2 > B^2, \quad B> 0,
    \] it activates a restart mechanism, thus erasing the accumulated momentum.
    See detailed algorithmic descriptions and further discussion in Appendix~\ref{app:alg}.

 In both cases, failing the safety check triggers a \textbf{momentum reset}. Momentum resets, or momentum restarts, can be interesting. They have been shown to provide benefits in the convex setting, for example by enhancing practical behavior in some situations \cite{o2015adaptive,renaud2025provably}, or by helping in the design of parameter-free algorithms \cite{aujol2024parameter}. However, such mechanisms are not mandatory to achieve theoretical acceleration in the convex setting. More generally, mechanisms that restart or control momentum do not seem to be necessary for momentum methods to perform well in many practical non-convex problems. It may be somewhat contradicted by the existing results, which suggest that these momentum control are needed to ensure accelerated convergence. 
This raises the following question:
\begin{center}
    \textbf{Does momentum fundamentally requires control mechanisms to achieve the $\bigO(\varepsilon^{-7/4})$ complexity?}
\end{center}


In a recent result, \cite{okamura2024primitive} shows that in the case of the vanilla Heavy Ball ordinary differential equation (ODE)
\begin{equation}\label{eq:HB}\tag{HB}
    \ddot{x}_t + \alpha\dot{x}_t + \nabla f(x_t) = 0,
\end{equation}
a suitable averaging of its solution $\overline{x}_t$ verifies $ \min_{0\le t \le T} \norm{\nabla f(\overline{x}_t)} = \bigO(T^{-4/7})$, which in turn implies an $\varepsilon$-stationary solution is achieved in at most $\bigO(\varepsilon^{-7/4})$ execution time. Importantly, the Heavy Ball ODE can be viewed as a continuous-time analogue of momentum algorithms \cite{suboydcandes,siegel2021accelerated}. 
It means that this continuous-time momentum dynamic does not require such mechanisms, and naturally raises the question of whether a similar result can be established in the discrete setting. 
To the best of our knowledge, no such discrete counterpart is currently known.\footnote{In an independent work, \cite{ushiyama2026restart} use the Performance Estimation Problem framework to show that a deterministic discretization of \eqref{eq:HB} achieves this complexity without restarts. Yet, their algorithm still incorporates a momentum-control mechanism, that does not set its value to zero but reduces it when the velocity $\norm{\tilde x_{k+1}-\tilde x_k}$ exceeds a threshold.}

\noindent
\textbf{The continuized method\:}
Even, Berthier et al. \cite{even2021continuized} introduce the continuized Nesterov equations
\begin{align}\label{eq:nest_continuized}\tag{CNE}\left\{
    \begin{array}{ll}
        dx_t &= \eta(z_t-x_t)dt - \gamma \nabla f(x_{t^-}) dN_t, \\
        dz_t &= \eta'(x_t-z_t)dt - \gamma' \nabla f(x_{t^-}) dN_t,
    \end{array}
\right.
\end{align}
where $\eta,\eta',\gamma,\gamma'$ are real constants,  $t \in \R_+$ and $N_t$ is a Poisson process with intensity $dt$. The fundamental feature of this process is that it can be analyzed through continuous-time Lyapunov strategies using tools from stochastic calculus theory, while still yielding a computable algorithm that writes as \eqref{eq:nm}, with the specificity that the sequences $\{\alphaval_k\}_{k\in \N}$ and $\{\beta_k\}_{k\in \N}$ are random. These stochastic sequences are simple, in the sense that they can be readily generated at the beginning of the algorithm. Later, \cite{hermant2025continuized} argued that, in non-convex settings, this continuized system helps narrow the gap between convergence results derived for momentum ODEs, such as \eqref{eq:HB}, and those obtained for their algorithmic counterparts. Our motivation combines this perspective with the absence of a discrete counterpart to the result of Okamura et al.~\cite{okamura2024primitive}.

\subsection{Contributions}
We draw inspiration from the ODE-based analysis of \cite{okamura2024primitive} to analyze the solution of \eqref{eq:nest_continuized}, which yields our main result, stated as follows (see Theorem~\ref{thm:hess_lip} for the formal version).
\begin{theorem}[Informal version] Let $f$ be with Lipschitz gradient and Hessian, and let $\tilde \varepsilon \in (0,1/2)$. For a number of iterations $n$ large enough, \eqref{eq:nm} with a suitable stochastic parameterization outputs a point $\tilde x$
    $$ \E{\1_{\Abar}\norm{\nabla f(\tilde x)}}  = \bigO(n^{-4/7}),$$
    where $\Abar$ is a subset of the realizations that satisfies $\mathbb{P}(\Abar) \ge 1- \tilde \varepsilon$.
\end{theorem}
It further implies that a stochastic parameterization of \eqref{eq:nm} is able to achieve a point that satisfies $ \E{\norm{\1_{\Abar}\nabla f(\tilde x)}} \le \varepsilon$ in at most $\bigO(\varepsilon^{-7/4})$ gradient evaluations,  without any momentum-reset mechanisms. We note that for $n$ large enough, the hidden numerical factor of our bound improves over the best existing bound from \cite{li2023restarted}.  Conceptually, our result indicates that \eqref{eq:nm} by itself can achieve the complexity $\bigO(\varepsilon^{-7/4})$, at the cost of randomizing parameters, and thus, of a result in expectation. We also show that under Lipschitz gradient only, our algorithm recovers in expectation the best existing complexity $\bigO(\varepsilon^{-2})$.

The presence of stochastic parameters in the algorithm raises the question of the conceptual relation between the continuized method and more classical, deterministic parameterizations of \eqref{eq:nm}. We show that in the vanishing stepsize limit, a time-rescaled version of \eqref{eq:nest_continuized} converges in probability to \eqref{eq:HB}.
This parallels the fact that classical deterministic parameterizations of \eqref{eq:nm} converge to the same ODE in this vanishing stepsize limit \cite{suboydcandes,shi2018understanding}.
In this sense, \eqref{eq:nest_continuized} may be seen as an alternative, stochastic way to discretize the Heavy Ball \eqref{eq:HB} ODE.  

\subsection{Related works}

\noindent
\textbf{First order acceleration under Lipschitz gradient and Hessian \:} The $\tilde \bigO(\varepsilon^{-7/4})$ complexity for first-order algorithms was first achieved by \cite{convexguilty}.
The proposed nested-loop algorithm combines \eqref{eq:nm_prime} with negative-curvature exploitation and the minimization of a regularized surrogate function. This approach is subsequently simplified in \cite{momentumsaddle}, where the algorithm reduces to \eqref{eq:nm_prime} combined with negative-curvature steps.
Replacing the negative curvature exploitation with a restart mechanism, \cite{li2023restarted} achieves the $\bigO(\varepsilon^{-7/4})$ bound, improving by a logarithmic factor. This restart-based line of work has since been extended to a parameter-free algorithm \cite{Marumo2022ParameterFreeAG}, a forward-backward version
\cite{renaud2025provably} and a universal algorithm under Hölder continuous Hessian \cite{marumo2025universal}. A different line of work, based on online-learning techniques rather than momentum, was initiated by \cite{cutkosky2023optimal}, who achieved the complexity bound $\tilde\bigO(\varepsilon^{-7/4})$ under the supplementary assumption that $f$ is Lipschitz. Their idea was later extended by \cite{jiang2025improved}, improving the dependence on $\varepsilon$ to $\bigO(d^{1/4}\varepsilon^{-13/8})$ at the cost of an explicit dimension dependence, without this supplementary Lipschitzness assumption on $f$. These online-learning based methods are somewhat intricate, relying on nested algorithms and/or nested loops, but they show that in some settings, alternatives to momentum can achieve this accelerated bound. The first known lower-bound on the number of gradient evaluations needed to find an $\varepsilon$-stationary point in this setting was $\bigO(\varepsilon^{-12/7})$ \cite{lowerboundII}, and was recently improved to $\bigO(\varepsilon^{-7/4})$ \cite{zhou2026sharp}, closing the gap with the upper-bound.

\noindent
\textbf{Second-order stationary points \:} In contrast to our setting,
several works have focused on second-order stationary points,
which under suitable assumptions correspond to local minimizers, see a formal definition in Appendix~\ref{app:second_order_critic}. With probability $1-\delta$, such a point is found in at most $\bigO(\varepsilon^{-7/4}\log(\frac{d}{\delta \varepsilon}))$ gradient and Hessian-vector product evaluations \cite{agarwal,carmon2018accelerated}, or with gradient evaluation only \cite{momentumsaddle,li2023restarted}.
 Procedures based on first-order information have been designed that transform algorithms for finding first-order stationary points into ones that find second-order stationary points \cite{xu2018first,allen2018neon2}. 
 Even though designing methods that achieve second-order stationary points efficiently is an important avenue of research, we do not consider this direction in this work.

\noindent
\textbf{Continuized Nesterov\:} The continuized Nesterov equations are introduced in \cite{even2021continuized}, motivated by the study of asynchronous algorithms. They show that the continuized Nesterov algorithm recovers, in expectation, the existing rate of the classic Nesterov Momentum for (strongly)-convex functions, further generalized to (strongly)-quasar convex functions \cite{wang2023continuizedaccelerationquasarconvex,hermant2025continuized}.   This method has known some extensions in the specific field of decentralized asynchronous algorithms
\cite{nabli2023dadao,nabli2023a2cid2}, but remains largely unexplored in optimization.
\section{Background}
\noindent
In this work, we fix an underlying probability space $(\Omega,\mathcal{F},\mathbb{P})$. We denote by $\mathcal{E}(1)$ the exponential distribution with parameter $1$, $\Gamma(k,1)$ the Gamma distribution with shape parameter $k$ and rate $1$ and $\mathcal{P}(\lambda)$ the Poisson distribution with parameter $\lambda$. For $\phi : \R \to \R$, $t \in \R$, we note $\phi(t^-) := \lim_{\substack{s < t \\s \to t}} \phi(s)$ the left-limit. $\norm{\; \cdot\;}$ denotes the Euclidean norm for vectors and $\norm{\; \cdot\;}_2$ denotes the spectral norm for matrices. We use the notations $f^\ast := \min_{x \in \R^d} f(x)$ and $\gapf :=f(x_0)-f^\ast $.

We consider the two following assumptions.
\begin{assumption}\label{ass:l_smooth}(Lipschitz gradient)
    $f$ is such that $ \forall x,y\in \R^d,~  \norm{\nabla f(x)-\nabla f(y)}\le L\norm{x-y}$.
\end{assumption}
\begin{assumption}\label{ass:hess_lip}(Lipschitz Hessian)
    $f$ is such that $ \forall x,y\in \R^d,~  \norm{\nabla^2 f(x)-\nabla^2 f(y)}_2\le L_2\norm{x-y}$.
\end{assumption}
For functions satisfying Assumption~\ref{ass:l_smooth} and/or Assumption~\ref{ass:hess_lip}, a critical point could be a non-global minimizer, a saddle point, or a maximizer.
 As mentioned in Section~\ref{sec:intro}, Assumption~\ref{ass:l_smooth} alone is sufficient to find an $\varepsilon$-stationary point with first-order methods \cite{nesterovbook,ghadimi2016accelerated}, but the complexity of gradient descent cannot be improved \cite{lowerboundI}. Considering  Assumptions~\ref{ass:l_smooth}-\ref{ass:hess_lip} together allows for such improvement.
In existing analysis \cite{momentumsaddle,li2023restarted}, Assumption~\ref{ass:hess_lip} enables working with local quadratic approximations, on which momentum is known to be effective \cite{o2019behavior}. Our analysis does not rely on such approximations.
\subsection{Continuized Nesterov Algorithm}
In this section, we briefly describe the process studied in this work, and the resulting algorithm. We refer to \cite[Section 3]{hermant2025continuized} for a more detailed presentation of the continuized framework.
An intuitive way to present the process is as follows: consider a sequence of random times $\tk$ such that $T_0 = 0$ and $T_{k+1}-T_k \overset{i.i.d}{\sim} \mathcal{E}(1)$ for $k \in \N$. For some constant parameters $\eta, \eta', \gamma, \gamma'$, we define a continuous-time process $\xz$ on each interval $(T_k,T_{k+1})$ as the solution of the following equation
\begin{align*}\left\{
    \begin{array}{ll}
        d{x}_t &= \eta(z_t-x_t)dt, \\
        d{z}_t &= \eta'(x_t-z_t)dt.
    \end{array}
\right.
\end{align*}
At $t=T_k$, the process jumps by performing gradient steps,
\begin{equation*}
    x_{T_k} = x_{T_k^{-}} - \gamma\nabla f(x_{T_k^-}), \quad z_{T_k} = z_{T_k^{-}} - \gamma'\nabla f(x_{T_k^-}).
\end{equation*}
This defines the \textit{continuized Nesterov equation}, where we assume $x_0 = z_0$.
It can be written in a more compact form as the following stochastic differential equation 
\begin{align}\label{eq:nest_continuized}\tag{CNE}\left\{
    \begin{array}{ll}
        dx_t &= \eta(z_t-x_t)dt - \gamma \nabla f(x_{t^-}) dN_t, \\
        dz_t &= \eta'(x_t-z_t)dt - \gamma' \nabla f(x_{t^-}) dN_t,
    \end{array}
\right.
\end{align}
where $dN_t = \sum_{k\ge 0} \delta_{T_k}(dt)$ is a Poisson point measure with intensity $dt$. It mixes the continuous component, the $dt$ factor, with the gradient component that acts at discrete random times, the $dN_t$ factor. Replacing $dN_t$ by $dt$ in \eqref{eq:nest_continuized} would yield a momentum ODE, close to the classic Heavy-Ball ODE, see \citep[Appendix E.2]{hermant2025continuized}.

A remarkable property of \eqref{eq:nest_continuized} is that, upon defining the sequences $\tilde{y}_k := x_{T_{k+1}^-}$, $\tilde{x}_k := x_{T_{k}}$ and $\tilde{z}_k := z_{T_{k}}$, $k \in \N$, it satisfies a recursive relation that takes the form of a Nesterov momentum algorithm, with stochastic parameters depending on the random times $\tk$. 
\begin{proposition}[\cite{hermant2025continuized}, Proposition 5]\label{prop:discretization_constant_param}
Let $\xz$ follow \eqref{eq:nest_continuized} with $\eta + \eta' > 0$ and with underlying jumping times $\tk$. Define $\tilde{y}_k := x_{T_{k+1}^-}$, $\tilde{x}_{k+1} := x_{T_{k+1}}$ and $\tilde{z}_{k+1} := z_{T_{k+1}}$ as evaluations of this process. Then, $(\tilde{y}_k,\tilde{x}_k, \tilde{z}_k)$ writes as a specific parameterization of \eqref{eq:nm}, that we call the \textit{continuized Nesterov algorithm} \eqref{alg:constant_param_det}, namely 
\begin{align}\label{alg:constant_param_det}\tag{CNA}\left\{
    \begin{array}{ll}
        \tilde{y}_k &=\tilde{x}_k +\frac{\eta}{\eta+\eta'}\bpar{1 -e_k } (\tilde{z}_{k} - \tilde{x}_k)  \\
        \tilde{x}_{k+1} &=  \tilde{y}_k  - \gamma \nabla f(\tilde{y}_k)\\
        \tilde{z}_{k+1} &=   \tilde{z}_{k} + \eta' \frac{1- e_k}{\eta' + \eta e_k}(\tilde{y}_k - \tilde{z}_{k}) - \gamma' \nabla f(\tilde{y}_k)
    \end{array}
\right.
\end{align}
where $e_k=e^{-(\eta+\eta')(T_{k+1}-T_k)}.$
\end{proposition}
In words, \eqref{alg:constant_param_det} is an exact evaluation of a realization of \eqref{eq:nest_continuized}. Intuitively, exact access to the continuous process \eqref{eq:nest_continuized} is possible because gradient information is only queried at discrete times.
This direct correspondence between \eqref{alg:constant_param_det} and \eqref{eq:nest_continuized} is powerful: it enables convergence guarantees for a discrete algorithm to be derived via continuous-time analysis. This comes at the cost that even with deterministic gradients, the resulting algorithm is inherently stochastic, thus leading to non-deterministic convergence guarantees.

\noindent
\textbf{Practical implementation of \eqref{alg:constant_param_det}\:}
From a practical point of view, compared to a deterministic parameterization of \eqref{eq:nm}, the randomization procedure in \eqref{alg:constant_param_det} does not introduce any significant additional difficulty. This is because, for any $k \in \mathbb{N}$, the increments $T_{k+1} - T_k$ are independent and identically distributed following an exponential law of parameter $1$, and can therefore be easily and independently simulated. As a consequence, to run \eqref{alg:constant_param_det} with $n \in \N^\ast$ iterations, we can first generate at once $n$ random variable $\Delta_k \overset{i.i.d}{\sim} \mathcal{E}(1)$, and save the vector $\{\Delta_k \}_{k\in \{1,\cdots, n\}}$. Then, for fixed parameters $\eta, \eta', \gamma$, and $\gamma'$, the algorithm \eqref{alg:constant_param_det} can be implemented in the same way as \eqref{eq:nm}, using the sequences
\[
\alpha_k = \frac{\eta}{\eta+\eta'}\left(1 - e^{-(\eta+\eta')\Delta_k}\right),
\qquad
\beta_k = \eta'\frac{1 - e^{-(\eta+\eta')\Delta_k}}{\eta' + \eta e^{-(\eta+\eta')\Delta_k}},
\quad k=1,\dots,n.
\] 

\paragraph{Link between \eqref{alg:constant_param_det} and \eqref{eq:nm_prime}}

Former results that achieved the $\bigO(\varepsilon^{-7/4})$ complexity used a two-sequence formulation \eqref{eq:nm_prime} as their momentum component. Using \citep[Proposition 11]{hermant2024study}, we can rewrite \eqref{alg:constant_param_det} as a two-sequence algorithm:

\[
\left\{
\begin{array}{ll}
\tilde y_k
=
\tilde x_k
+
\displaystyle
\frac{(1-e_k)e_{k-1}}{1-e_{k-1}}
(\tilde x_k-\tilde x_{k-1})
+
\displaystyle
\frac{\eta}{\eta+\eta'}(1-e_k)
\left(
\frac{\gamma'}{\gamma}
-
1
-
\frac{(\eta+\eta')e_{k-1}}{\eta(1-e_{k-1})}
\right)
(\tilde x_k-\tilde y_{k-1}),
\\[12pt]
\tilde x_{k+1}=\tilde y_k-\gamma \nabla f(\tilde y_k),
\end{array}
\right.
\]
with $e_k=e^{-(\eta+\eta')(T_{k+1}-T_k)}.$ Compared with \eqref{eq:nm_prime}, there is a supplementary term, factor of $\tilde x_k- \tilde y_{k-1} = -\gamma \nabla f(\tilde y_{k-1})$.
This supplementary term appears in algorithms that get an improved constant in the convergence bound in the smooth strongly convex setting, such as the Optimized Gradient Method (OGM) \citep{kim2016optimized}. It also appears in results that achieves acceleration in the non-convex regime of strongly quasar-convex functions \citep[Proposition 4]{hermant2024study}.
With our choice of parameter in Theorem~\ref{thm:hess_lip}, this supplementary factor does not reduce to zero. It may indicate that algorithms of the form \eqref{eq:nm_prime} are unable to achieve the $\bigO(\varepsilon^{-7/4})$ complexity by their own, even if randomizing the parameter.

\section{Convergence Result under Lipschitz Gradient Only}\label{sec:l_smooth}
We derive a convergence result under Assumption~\ref{ass:l_smooth} only.
\begin{restatable}{proposition}{cvLsmooth}\label{thm:cv_L_smooth}
Assume $T_1,\dots,T_k$ are random variables such that $T_{i+1} - T_i$ are i.i.d. of law $\mathcal{E}(1)$, with convention $T_0 = 0$. Under Assumption~\ref{ass:l_smooth}, iterations of (\ref{alg:constant_param_det}) with $\gamma \le \frac{1}{L}$, $\gamma' = \gamma + \sqrt{\frac{\gamma}{2}}$, $\eta = \sqrt{\frac{\gamma}{2}}$ and $\eta' > -\eta$ verify
\begin{equation*}
    \mathbb{E}\left[\min_{0\le i \le k-1}\norm{\nabla f(\tilde y_i)}^2 \right]
    \le \frac{4}{\gamma}\frac{\gapf}{k}.
\end{equation*}
\end{restatable}
Proposition~\ref{thm:cv_L_smooth} implies that we find, in expectation, an $\varepsilon$-stationary point in at most $O(\varepsilon^{-2})$ gradient evaluations. Perhaps surprisingly, the parameter $\eta'$ can be chosen freely in $(-\eta,+\infty)$, without affecting the result. This is a new result for the continuized algorithm \eqref{alg:constant_param_det}, which shows it recovers the known bound of the classical Nesterov momentum algorithms in this setting, see \cite{ghadimi2016accelerated}. 
Proposition~\ref{thm:cv_L_smooth} is a consequence of the following bound, established using Lyapunov analysis for the continuous-time system \eqref{eq:nest_continuized}.
\begin{restatable}{lemma}{lsmooth}\label{lem:l_smooth}
If $\xz$ is solution of (\ref{eq:nest_continuized}) under Assumption~\ref{ass:l_smooth} with $\gamma \le \frac{1}{L}$, $\gamma' = \gamma + \sqrt{\frac{\gamma}{2}}$, $\eta = \gamma' - \gamma$ and $\eta' > -\eta$, we have
\begin{align}\label{eq:l_smooth_use_hess}
    \mathbb{E}\left[ \int_{0}^{t} (\eta + \eta')\norm{x_s - z_s}^2 + \frac{\gamma}{4}\norm{\nabla f(x_{s})}^2 ds \right]
    \le \mathbb{E}\left[  f(x_0)- f^\ast \right].
\end{align}
\end{restatable}
\begin{sproof}

We set $\overline{x}_t= (x_t,z_t)$ as a concatenation of $\xcont$ and $(z_t)_{t\ge 0}$, which satisfies $d\overline{x}_t = \zeta(\overline{x}_t)dt + G(\overline{x}_t)dN(t)$, where
\begin{equation*}
    \zeta(\overline{x}_t) = \begin{pmatrix}
     \eta(z_t-x_t) \\ \eta'(x_t-z_t)
    \end{pmatrix},\quad G(\overline{x}_t) = \begin{pmatrix}
    -\gamma_t \nabla f(x_t) \\ -\gamma_t' \nabla f(x_t)
    \end{pmatrix}.
\end{equation*}
We set the Lyapunov function
\begin{equation*}
    \varphi(x,z) = f(x) + \frac{1}{2}\norm{x-z}^2.
\end{equation*}
Intuitively, one can stochastically derivate  $\varphi(\overline{x}_t)$ using an Itô formula, yielding the following relation
\[    \varphi(\overline{x}_t) = \varphi(\overline{x}_0) + \int_{0}^t \dotprod{\nabla \varphi(\overline{x}_s),\zeta(\overline{x}_s)}ds + \int_0^t \varphi(\overline{x}_s + G(\overline{x}_s)) - \varphi(\overline{x}_s)ds + M_t,\]
where $\mart$ is a martingale verifying $\E{M_t} = 0$ for all $t\ge 0$. Using Assumption~\ref{ass:l_smooth} and the parameter choice, we compute that
\[     \dotprod{\nabla \varphi(\overline{x}_s),\xi(\overline{x}_s)}  +    \varphi(\overline{x}_s+ G(\overline{x}_s)) - \varphi(\overline{x}_s)\le -(\eta + \eta')\norm{x_s - z_s}^2 - \frac{\gamma}{4}\norm{\nabla f(x_{s})}^2,\]
which yields
\begin{equation}\label{eq:l_smooth_lyap_dec}
    \int_{0}^{t} (\eta + \eta')\norm{x_s - z_s}^2 + \frac{\gamma}{4}\norm{\nabla f(x_{s})}^2 ds\le f(x_0)- f^\ast + M_t.
\end{equation}
One can conclude using $\E{M_t} = 0$ for all $t\ge 0$. See the complete proof in Appendix~\ref{app:l_smooth}.
\end{sproof}
From Lemma~\ref{lem:l_smooth}, it follows that $\E{\min_{s\in [0,t]}\norm{ \nabla f(x_{s})}^2} = \bigO(1/t)$. We note that a similar bound is obtained under Assumption~\ref{ass:l_smooth} for a continuized system in \cite[Proposition 6]{nabli2023a2cid2}, but with different dynamics from \eqref{eq:nest_continuized}, specifically adapted to asynchronous algorithms. 
With simple terms, because $\tilde{x}_{k}$ in \eqref{alg:constant_param_det} is exactly $x_{t}$ in \eqref{eq:nest_continuized} evaluated at the time $T_k$, Proposition~\ref{thm:cv_L_smooth} follows from evaluating \eqref{eq:l_smooth_lyap_dec} at $t=T_k$, from which will follow a convergence result for \eqref{alg:constant_param_det}, up to some technical considerations, see Appendix~\ref{app:l_smooth}.
Lemma~\ref{lem:l_smooth}, in particular \eqref{eq:l_smooth_lyap_dec}, is not only an intermediate result for Proposition~\ref{thm:cv_L_smooth}. It also serves as a key building block for proving our main theorem in the Lipschitz Hessian setting (Theorem~\ref{thm:hess_lip}).



\section{Main Results}
We state in Section~\ref{sec:cv_main} our acceleration result for \eqref{alg:constant_param_det}. In Section~\ref{sec:limit_ode}, we show that a time-rescaled version of the continuized system \eqref{eq:nest_continuized} converges in probability to the Heavy Ball ODE \eqref{eq:HB} as the stepsize $\gamma$ goes to zero. 

\subsection{Convergence Result}\label{sec:cv_main}
We show that \eqref{alg:constant_param_det} achieves the $\bigO(\varepsilon^{-7/4})$ complexity under Assumptions~\ref{ass:l_smooth}-\ref{ass:hess_lip}.
It may be seen as a discrete-time version of a convergence result that considers the Heavy Ball ODE \eqref{eq:HB} \cite{okamura2024primitive}.

         \begin{theorem}\label{thm:hess_lip}
            Let $\tk$ such that $T_0 = 0$, $T_{k+1} - T_k \overset{i.i.d}\sim \mathcal{E}(1)$.
Let $\tilde \varepsilon \in (0,1/2]$, and $\cupp \ge 1$ such that it satisfies $1-\tilde \varepsilon-\log(1-\tilde \varepsilon) = \cupp - \log(\cupp)$, which implies $\cupp \le 1+2\tilde \varepsilon$. Under Assumptions~\ref{ass:l_smooth}-\ref{ass:hess_lip}, consider the iterations of (\ref{alg:constant_param_det}) with $\gamma = \frac{1}{ L}$, $\gamma' = \frac{1}{ L} + \sqrt{\frac{1}{2L }}$, $\eta = \sqrt{\frac{1}{2L }}$, $\eta'= \alphaval-\sqrt{\frac{1}{2L }} $, where 
\[\alphaval = \left(\frac{\sqrt{3}\cdot 64\cupp e}{(1-\tilde \varepsilon)^2}\right)^{2/7}\bpar{\frac{L_2^2\gapf}{L^3n}}^{\frac{1}{7}}.\]
Then, if $n$ is large enough to ensure
$2((1+2\tilde \varepsilon)n)^{-1}< \alphaval \le 1$, we have
        \[\E{\1_{\Abar}\min_{ k \in \{1,\dots n \}}\norm{\nabla f(\overline{x}_k)}}\le \Bconst L^{\frac{2}{7}}L_2^{\frac{1}{7}} \gapf^{\frac{4}{7}} n^{-4/7},\]
        where $ \overline{x}_k = \sum_{i=0}^{k-1} \lambda_{i,k}\tilde{y}_i$ with $\lambda_{i,k} = \frac{ e^{\alphaval T_{i+1}}}{\sum_{j =1}^k  e^{\alphaval T_j}}$, $\mathbb{P}(\Abar) \ge  1-\tilde \varepsilon$, and $\Bconst$ is such that it is uniformly bounded with $n$ and $\lim_{n \to +\infty} \Bconst \le 9.2(1+2\tilde \varepsilon)^2\bpar{\frac{1+2\tilde \varepsilon}{(1-\tilde \varepsilon)^2}}^{1/7} $.
         \end{theorem}

A more general statement can be found in Theorem~\ref{thm:hess_lip_full}, in which we provide a result while relaxing the constraint of the value of $n$, to the cost of a less clean bound. From Theorem~\ref{thm:hess_lip}, we deduce that \eqref{alg:constant_param_det} outputs a point $\tilde x := \argmin_{ k \in \{1,\dots n \}}\norm{\nabla f(\overline{x}_k)}$ that verifies $\E{\1_{\Abar}\norm{\nabla f(\tilde x)}} \le \varepsilon$ in at most $\bigO\bpar{\frac{L^{1/2}L_2^{1/4}\Delta_f}{\varepsilon^{7/4}}}$ gradient evaluations, as long as $\varepsilon$ is small enough. We may remove the expectation and obtain a result with a given probability, the Markov property being a direct way to do so. 
The stochasticity induced by the process created significant challenges in the analysis, that we solved by deriving new concentration inequalities.
One of these establishes that for any $\tilde \varepsilon \in (0,1)$ and $n \in \N^\ast$, there exists $\Ccont$ such that with probability $1-\tilde \varepsilon$
\[\forall\,i \in \{1,\dots,n\}, \quad \sum_{j=1}^i e^{-\alphaval (T_i-T_j)}
\le \Ccont\,\E{\sum_{j=1}^i e^{-\alphaval (T_i-T_j)}},\]
where $\Ccont$ is uniformly upper-bounded with $n$ and $T_i-T_j \overset{i.i.d}{\sim} \Gamma(i-j,1)$. The set $\Abar$ arises from these inequalities, see details in Theorem~\ref{thm:high_proba_bound}. 
\paragraph{Numerical factor and asymptotic improvement}
The factor $\Bconst$, explicitly stated in \eqref{eq:b_n_value}, has a somewhat intricate finite-time expression. Choosing a $\tilde \varepsilon$ value close to zero increases the size of the set $\Abar$, on which our result holds, at the cost of a poorer constant $\Bconst$ in \textit{finite time}. However, asymptotically with $n$, the factors that deteriorate the bound when $\tilde \varepsilon$ is small vanish. So, in this asymptotic regime, we can choose $\tilde \varepsilon$ small enough to ensure for instance $\lim_{n \to +\infty} \Bconst \le 10$. Then for $\varepsilon$ small enough, from Theorem~\ref{thm:hess_lip} we deduce that we reach in expectation a $\varepsilon$-solution in at most $\frac{10^{7/4}L^{1/2}L_2^{1/4}\Delta_f}{\varepsilon^{7/4}} \approx \frac{56L^{1/2}L_2^{1/4}\Delta_f}{\varepsilon^{7/4}}$ gradient evaluations. The best existing bound \cite{li2023restarted} is $\frac{82^{7/4}L^{1/2}L_2^{1/4}\Delta_f}{\varepsilon^{7/4}}\approx \frac{2234L^{1/2}L_2^{1/4}\Delta_f}{\varepsilon^{7/4}}$, such that we considerably improve the numerical constant factor in the asymptotic regime.

\begin{remark}
     We note that as the analysis of \cite{okamura2024primitive} requires knowledge of the terminal time $T$ to tune parameters in the continuous setting, our result uses the final iteration $n$ to tune $\eta'$ in Theorem~\ref{thm:hess_lip}. There is a possibility that this dependence can be removed by using an iteration-dependent schedule of the form $\eta'_k = \bigO(k^{-1/7})-\sqrt{\gamma/2}$, at the cost of significant additional technical complexity. In a similar way, prior works achieving the $\bigO(\varepsilon^{-7/4})$ complexity involve $\varepsilon$-depending parameters  \cite{convexguilty,momentumsaddle,li2023restarted}, with the notable exception of the parameter-free restart method of \cite{Marumo2022ParameterFreeAG}, or the independent work \cite{ushiyama2026restart}.
\end{remark}

\subsection{Convergence of the continuized system to the Heavy Ball equation}\label{sec:limit_ode}
Compared with more classical momentum algorithms, the continuized method may seem somewhat atypical. In this section, we draw a link with a more classical object of the momentum literature: the Heavy Ball equation \eqref{eq:HB}.
In the $L$-smooth, convex/strongly convex case, \cite{even2021continuized} uses the heuristic $dN(\gamma^{-1/2}s) \approx \gamma^{-1/2}s$ in the limit $\gamma \to 0$ to argue that \eqref{eq:nest_continuized}, parameterized with values ensuring acceleration, converges to the Heavy Ball equation. We provide a rigorous statement in our setting. To do so, we want to consider \eqref{eq:nest_continuized} with $0<\gamma \le 1/L$, $\gamma' = \gamma+ \sqrt{\gamma/2}$, $\eta = \sqrt{\gamma/2}$, $\eta'= \alphavalcont-\sqrt{\gamma/2} $, where $\alphaval = \left(\frac{\sqrt{3}\cdot 64\cupp e}{(1-\tilde \varepsilon)^2}\right)^{2/7}\bpar{\frac{\gamma^3 L_2^2\gapf}{n}}^{\frac{1}{7}}$, and study the limit of this system as $\gamma$ goes to zero after a suitable time-rescaling. We recall that Theorem~\ref{thm:hess_lip} is obtained using these parameters, with the special choice $\gamma = \frac{1}{L}$. However, we will provide a modification on these parameters. Because $\alphaval$ is tuned to optimize discrete quantities, it involves $n \in \N$ that corresponds to a number of iterations. We did not derive a convergence result for \eqref{eq:nest_continuized} in continuous time, but it is reasonable to make the correspondence between the final iteration $n$ and a terminal time $T \in \R_{\ge 0}$, such that we replace $n$ by $T$ in the definition of $\alphaval$. 
     \begin{theorem}\label{thm:cv_hb}
     Let $\xz$ be solution of \eqref{eq:nest_continuized}, which we recall to be
 \begin{align}\label{eq:nest_continuized}\tag{CNE}\left\{
    \begin{array}{ll}
        dx_t &= \eta(z_t-x_t)dt - \gamma \nabla f(x_{t^-}) dN_t, \\
        dz_t &= \eta'(x_t-z_t)dt - \gamma' \nabla f(x_{t^-}) dN_t,
    \end{array}
\right.
\end{align}
where
$\gamma' = \gamma + (\gamma/2)^{\frac{1}{2}}$, $\eta =  (\gamma/2)^{\frac{1}{2}}$, $\eta'= \alphaval_{(\gamma/2)^{-1/2}T}- (\gamma/2)^{\frac{1}{2}} $, with
\[\alphavalcontf{(\gamma/2)^{-1/2}T} = \left(\frac{\sqrt{3}\cdot 64\cupp e}{(1- \tilde \varepsilon)^2}\right)^{2/7}\bpar{\frac{\gamma ^{3} L_2^2\gapf}{T(\gamma/2)^{-1/2}}}^{\frac{1}{7}}=K_{\tilde \varepsilon,\cupp}\bpar{\frac{L_2^2\gapf}{\sqrt{2}T}}^{\frac{1}{7}}\gamma^{1/2},\]
for some $T>0$, $\tilde \varepsilon \in (0,1/2]$ and $K_{\tilde \varepsilon,\cupp} =\left(\frac{\sqrt{3}\cdot 64\cupp e}{(1- \tilde \varepsilon)^2}\right)^{2/7} $.
Then, the time-rescaled process $\xrescaled := x_{ (\gamma/2)^{-1/2}s}$ converges uniformly in probability on $[0,T]$ to the Heavy Ball ODE as $\gamma \to 0$, namely 
\[\sup_{s \in [0,T]} \norm{X_s-\xrescaled} \overset{\mathbb{P}}{\to}_{\gamma \to 0} 0,\]
where
 \[\ddot X_s +  2^{3/7}K_{\tilde \varepsilon,\cupp}\bpar{\frac{L_2^2 \Delta_f}{T}}^{1/7}\dot X_s + \nabla f(X_s) = 0,\]
 as long as $x_0 = X_0$.
    \end{theorem}

   See the proof in Appendix~\ref{app:proof_hb_cv}. This result indicates that we may view \eqref{eq:nest_continuized} as a specific, stochastic discretization of \eqref{eq:HB}. It provides a strong theoretical connection with more classical, deterministic parameterization of Nesterov momentum, which converge to the same kind of ODE \cite{suboydcandes,shi2018understanding}.
Empirically, we observe that in the small stepsize regime, the behavior of \eqref{alg:constant_param_det} is indeed very close to a deterministic version. On Figure~\ref{fig:placeholder}, we run \eqref{alg:constant_param_det} on a matrix factorization problem, that satisfies Assumptions \ref{ass:l_smooth} and \ref{ass:hess_lip} on a subspace of the domain, see \cite[Lemma 6]{jin2017escape} for admissible values of $L$ and $L_2$. We also run a deterministic instance of \eqref{eq:nm} by approximating the parameters of \eqref{alg:constant_param_det}, designed as follows. We have $\frac{\eta}{\eta+\eta'}\bpar{1 - e^{-(\eta+\eta')\Delta_k}} \approx \eta\Delta_k$ if $\eta + \eta'$ is small, while $\eta'\frac{1 - e^{-(\eta+\eta')\Delta_k}}{\eta' + \eta e^{-(\eta+\eta')\Delta_k}} \approx \frac{\eta'(\eta+\eta)\Delta_k}{\eta' + \eta(1-(\eta+\eta'))\Delta_k}$. Replacing roughly $\Delta_k$ by its expectation $1$, we then chose $\alphaval_k = \eta$ and $\beta_k = \frac{\eta'(\eta+\eta)}{\eta' + \eta(1-(\eta+\eta'))}$ as our deterministic sequences of parameters for \eqref{eq:nm}. We observe on Figure~\ref{fig:placeholder} that in the small stepsize regime, the continuized version behaves very similarly to the deterministic version, with function values curves being superposed. Behaviors start differentiating when increasing stepsize and choosing a more aggressive momentum parameter.    
\begin{figure}
    \centering
    
    \includegraphics[width=0.4\linewidth]{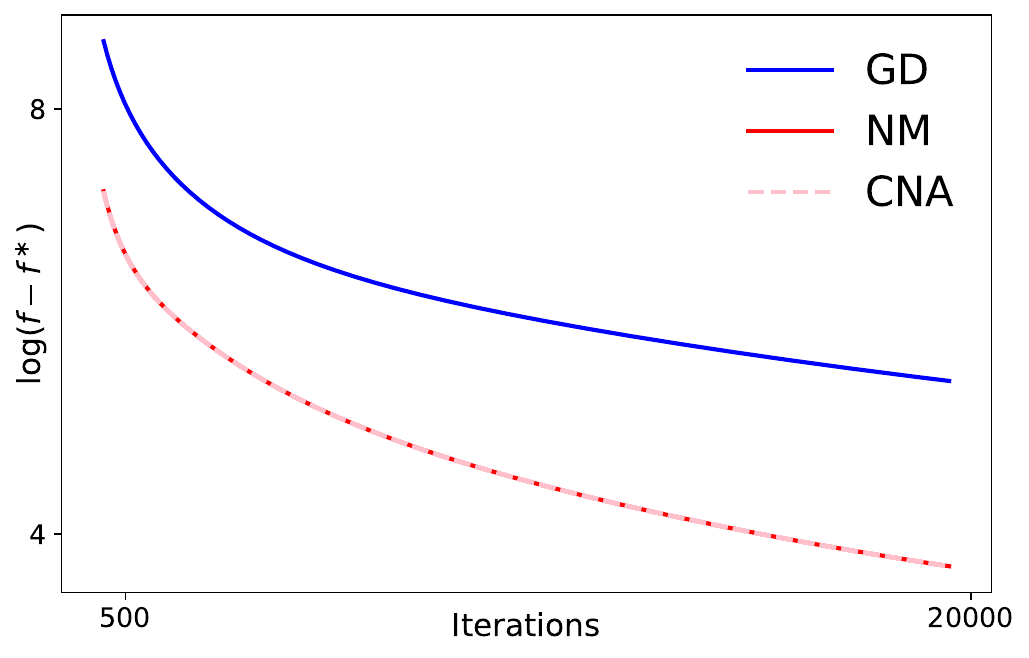}
    \includegraphics[width=0.405\linewidth]{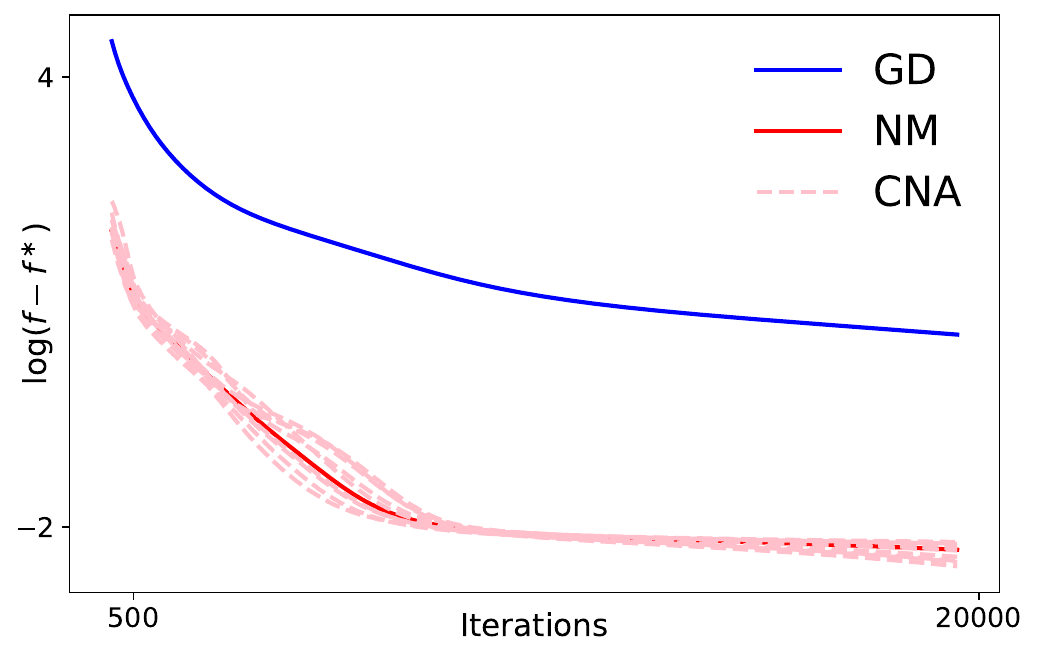}
    \caption{We plot one run of gradient descent, one run of a deterministic instance of \eqref{eq:nm}, and ten runs of \eqref{alg:constant_param_det} on a matrix factorization problem. On the left, we choose $\gamma$ according to the theoretical result of \cite[Lemma 6]{jin2017escape}. On the right, we chose a more aggressive stepsize and momentum parameter $\alphaval$, i.e. smaller. We note that because we pushed the parameters closer to the stability limit, one of the 10 runs of \eqref{alg:constant_param_det} diverged and is not plotted. }
    \label{fig:placeholder}
\end{figure}

\section{Overview of the proof of Theorem~\ref{thm:hess_lip}}\label{sec:proof}

     This section provides a proof outline of Theorem~\ref{thm:hess_lip}, with most technical details deferred to the appendix.
     We recall that $\xz$ are defined as follows
     \begin{align}\tag{CNE}\left\{
    \begin{array}{ll}
        dx_t &= \eta(z_t-x_t)dt - \gamma \nabla f(x_{t^-}) dN_t, \\
        dz_t &= \eta'(x_t-z_t)dt - \gamma' \nabla f(x_{t^-}) dN_t,
    \end{array}
\right.
\end{align}
with $dN_t =\sum_{k\ge 1} \delta_{T_k}(dt)$ for $\tk$ a sequence of random variables satisfying $T_{k+1}-T_k \overset{i.i.d}\sim \mathcal{E}(1)$, with convention $T_0=0$. Throughout the proof, we note $\alphaval := \eta + \eta'$ and use the shorthand notation $\int_0^t := \int_{[0,t]}$ to save space. Inspired from the averaging \cite[Equations (4)-(5)]{okamura2024primitive}, we consider a Poisson-average of the trajectory $x_t$
\begin{equation}\label{eq:poisson_average}
      \overline{x}_t := \int_{0}^t w_t(s)x_{s^-} dN_s,
\end{equation}
where we set $w_t(s) = \frac{\alphaval e^{\alphaval s}}{\int_0^t \alphaval e^{\alphaval s'}dN(s')}$. To readers unfamiliar with Poisson processes, we emphasize that Poisson integrals are stochastic sums. Namely, for a measurable function $\phi$,
\[\int_0^t \phi(s)dN_s = \sum_{i=1}^{N_t}\phi(T_i)\]
where $N_t = \sum_{k\ge 1} \mathds{1}_{T_k \le t} \sim \mathcal{P}(t)$. So, we can also write $\overline{x}_t= \sum_{i=1}^{N_t} w_t(T_i^-)x_{T_i^-}.$
In particular, at $t=T_n$, by definition $N_{T_n} = n$ and the number of terms in the sum is fixed, becoming $\sum_{i = 1}^{n} w_t(T_i^-)x_{T_i^-}$. 
To ensure $\int_0^t w_t(s)dN_s = 1$, we assume $N_t \ge 1$ (otherwise the integral is zero). This is not a restriction, because $N_t < 1$ would imply no jumps occurred, \textit{i.e.} the underlying algorithm has not started yet.

\paragraph{Plan of the Proof} Our proof strategy is the following, divided into seven steps.
\begin{enumerate}[label=\Roman*.]
    \item We establish some upper-bounds on trajectory-depending quantities, that follow from the Lyapunov control of Section~\ref{sec:l_smooth}, which uses the Lipschitz gradient property. These controls are used in step III and IV. 
    \item We upper-bound by three terms an expected weighed average $\E{\int_0^{T_n} \Lambda_t  \norm{\nabla f(\overline{x}_{t})}dN_t}$, for some weighting parameter $\Lambda_t$. It uses the Hessian Lipschitz property, and a stochastic integration by
part formula. 
    \item We bound the first of these three terms, by a quantity of the order $\bigO(\sqrt{\alphaval n})$. The main argument is the Cauchy-Schwartz inequality and our Lyapunov control from I.
    \item We bound with high probability the two remaining terms, by a quantity of the order $\bigO(\alphaval^{-3})$. This part is the most tedious of the proof. Apart from our Lyapunov control from I, it requires to use a concentration inequality on the underlying jump times $\tk$, thus a result holding on a subset of realization $\Abar$.
    \item We combine the bounds derived in III and IV, and show that with the choice $\alphaval = \bigO(n^{-1/7})$, we obtain a result of the form
    \[\E{\min_{t \in \{T_1,\dots,T_n \}} \norm{\nabla f(\overline{x}_t)} \int_0^{T_n}\bpar{\int_0^t \alphaval e^{\alphaval(s-t)}dN_s}^2 dN_t} = \bigO(n^{3/7}).\]
    \item We obtain a lower-bound of the form $\int_0^{T_n} \bpar{\int_0^t \alphaval e^{\alphaval(s-t)}dN_s}^2 dN_t \ge Cn$ for some constant $C<1$, that holds with high probability.
    \item We state our final convergence results. The final key argument it that although it involves the continuous-time process $\overline{x}_t$, $\arg\min_{t \in \{T_1,\dots,T_n \}}\norm{\nabla f(\overline{x}_t)} $,  can in fact be computed by \eqref{alg:constant_param_det}. 
\end{enumerate}

\paragraph{Difference and similarities in the analysis with Okamura et al. \cite{okamura2024primitive}} Some of our arguments are directly inspired from \cite{okamura2024primitive}. This is the case for our definition of averaging \eqref{eq:poisson_average}, that is essentially \cite[Equations (4)-(5)]{okamura2024primitive} with $ds$ replaced by $dN_s$. Our Part II is also directly inspired: Lemma~\ref{lem:hess_lip_2} is a continuized-adapted version of Lemma~1 in \cite{okamura2024primitive}, the difference in our analysis mainly coming from the Poisson process and the fact that \eqref{eq:nest_continuized} is a two-trajectory system, in contrast with the Heavy Ball equation \eqref{eq:HB}.  
Similarly, our Lemma~\ref{lem:transfer_grad_derivative} provides a stochastic integration by part, inspired by the integration by parts \cite[Equation 18]{okamura2024primitive}, but in our case the stochastic nature of the process requires some care when using such arguments. Beyond these directly-adaptable points, the rest of the analysis differs because the stochasticity induced by the process create non trivial difficulties. In particular, Theorem~\ref{thm:high_proba_bound} provides new concentrations inequalities that specifically address such difficulties, and is an essential argument of the proof.

\subsection{Part I - Lyapunov Controls from Section~\ref{sec:l_smooth}}\label{sec:proof_lyap_control}
In this first part, we establish bounds on trajectory-depending quantities that we can deduce from the analysis carried in Section~\ref{sec:l_smooth}.
Under Assumption~\ref{ass:l_smooth}, we proved in Lemma~\ref{lem:l_smooth} inequality \eqref{eq:l_smooth_lyap_dec}, namely
\[\int_{0}^{t} (\eta + \eta')\norm{x_s - z_s}^2 + \frac{\gamma}{4}\norm{\nabla f(x_{s})}^2 ds\le f(x_0)- f^\ast + M_t,\]
with $M_t$ a martingale such that $\E{M_{t}}=0$.
From this, we deduce
\begin{align}
      & \int_{0}^{t} (\eta + \eta')\norm{x_s - z_s}^2ds  \le  f(x_0)- f^\ast +M_t,\nonumber\\
       & \int_{0}^{t} \frac{\gamma}{4}\norm{\nabla f(x_{s^-})}^2 dN_s\le   f(x_0)- f^\ast +U_t,\label{eq:lyap_control_2_pre}\\
       &\int_{0}^{t} (\eta + \eta')\norm{x_{s^-} - z_{s^-}}^2dN_s\le  f(x_0)- f^\ast +V_t,\label{eq:lyap_control_3_pre}
\end{align}
where for some martingales $U_t$, $V_t$ with mean zero, \eqref{eq:lyap_control_2_pre} and \eqref{eq:lyap_control_3_pre} hold because $\int_{0}^{t} (\eta + \eta')\norm{x_s - z_s}^2ds  = \int_{0}^{t} (\eta + \eta')\norm{x_{s^-} - z_{s^-}}^2dN_s$ and $\int_{0}^{t} \frac{\gamma}{4}\norm{\nabla f(x_{s})}^2 ds = \int_{0}^{t} \frac{\gamma}{4}\norm{\nabla f(x_{s^-})}^2 dN_s$, up to martingales additive terms whose expectations are zero. Now, using Theorem~\ref{thm:martingal_stopping}, evaluating these bounds at $t=T_n$, they become
\begin{align}
      &\mathbb{E}\left[ \int_{0}^{T_n} (\eta + \eta')\norm{x_s - z_s}^2ds \right] \le \mathbb{E}\left[  f(x_0)- f^\ast \right], \label{eq:lyap_control_1}\\
       & \mathbb{E}\left[ \int_{0}^{T_n} \frac{\gamma}{4}\norm{\nabla f(x_{s^-})}^2 dN_s \right] \le \mathbb{E}\left[  f(x_0)- f^\ast \right],\label{eq:lyap_control_2}\\
       &\mathbb{E}\left[ \int_{0}^{T_n} (\eta + \eta')\norm{x_{s^-} - z_{s^-}}^2dN_s \right] \le \mathbb{E}\left[  f(x_0)- f^\ast \right].\label{eq:lyap_control_3}
\end{align}
The three inequalities \eqref{eq:lyap_control_1}, \eqref{eq:lyap_control_2}, \eqref{eq:lyap_control_3} will play a central role. A large part of the remainder of the proof is devoted to preparing the ground for the application of these inequalities.
\subsection{Part II - Using the Hessian Lipschitz Property to Obtain a Three Terms Upper-Bound}
We adapt the analysis of \eqref{eq:HB} carried by \citep{okamura2024primitive} to derive an analysis of \eqref{eq:nest_continuized}.
Ultimately, our goal is to obtain a control of the norm of the gradient of the average \eqref{eq:poisson_average}, namely of $\norm{\nabla f(\overline{x}_t)}$.
The Lipschitz Hessian property (Assumption~\ref{ass:hess_lip}) allows us to control the distance of $\nabla f(\overline{x}_t)$ to a quantity of interest, namely the average of the gradient.
\begin{lemma}\label{lem:hess_lip_2}
Let $\xz$ be solution of \eqref{eq:nest_continuized}. Under Assumption~\ref{ass:hess_lip}, we have
    \begin{align}
    &\norm{\int_0^t w_t(s)\nabla f(x_{s^-})dN_s - \nabla f(\overline{x}_t)}\nonumber \\
    &\le L_2\int_0^t \eta \norm{z_s-x_s}^2 \bpar{\int_s^t \int_0^s    w_t(\sigma)w_t(\tau)(\tau-\sigma)dN_\sigma dN_\tau }ds \label{eq:rewrite_0}\\
    &+L_2\int_0^t \gamma^2 \norm{\nabla f(x_{s^{-}})}^2 \bpar{\int_s^t \int_0^s   w_t(\sigma)w_t(\tau)(N_{\tau^-}-N_{\sigma^-}) dN_\sigma dN_\tau}dN_s. \label{eq:rewrite_1}
    \end{align}
\end{lemma}
See the proof in Appendix~\ref{app:lem_1}.
The quantities $\norm{z_s-x_s}^2$ and $\norm{\nabla f(x_{s^-})}^2$ appearing in \eqref{eq:rewrite_0} and \eqref{eq:rewrite_1} are interesting. Indeed, as showed in Section~\ref{sec:proof_lyap_control}, we can bound the expectation of integrals of these quantities.

In contrast, we do not have a direct control of $\int_0^t w_t(s)\nabla f(x_{s^-})dN_s$. However, because of the choice of averaging $w_s(\cdot)$ and the definition of \eqref{eq:nest_continuized}, a stochastic integration by parts formula reduces it to a more tractable quantity.
\begin{lemma}\label{lem:transfer_grad_derivative}
If $\xz$ is solution of \eqref{eq:nest_continuized}, for $t > 0$ we have
\begin{equation}
     \int_0^t w_t(s)\nabla f(x_{s^{-}})dN_s = -\frac{w_t(t)(z_t-x_t)}{\gamma'- \gamma}.
\end{equation}
\end{lemma}
See the proof in Appendix~\ref{app:by_part_int}.
Combining Lemmas~\ref{lem:hess_lip_2} and \ref{lem:transfer_grad_derivative} and using a triangular inequality yields the following upper-bound 
\begin{equation*}
    \begin{aligned}
    \norm{\nabla f(\overline{x}_t)}
    &\le \frac{w_t(t)\norm{z_t-x_t}}{\gamma'-\gamma}\\
    &+ L_2\int_0^t \eta \norm{z_s-x_s}^2 \bpar{\int_s^t \int_0^s    w_t(\sigma)w_t(\tau)(\tau-\sigma)dN_\sigma dN_\tau }ds \\
    &+L_2\int_0^t \gamma^2 \norm{\nabla f(x_{s^{-}})}^2 \bpar{\int_s^t \int_0^s   w_t(\sigma)w_t(\tau)(N_{\tau^-}-N_{\sigma^-}) dN_\sigma dN_\tau}dN_s. 
    \end{aligned}
\end{equation*}
This bound is nonetheless not sufficient to use our controls of Section~\ref{sec:proof_lyap_control}; in particular, we do not have a bound involving the non-integrated quantity $\norm{z_t-x_t}$. So, we integrate between $t=0$ and $t=T_n$ and take the expectation on the above inequality, after multiplication of each side by a suitable quantity.

\begin{lemma}\label{lem:hess_lip_3}
Noting $\Lambda_t := \bpar{\int_0^t \alphaval e^{\alphaval(s-t)}dN_s}^2 $, for $\xz$ solution of \eqref{eq:nest_continuized} under Assumption~\ref{ass:hess_lip}, we have
\begin{equation}
    \begin{aligned}\label{eq:hessLip_eq_1}
        &\E{\int_0^{T_n}\Lambda_{t} \norm{\nabla f(\overline{x}_{t})}dN_t}\le \frac{1}{\gamma'-\gamma}\underbrace{\E{\int_0^{T_n}\Lambda_{t}w_{t}({t})\norm{z_{t}-x_{t}}dN_t}}_{\text{Term 1}} \\
    &+\alphaval^2\eta^2L_2\underbrace{\E{\int_0^{T_n}  \norm{z_s-x_s}^2 \bpar{\int_s^{T_n}\int_s^t \int_0^s    e^{\alphaval(\tau-t)}e^{\alphaval(\sigma-t)}(\tau-\sigma) dN_\sigma dN_\tau dN_t}ds}}_{\text{Term 2}}\\
    &+\alphaval^2\gamma^2L_2 \underbrace{\E{\int_0^{T_n}  \norm{\nabla f(x_{s^-})}^2  \bpar{\int_s^{T_n}\int_s^t \int_0^s    e^{\alphaval(\tau-t)}e^{\alphaval(\sigma-t)}(N_{\tau^-}-N_{\sigma^-})dN_\sigma  dN_\tau dN_t}dN_s}}_{\text{Term 3}}.
    \end{aligned}
\end{equation}
\end{lemma}
See the proof in Appendix~\ref{app:lem_2}. Section~\ref{sec:part_3} is devoted to the control of Term 1. Section~\ref{sec:part_4} is devoted to the control of Term 2 and Term 3, these two terms sharing a similar analysis.

\subsection{Part III: Control of Term 1}\label{sec:part_3}

We obtain the following control of Term 1 in \eqref{eq:hessLip_eq_1}.
\begin{lemma}\label{lem:term_1_control}
     Under Assumptions~\ref{ass:l_smooth} and \ref{ass:hess_lip}, let $\xz$ be solution of \eqref{eq:nest_continuized} with $\gamma \le \frac{1}{L}$, $\gamma' = \gamma + \sqrt{\frac{\gamma}{2 }}$, $\eta = \sqrt{\frac{\gamma}{2 }}$ and $\eta' = \alphaval - \eta$ for $\alphaval > 0$. Then, we have
     \[\mathbb{E}\left[\int_0^{T_n}\bpar{\int_0^t \alphaval e^{\alphaval(s-t)}dN_s}^2\norm{\frac{w_t(t)(z_{t} - x_{t})}{\gamma' - \gamma}}dN_t \right] \le \sqrt{\alphaval n}\sqrt{\frac{A_n}{\gamma}}\sqrt{\gapf},\]
 with     $A_n = 12\bpar{1+\frac{3}{2}\alphaval}(1+2\alphaval)$.
\end{lemma}
\begin{sproof}

   The main idea is to use Cauchy-Schwartz inequality and our controls from Section~\ref{sec:proof_lyap_control}. 
Indeed, using the notation $\Lambda_t := \bpar{\int_0^t \alphaval e^{\alphaval(s-t)}dN_s}^2 $ and the definition $w_t(t) = \frac{ \alphaval e^{\alphaval t}}{\int_0^t \alphaval e^{\alphaval s}dN_s}$, Term 1 becomes
\begin{equation}
    \begin{aligned}\label{eq:hess_bound_1_bis}
              \mathbb{E}\left[\int_0^{T_n}\Lambda_{t}\norm{\frac{w_t(t)(z_{t} - x_{t})}{\gamma' - \gamma}}dN_t \right]&=\mathbb{E}\left[\int_0^{T_n} \frac{\alphaval}{\gamma' - \gamma}\bpar{\int_0^{t} \alphaval e^{\alphaval(s-t)}dN_s}\norm{z_{t} - x_{t}}dN_t\right]\\
    &\le \frac{\alphaval}{\gamma' - \gamma}\sqrt{\mathbb{E}\left[ \int_0^{T_n} \Lambda_{t} dN_t\right]}\sqrt{\mathbb{E}\left[\int_0^{T_n}\norm{z_{t} - x_{t}}^2dN_t \right]},
    \end{aligned}
\end{equation}
where we used the Cauchy-Schwarz inequality. $\sqrt{\mathbb{E}\left[ \int_0^{T_n} \Lambda_{t} dN_t\right]}$ can be computed explicitly. For the other term, we still cannot use our controls from Section~\ref{sec:proof_lyap_control}. Before, we use that
\begin{equation*}
    \begin{aligned}
        \int_0^{T_n}\norm{z_{t} - x_{t}}^2dN_t  &= \int_0^{T_n}\norm{z_{t^-} - x_{t^-} - (\gamma'-\gamma)\nabla f(x_{t^-})}^2dN_t \\
        &\le 2\int_0^{T_n}\norm{z_{t^-} - x_{t^-}}^2dN_t + 2(\gamma'-\gamma)^2\int_0^{T_n}\norm{\nabla f(x_{t^-})}^2dN_t  
    \end{aligned}
\end{equation*}
We note here that this step is necessary because our control \eqref{eq:lyap_control_3_pre} holds when integrating $\norm{z_{t^-} - x_{t^-}}^2$, not $\norm{z_{t} - x_{t}}^2$. The reason is that the relation $\int_0^t \phi(s)dN_s = \int_0^t \phi(s)ds + M_t$ for some centered martingale $\mart$ holds as long as $\phi$ is a \textit{previsible} process with respect to $N$, namely if $\phi(t)$ is measurable with respect to the filtration $\{N_s\}_{0\le s <t}$. While $\norm{z_{t^-} - x_{t^-}}^2$ is previsible, $\norm{z_{t} - x_{t}}^2$ is not.\end{sproof}
See Appendix~\ref{app:term_1} for the full proof.
\subsection{Part IV: Control of Term 2 and Term 3}\label{sec:part_4}
Compared with Term 1, the remaining Terms 2 and 3 from \eqref{eq:hessLip_eq_1} are more challenging. When using the classic deterministic \eqref{eq:HB} equation, the analogue of the triple integrals inside the parentheses are deterministic, and can be directly computed \cite{okamura2024primitive}.
In our case, if the expectation of these triple integrals can be computed, we have to deal with the expectation of the full expression. To do so, we are faced with the lack of independence between the trajectory-dependent factors and these triple integrals. 

To illustrate this, we consider Term 3, namely
\[\int_0^{T_n}  \norm{\nabla f(x_{s^-})}^2  \bpar{\int_s^{T_n}\int_s^t \int_0^s    e^{\alphaval(\tau-t)}e^{\alphaval(\sigma-t)}(N_{\tau^-}-N_{\sigma^-})dN_\sigma  dN_\tau dN_t}dN_s.\]
It can be rearranged as
\[\int_0^{T_n}  \bpar{\int_s^{T_n}\int_s^t  e^{\alphaval(\tau-t)}e^{\alphaval(s-t)} \norm{\nabla f(x_{s^-})}^2 \int_0^s  e^{\alphaval(\sigma-s)}(N_{\tau^-}-N_{\sigma^-}) dN_\sigma dN_\tau dN_t}dN_s.\]
Writing it as a sum, it becomes

\begin{align}\label{eq:exp_separate_2:1:illustrate}
\sum_{i=1}^n \sum_{l=i}^n \sum_{k=i}^l e^{\alphaval(T_k-T_l)}e^{\alphaval(T_i-T_l)} \norm{\nabla f(x_{T_i^-})}^2\sum_{j=1}^i e^{\alphaval (T_j-T_i)}(k-j).
\end{align}
Taking the expectation, we have
\begin{equation*}
    \begin{aligned}
        \E{\eqref{eq:exp_separate_2:1:illustrate}} &= \E{\sum_{i=1}^n \sum_{l=i}^n \sum_{k=i}^l e^{\alphaval(T_k-T_l)}e^{\alphaval(T_i-T_l)} \norm{\nabla f(x_{T_i^-})}^2\sum_{j=1}^i e^{\alphaval (T_j-T_i)}(k-j)}\\
        &= \sum_{i=1}^n \sum_{l=i}^n \sum_{k=i}^l \E{e^{\alphaval(T_k-T_l)}e^{\alphaval(T_i-T_l)} \norm{\nabla f(x_{T_i^-})}^2\sum_{j=1}^i e^{\alphaval (T_j-T_i)}(k-j)}
    \end{aligned}
\end{equation*}
Importantly, the increments of $\{ T_k\}_{k\in \N}$ are independent, namely for any integers $1 \le \ell_0< \ell_1$, we have $T_{\ell_1} - T_{\ell_0} \ind \phi(T_1,\cdots, T_{\ell_0})$, for any measurable function $\phi$. In our case, as $l,k \ge i$, we have $T_k-T_l \ind \norm{\nabla f(x_{T_i^-})}$ and $T_i-T_l \ind  \norm{\nabla f(x_{T_i^-})}$. So, we deduce
\[\E{\eqref{eq:exp_separate_2:1:illustrate}} = \sum_{i=1}^n \sum_{l=i}^n \sum_{k=i}^l \E{e^{\alphaval(T_k-T_l)}e^{\alphaval(T_i-T_l)}}\E{\norm{\nabla f(x_{T_i^-})}^2\sum_{j=1}^i e^{\alphaval (T_j-T_i)}(k-j)}\]
The challenge we face here is that $\nabla f(x_{T_i^-})$ and $\sum_{j=1}^i e^{\alphaval (T_j-T_i)}(k-j)$ are not independent. We address this difficulty in the next section.

\subsubsection{Splitting the expectations}
In the previous section, we showed that we arrive at a point where trajectory-dependent factors are multiplied with some trajectory-independent terms, from which they are no independent, which prevents to separate the expectations. Precisely, the problematic trajectory-independent factors we obtain are
$$\sum_{j=1}^i (T_i-T_j)e^{-\alphaval (T_i-T_j)}, \;  \sum_{j=1}^i e^{-\alphaval (T_i-T_j)}, \; \sum_{j=1}^i (i-j)e^{-\alphaval (T_i-T_j)},\quad i \in \{1,\dots,n\}.$$
Our solution is to show that with high probability, these terms can be upper-bounded by their expectation multiplied by a constant factor. The trickiest part is to find a probability and a constant that do not become trivial as $n$ grows. In particular, it cannot work by directly using the Markov property.



\begin{restatable}{theorem}{maintheorem}\label{thm:high_proba_bound}
Let $n \in \N^\ast$, $\cinf, \varepsilon \in (0,1)$, $\alphaval \in (0,1)$. Let $\Kval  = \ceil{(\cinf - 1 - \log(\cinf))^{-1}
\log\bpar{2\frac{n}{\varepsilon}}}$, and $\cupp \ge 1$ be chosen such that
$\cupp -\log(\cupp) = \cinf - \log(\cinf)$.
Let 
\[\Ccont = 32(4+4C_n+C_n^2)\frac{\cupp e}{\cinf^2}, \text{ where }~ C_n =  \cinf \alphaval \Kval .\] 
There exists a set $\Abar \subset \Omega$ of realizations satisfying $\mathbb{P}(\Abar) \ge 1-\varepsilon$ such that the three following assertions hold
\[
\begin{array}{r l l l}
(1) & \forall\,i \in \{\Kval +1,\dots,n\},
&\sum_{j=1}^i (T_i-T_j)e^{-\alphaval (T_i-T_j)}
&\le \Ccont\,\E{\sum_{j=1}^i (T_i-T_j)e^{-\alphaval (T_i-T_j)}}, \\[6pt]
(2) & \forall\,i \in \{1,\dots,n\},
&\sum_{j=1}^i e^{-\alphaval (T_i-T_j)}
&\le \Ccont\,\E{\sum_{j=1}^i e^{-\alphaval (T_i-T_j)}}, \\[6pt]
(3)& \forall\,i \in \{1,\dots,n\},
&\sum_{j=1}^i (i-j)e^{-\alphaval (T_i-T_j)}
&\le \Ccont\,\E{\sum_{j=1}^i (i-j)e^{-\alphaval (T_i-T_j)}}.
\end{array}
\]
\end{restatable}

Section~\ref{app:proof_high_proba} is devoted to the complete proof of Theorem~\ref{thm:high_proba_bound}. 
One of its key step is to use concentration inequalities to obtain that with a high probability that does not depend on $n$, we ensure both an upper and lower-bound on the increments $T_i-T_j$, of the form
    \[ \cinf \E{T_i-T_j } - \cinf \Kval  \le  T_i-T_j \le \cupp \E{T_i-T_j }+\cupp \Kval,\]
where $\Kval = \bigO(\log(n/\varepsilon))$. Precisely, the set $\Abar$ is the set on which the above holds. These bounds then allow to relate the sums with their expectations.
    
Thanks to Theorem~\ref{thm:high_proba_bound}, with high probability 
one can fully separate the expectation of the factor that depend on the trajectory from those who do not. This is formalized in the following result.

\begin{lemma}\label{lem:hess_lip_4}
We assume the same setting as Theorem~\ref{thm:high_proba_bound}.


\begin{enumerate}
    \item Noting $A_k := \int_{[T_k,T_n]}\int_{[T_k, t]}  \int_{[0,T_k]}e^{\alphaval(T_k-t)}e^{\alphaval(\tau-t)}e^{\alphaval(\sigma - T_k)}(\tau - \sigma)dN_\sigma dN_\tau dN_t $, we have
\begin{equation}
\begin{aligned}\label{eq:exp_separate_1}
    &\E{\1_{\Abar} \int_0^{T_n} \eta^2 \norm{z_s-x_s}^2 \bpar{\int_s^{T_n}\int_s^t \int_0^s    e^{\alphaval(\tau-t)}e^{\alphaval(\sigma-t)}(\tau-\sigma) dN_\sigma dN_\tau dN_t}ds}\\
   & \le \Ccont \max_{k\in \{1,\dots,n\}}\E{A_k}  \E{\int_0^{T_n} \eta^2 \norm{z_s-x_s}^2 ds}
\end{aligned}
\end{equation} 
\item Noting $B_k := \int_{T_k}^{T_n}\int_{T_k}^t     \int_0^{T_k}e^{\alphaval(\tau-t)}e^{\alphaval(T_k-t)}e^{\alphaval(\sigma - T_k)}(N_{\tau^-}-N_{\sigma^-})dN_\sigma  dN_\tau dN_t$, we have
\begin{equation}
    \begin{aligned}\label{eq:exp_separate_2}
        &\E{\1_{\Abar} \int_0^{T_n} \gamma^2 \norm{\nabla f(x_{s^-})}^2  \bpar{\int_s^{T_n}\int_s^t \int_0^s    e^{\alphaval(\tau-t)}e^{\alphaval(\sigma-t)}(N_{\tau^-}-N_{\sigma^-})dN_\sigma  dN_\tau dN_t}dN_s}\\
        &\le  \Ccont\max_{k\in \{1,\dots,n\}}\E{B_k}\E{\int_0^{T_n} \gamma^2 \norm{\nabla f(x_{s^-})}^2  dN_s}
    \end{aligned}
\end{equation}
\end{enumerate}
\end{lemma}
The proof of Lemma~\ref{lem:hess_lip_4} uses the ideas introduced in the beginning of Section~\ref{sec:part_4}, such as writing the integrals as sums (see \eqref{eq:exp_separate_2:1:illustrate}). See Appendix~\ref{app:lem_3} for the complete proof.
\subsubsection{Final Control}
Thanks to Lemma~\ref{lem:hess_lip_4}, we see that we can control separately the expectations of trajectory-dependent factors, and the factors $\max_{k\in \{1,\dots,n\}}\E{A_k}$ and $\max_{k\in \{1,\dots,n\}}\E{B_k}$. The first can be dealt with using the Lyapunov controls of Section~\ref{sec:proof_lyap_control}. For the two other terms, we note that
we can write the quantities $A_k$ and $B_k$ as sums, for instance in the case of $B_k$
\begin{align*}
\int_s^{T_n}\int_s^t \int_0^s    e^{\alphaval(\tau-t)}e^{\alphaval(\sigma-t)}(N_{\tau^-}-N_{\sigma^-})dN_\sigma  dN_\tau dN_t=\sum_{i=k}^n \sum_{j=1}^k\sum_{l=k}^i e^{2\alphaval(T_l-T_i)}e^{\alphaval(T_j-T_l)}(l-j).
\end{align*}
The expectation of the terms in this sum can be computed using basic properties of Gamma laws. Indeed, as $T_l-T_i \sim \Gamma(l-i,1)$, we recognize that $\E{e^{2\alphaval(T_l-T_i)}}$ is the characteristic function of a gamma law evaluated at $2 \alphaval$, whose value is $(1+2\alphaval)^{-(l-i)}$. The terms appearing in $\E{A_k}$ are not significantly more difficult. Then, $\max_{k\in \{2,\dots,n\}}\E{A_k}$ and $\max_{k\in \{2,\dots,n\}}\E{B_k}$ can be upper-bounded with basic algebraic manipulations, yielding the following bounds.

\begin{lemma}\label{lem:hess_lip_5}
    Consider $A_k$ and $B_k$ as defined in Lemma~\ref{lem:hess_lip_4}. Then
    \begin{enumerate} 
        \item $\max_{k\in \{2,\dots,n\}}\E{A_k} \le \frac{(1+\alphaval)(1+2\alphaval)}{\alphaval^4}.$
        \item $\max_{k\in \{1,\dots,n\}}\E{B_k}  \le \frac{(1+\alphaval)^2(1+2\alphaval)}{\alphaval^4}.$
    \end{enumerate}
\end{lemma}
See the proof in Appendix~\ref{app:lem_5}. Combining Lemma~\ref{lem:hess_lip_5} with our controls from Section~\ref{sec:proof_lyap_control}, we finally obtain the following bound on Terms 2 and 3.
\begin{lemma}\label{lem:term_2_3_control}
Assume the same setting as in Theorem~\ref{thm:high_proba_bound}.
      Under Assumptions~\ref{ass:l_smooth} and \ref{ass:hess_lip}, let $\xz$ be solution of \eqref{eq:nest_continuized} with $\gamma \le \frac{1}{L}$, $\gamma' = \gamma + \sqrt{\frac{\gamma}{2 }}$, $\eta = \sqrt{\frac{\gamma}{2 }}$ and $\eta' = \alphaval - \eta$ for $\alphaval \in (0,1)$. Then, we have
      \begin{align*}
           &\alphaval^2\eta^2L_2\E{\1_{\Abar}\int_0^{T_n}  \norm{z_s-x_s}^2 \bpar{\int_s^{T_n}\int_s^t \int_0^s    e^{\alphaval(\tau-t)}e^{\alphaval(\sigma-t)}(\tau-\sigma) dN_\sigma dN_\tau dN_t}ds}\\
    &+\alphaval^2\gamma^2L_2 \E{\1_{\Abar}\int_0^{T_n}  \norm{\nabla f(x_{s^-})}^2  \bpar{\int_s^{T_n}\int_s^t \int_0^s    e^{\alphaval(\tau-t)}e^{\alphaval(\sigma-t)}(N_{\tau^-}-N_{\sigma^-})dN_\sigma  dN_\tau dN_t}dN_s}\\
    &\le\frac{\Ccont L_2 \gamma}{\alphaval^3}B_n\gapf,
      \end{align*}

with $B_n := \frac{1}{2}(1+\alphaval)(1+2\alphaval)\bpar{1 + \alphaval8(1+\alphaval)}$.      
\end{lemma}
See the proof in Appendix~\ref{app:term_2_3}.
\subsection{Part V: Tuning $\alphaval$}\label{sec:part_5}


Combining Lemma~\ref{lem:hess_lip_3}, and the controls provided by Lemma~\ref{lem:term_1_control} (Term 1) and Lemma~\ref{lem:term_2_3_control} (Terms 2 and 3), we obtain that under the statement of Theorem~\ref{thm:hess_lip}, we have
       \begin{equation}
           \begin{aligned}\label{eq:almost_final}
                   \E{\1_{\Abar}\int_0^{T_n}\bpar{\int_0^t \alphaval e^{\alphaval(s-t)}dN_s}^2 \norm{\nabla f(\overline{x}_t)}dN_t} &\le \sqrt{\alphaval n}\sqrt{\frac{A_n}{\gamma}}\sqrt{\gapf} \\
    &+\frac{\Ccont L_2 \gamma}{\alphaval^3}B_n\gapf
           \end{aligned}
       \end{equation}
with, $A_n := 12\bpar{1+\frac{3}{2}\alphaval}(1+2\alphaval)$ and $B_n := \frac{1}{2}(1+\alphaval)(1+2\alphaval)\bpar{1 + \alphaval8(1+\alphaval)}$.

   Because we integrate with respect to $dN_t$ on $[0,T_n]$, we have
   \[\min_{t \in \{T_1,\dots,T_n \}} \norm{\nabla f(\overline{x}_t)} \int_0^{T_n}\bpar{\int_0^t \alphaval e^{\alphaval(s-t)}dN_s}^2 dN_t\le \int_0^{T_n}\bpar{\int_0^t \alphaval e^{\alphaval(s-t)}dN_s}^2 \norm{\nabla f(\overline{x}_t)}dN_t.\]
     Therefore, from \eqref{eq:almost_final}, it follows that
           \begin{equation}
           \begin{aligned}\label{eq:almost_final_bis}
            & \E{\1_{\Abar}\min_{t \in \{T_1,\dots,T_n \}} \norm{\nabla f(\overline{x}_t)} \int_0^{T_n}\bpar{\int_0^t \alphaval e^{\alphaval(s-t)}dN_s}^2 dN_t}\\
                  &\le \E{\1_{\Abar}\int_0^{T_n}\bpar{\int_0^t \alphaval e^{\alphaval(s-t)}dN_s}^2 \norm{\nabla f(\overline{x}_t)}dN_t} \\
                  &\overset{\eqref{eq:almost_final}}{\le} \sqrt{\alphaval n}\sqrt{\frac{A_n}{\gamma}}\sqrt{\gapf}+\frac{\Ccont L_2 \gamma}{\alphaval^3}B_n\gapf \\
                  &=n^{3/7}\bpar{\sqrt{A_n\constalpha L^{1+a_1}L_2^{a_2}\gapf^{1+a_3}} + B_n\constalpha^{-3} \Ccont L^{-1-3a_1}L_2^{1-3a_2} \gapf^{1-3a_3}},
    \end{aligned}
       \end{equation}
 where in the last inequality, we chose $\gamma= 1/L$ and defined 
\[\alphaval = \constalpha n^{-1/7}L^{a_1}L_2^{a_2}\Delta_f^{a_3},\]
where $\constalpha  >0$, $a_1,a_2,a_3 \in \R$ are defined such that $\alphaval \le 1$. The choice $\alphaval = \bigO(n^{-1/7})$ is derived from a trade-off between $\sqrt{\alphaval}$ and $\alphaval^{-3}$. It is consistent with the choice made in \cite{okamura2024primitive}, in the case of studying the Heavy Ball \eqref{eq:HB} ODE.

We exhibit two admissible choices of $\alphaval$, subject to the constrain $\alphaval \le 1$. We note that in general, this forces to set a condition on the value of the number of iteration $n$, which may depend on $L$, $L_2$ or $\Delta_f$.
If one wants to avoid such conditions, one can choose $C_\alpha = 1$, $a_1=a_2=a_3 = 0$.

\begin{corollary}[Choice I]\label{cor:choice1}
   Let $\alphaval = n^{-1/7} \le 1$. In this case, we have
       \begin{equation}
           \begin{aligned}\label{eq:corr_1}
            & \E{\1_{\Abar}\min_{t \in \{T_1,\dots,T_n \}} \norm{\nabla f(\overline{x}_t)} \int_0^{T_n}\bpar{\int_0^t \alphaval e^{\alphaval(s-t)}dN_s}^2 dN_t}\\
                  &\le16(1+2n^{-1/7})^4\bpar{1 + 8(1+2n^{-1/7})} (4+4C_n+C_n^2)\frac{\cupp e}{\cinf^2}\bpar{\sqrt{ L\gapf} + L^{-1}L_2 \gapf} n^{3/7},
    \end{aligned}
       \end{equation}
with $C_n =  \cinf n^{-1/7} \ceil{(\cinf - 1 - \log(\cinf))^{-1}
\log\bpar{2\frac{n}{\varepsilon}}}.$
\end{corollary}

    Otherwise, if we set a condition of $n$, we can obtain a cleaner result.
    \begin{corollary}[Choice II]\label{cor:choice2}
    Let \[\alphaval =  \left(\frac{\sqrt{3}\cdot 64\cupp e}{\cinf^2}\right)^{2/7}\bpar{\frac{L_2^2\gapf}{Ln}}^{\frac{1}{7}}.\] If $\alphaval \le 1$, we have
     \begin{equation*}
           \begin{aligned}
            & \E{\1_{\Abar}\min_{t \in \{T_1,\dots,T_n \}} \norm{\nabla f(\overline{x}_t)} \int_0^{T_n}\bpar{\int_0^t \alphaval e^{\alphaval(s-t)}dN_s}^2 dN_t}\\
                  &\le (1 + C_n + C_n^2/4)\Hconst 7\,6^{-6/7}\,12^{3/7}
64^{1/7}\bpar{\frac{\cupp e}{\cinf^2}}^{1/7}\bpar{\frac{\cupp e}{\cinf^2}}^{1/7}L^{\frac{2}{7}}L_2^{\frac{1}{7}} \gapf^{\frac{4}{7}} n^{3/7},
    \end{aligned}
    \end{equation*}
    with $C_n =  \cinf \alphaval \ceil{(\cinf - 1 - \log(\cinf))^{-1}
\log\bpar{2\frac{n}{\varepsilon}}}$ and $\Hconst =(1+\alphaval)(1+2\alphaval)(1+\alphaval 8 (1+\alphaval))$.
\end{corollary}
The proofs of Corollary~\ref{cor:choice1} and \ref{cor:choice2} are in Appendix~\ref{app:part_5}.

    \subsection{Part VI: Control of the double Poisson integral factor}
    Up to now, we managed to get an upper-bound of the form
    \begin{equation}\label{eq:sketch:high_prob_0}
         \E{\1_{\Abar}\min_{t \in \{T_1,\dots,T_n \}} \norm{\nabla f(\overline{x}_t)} \int_0^{T_n}\bpar{\int_0^t \alphaval e^{\alphaval(s-t)}dN_s}^2 dN_t} = \bigO(n^{3/7}),
    \end{equation}
    where we recall that by Theorem~\ref{thm:high_proba_bound}, $\Abar$ is such that $\mathbb{P}(\Abar) \ge 1-\varepsilon$.
    As we chose $\alphaval = \bigO(n^{-1/7})$, we have (see Lemma~\ref{prop:calc_sto_comput_1}-$(ii)$) 
    \[ \E{\int_0^{T_n}\bpar{\int_0^t \alphaval e^{\alphaval(s-t)}dN_s}^2 dN_t}=\bigO(n),\]
    which heuristically suggests that $ \E{\1_{\Abar}\min_{t \in \{T_1,\dots,T_n \}} \norm{\nabla f(\overline{x}_t)}} = \bigO(n^{-4/7}).$ However, the two factors inside the expectations in \eqref{eq:sketch:high_prob_0} are not independent, as they both depend on the $n$ first jump times. We therefore cannot separate the expectation. To circumvent this difficulty, we use the following result.
    \begin{lemma}\label{lem:lower_bound_delta}
Let  $n\in \N^\ast$, $\cinf  , \varepsilon \in (0,1)$, $\alphaval \in (0,1]$ and   $\Kval  = \ceil{(\cinf -1-\log(\cinf ))^{-1}\log\bpar{2\frac{n}{\varepsilon}}}$. Let $\cupp \ge 1$ be chosen such that $\cinf -\log(\cinf ) = \cupp-\log(\cupp)$ and $C_n =  \cinf \alphaval \Kval$. 
Then, with the same set $\Abar$ as in Theorem~\ref{thm:high_proba_bound}, we have
    \[\1_{\Abar}\int_0^{T_n}\bpar{\int_0^t \alphaval e^{\alphaval(s-t)}dN_s}^2dN_t \ge\1_{\Abar}\frac{n}{\cupp^2}e^{-2\frac{\cupp}{\cinf}C_n}
\left(1-\frac{2}{\cupp \alphaval n}\right).\]
\end{lemma}
The proof of Lemma~\ref{lem:lower_bound_delta} borrows arguments used to prove Theorem~\ref{thm:high_proba_bound}, see full proof in Appendix~\ref{app:delta}.
\subsection{Part VII: Final Results}
We now state our final result. There is a result with our first choice $\alphaval = n^{-1/7}$, that (almost) does not require any restriction on $n$, at the cost of a less clean bound. A second result holds with a more careful choice of $\alphaval$, but it induces a contraint on $n$ that has to be large enough.
\begin{theorem}[Complete Version]\label{thm:hess_lip_full}
Let $\tk$ such that $T_0 = 0$, $T_{k+1} - T_k \overset{i.i.d}\sim \mathcal{E}(1)$.
Let $\cinf  , \varepsilon \in (0,1)$, and $\cupp \ge 1$ such that it satisfies $\cinf-\log(\cinf) = \cupp - \log(\cupp)$. 

Under Assumptions~\ref{ass:l_smooth}-\ref{ass:hess_lip}, consider the iterations of (\ref{alg:constant_param_det}) with $\gamma = \frac{1}{ L}$, $\gamma' = \frac{1}{ L} + \sqrt{\frac{1}{2L }}$, $\eta = \sqrt{\frac{1}{2L }}$, $\eta'= \alphaval-\sqrt{\frac{1}{2L }} $.
\begin{enumerate}[label=(\roman*)]
    \item Let $\alphaval = n^{-1/7}$. If $n\ge 3$, we have
    \[ \E{\1_{\Abar}\min_{ k \in \{1,\dots n \}}\norm{\nabla f(\overline{x}_k)}} \le \Aconst\bpar{\sqrt{ L\gapf} + L^{-1}L_2 \gapf}n^{-4/7},\]
    where 
    \[\Aconst := 16(1+2n^{-1/7})^4\bpar{1 + 8(1+2n^{-1/7})} (4+4C_n+C_n^2)\bpar{1-\frac{2}{\cupp n^{6/7}}}^{-1}\frac{\cupp^3}{\cinf^2} e^{1+2\frac{\cupp}{\cinf}C_n},\]
    with $C_n =  \cinf \alphaval \ceil{(\cinf - 1 - \log(\cinf))^{-1}
\log\bpar{2\frac{n}{\varepsilon}}}$.
    \item Let $\alphaval = \left(\frac{\sqrt{3}\cdot 64\cupp e}{\cinf^2}\right)^{2/7}\bpar{\frac{L_2^2\gapf}{L^3n}}^{\frac{1}{7}}$. If $n$ is large enough to ensure $2 (\cupp n)^{-1}< \alphaval \le 1$, then we have
        \[\E{\1_{\Abar}\min_{ k \in \{1,\dots n \}}\norm{\nabla f(\overline{x}_k)}}\le \Bconst L^{\frac{2}{7}}L_2^{\frac{1}{7}} \gapf^{\frac{4}{7}} n^{-4/7},\]
        with
        \[\Bconst :=(1 + C_n + C_n^2/4)\Hconst 7\,6^{-6/7}\,12^{3/7}
(64e)^{1/7}\bpar{1-\frac{2}{\cupp \alphaval n}}^{-1}\cupp^2\bpar{\frac{\cupp}{\cinf^2}}^{1/7}e^{2\frac{\cupp}{\cinf}C_n},\]
 with $C_n =  \cinf \alphaval \ceil{(\cinf - 1 - \log(\cinf))^{-1}
\log\bpar{2\frac{n}{\varepsilon}}}$ and $\Hconst =(1+\alphaval)(1+2\alphaval)(1+\alphaval 8 (1+\alphaval))$.
\end{enumerate} 
In both statements, $ \overline{x}_k = \sum_{i=0}^{k-1} \lambda_{i,k}\tilde{y}_i$ with $\lambda_{i,k} = \frac{ e^{\alphaval T_{i+1}}}{\sum_{j =1}^k  e^{\alphaval T_j}}$, and $\mathbb{P}(\Abar) \ge 1-\varepsilon$.
\end{theorem}
While the constant $\varepsilon$ determines the size of the set $\Abar$, the role of $\cinf$ and $\cupp$ might be more difficult to grasp. These two constants appear during the proof of Theorem~\ref{thm:high_proba_bound}. Choosing $\cinf$
close to one, which makes $\cupp$ close to $1$, makes $C_n$ increase. Choosing $\cinf\ll 1$, inducing $\cupp \gg 1$, reduces $C_n$ but still hurts the bound because of factors such as $\cupp^2$. 

\begin{proof}
Corollary~\ref{cor:choice1} and \ref{cor:choice2} provide bounds that involve $\min_{t \in \{T_1,\dots,T_n \}} \norm{\nabla f(\overline{x}_t)}$, which is about the continuous-time process $\overline{x}_t$. We have to make sure that the bound will be achievable by our algorithm, precisely, we have to show that the point $\arg \min_{t=\{T_1,\dots,T_n\}}\norm{\nabla f(\overline{x}_t)}$ can be outputted by \eqref{alg:constant_param_det}. We have
\begin{align*}
    \overline{x}_k : = \overline{x}_{T_k} = \int_{[0,T_k]}w_{T_k}(s)x_{s^-} dN_s = \sum_{i=1}^k w_{T_k}(T_i)x_{T_i^-} =  \sum_{i=0}^{k-1} w_{T_k}(T_{i+1})\tilde y_i,
\end{align*}
using $x_{T_i^-} = \tilde y_{i-1}$ by Proposition~\ref{prop:discretization_constant_param}. Also, recall that $w_t(s) = \frac{e^{\alphaval s}}{\int_0^t e^{\alphaval s}dN_s}$, such that we fix 
\begin{align*}
     \lambda_{i,k} := w_{T_k}(T_{i+1}) =  \frac{ e^{\alphaval T_{i+1}}}{\sum_{j =1}^k  e^{\alphaval T_j}}.
\end{align*}
So, for $ \overline{x}_k = \sum_{i=0}^{k-1} \lambda_{i,k}\tilde{y}_i$, $\lambda_{i,k} = \frac{ e^{\alphaval T_{i+1}}}{\sum_{j =1}^k  e^{\alphaval T_j}}$, we have
\[\min_{t \in \{T_1,\dots,T_n \}} \norm{\nabla f(\overline{x}_t)} = \min_{ k \in \{1,\dots n \}}\norm{\nabla f(\overline{x}_k)}.\]


    Then, the result is a direct combination of the bounds provided Corollary~\ref{cor:choice1} and \ref{cor:choice2} with Lemma~\ref{lem:lower_bound_delta}. We just have to make sure that $\left(1-\frac{2}{\cupp \alphaval n}\right) > 0$. In the case $\alphaval = n^{-1/7}$, it holds as long as $n\ge 3$. More generally it is verified as long as $\alphaval > \frac{2}{\cupp n}$. In the case $\alphaval = \left(\frac{\sqrt{3}\cdot 64\cupp e}{\cinf^2}\right)^{2/7}\bpar{\frac{L_2^2\gapf}{Ln}}^{\frac{1}{7}}$, it induces the following constraint on $n$
        \[n > \left(\frac{\sqrt{3}\cdot 64\cupp e}{\cinf^2}\right)^{-1/3}\bpar{\frac{L_2^2\gapf}{L}}^{-\frac{1}{6}} 2^{7/6}\cupp^{-7/6}.\]
\end{proof}

We can conclude with the proof of Theorem~\ref{thm:hess_lip}, which can be seen as a corollary of Theorem~\ref{thm:hess_lip_full}.
\begin{proof}[Proof of Theorem~\ref{thm:hess_lip}]
It is a direct application of $(ii)$ in Theorem~\ref{thm:hess_lip_full}, fixing $\cinf = 1- \tilde \varepsilon$. Then, recalling the expression of $\Bconst$
\begin{equation}\label{eq:b_n_value}
    \Bconst :=(1 + C_n + C_n^2/4)\Hconst 7\,6^{-6/7}\,12^{3/7}
(64e)^{1/7}\bpar{1-\frac{2}{\cupp \alphaval n}}^{-1}\cupp^2\bpar{\frac{\cupp}{(1- \tilde \varepsilon)^2}}^{1/7}e^{2\frac{\cupp}{1- \tilde \varepsilon}C_n},
\end{equation}
 with $C_n =  (1- \tilde \varepsilon) \alphaval \ceil{(-\tilde \varepsilon - \log(1- \tilde \varepsilon))^{-1}
\log\bpar{2\frac{n}{\tilde \varepsilon}}}$ and $\Hconst =(1+\alphaval)(1+2\alphaval)(1+\alphaval 8 (1+\alphaval))$. As $\alphaval = \bigO(1/n^{1/7})$, we have that $C_n \to_{n \to +\infty} 0$, $\Hconst \to_{n \to +\infty} 1$, and $\bpar{1-\frac{2}{\cupp \alphaval n}}^{-1} \to_{n \to + \infty} 1$. So,
\[\lim_{n \to +\infty}\Bconst=7\,6^{-6/7}\,12^{3/7}
(64e)^{1/7}\cupp^2\bpar{\frac{\cupp}{(1- \tilde \varepsilon)^2}}^{1/7}\le 9.2\cupp^2\bpar{\frac{\cupp}{(1- \tilde \varepsilon)^2}}^{1/7}.\]

It remains to show the bound $\cupp \le 1+2\tilde \varepsilon$, which is done in Lemma~\ref{lem:cupp}.

\end{proof}
\section{Conclusion}
In this work, we show that \eqref{eq:nm} with randomized parameter achieves the $\bigO(\varepsilon^{-7/4})$ complexity with high probability, demonstrating that it is not necessary to safeguard momentum with reset mechanisms. Prior to this work, the lack of necessity of such mechanisms was known in the continuous case \cite{okamura2024primitive}, but not in the discrete case.
It was argued in \citep{hermant2025continuized} that in general, the continuized method can tighten the gap between result derived for ODEs--such as \eqref{eq:HB}--and for algorithms. In this work, we provide a new concrete example for the class of functions with Lipschitz gradient and Hessian.


In another line of work, we note that \citep{convexguilty} showed that for functions with Lipschitz gradient, Hessian and third derivative, one can further improve the complexity to $\bigO(\varepsilon^{-5/3})$. The underlying algorithm is arguably complicated, but still based on \eqref{eq:nm_prime} as a core component. It is of interest whether simpler algorithms--such as \eqref{alg:constant_param_det}--can achieve this complexity, even with a certain probability.

\section*{Acknowledgment}
This work was supported by PEPR PDE-AI and ANR SOS2ID (grant ANR-24-CE40-3786). We thank Raphaël Berthier for helpful discussions on the limit of the continuized system in the vanishing stepsize regime.

\bibliographystyle{plain} 
\bibliography{bib}

\appendix

\section{Complementary details to Section~\ref{sec:intro}}\label{app:intro}

\subsection{Accelerated gradient algorithms under convexity}\label{app:convex_acceleration}
A seminal work \cite{POLYAK19641} shows that the heavy ball momentum algorithm
$$ x_{n+1} = x_n + \alphaval (x_n-x_{n-1}) - \beta \nabla f(x_n)$$ can significantly improve over the convergence speed of gradient descent when minimizing strongly convex quadratic functions. \cite{nesterov1983method,nesterov2004introductory} further shows \eqref{eq:nm} ensures an acceleration result for $L$-smooth (strongly)-convex functions. As mentioned in Section~\ref{sec:intro}, the corresponding bounds are optimal among all first-order algorithms \cite{nemirovskij1983problem,nesterovbook}. There exists a rich literature building upon these contributions. Far from being exhaustive, it includes new interpretations of the algorithms \cite{suboydcandes,AllenZhuOrecchia2017}, extensions to the non-smooth case \textit{e.g.} via composite proximal optimization \cite{beck2009fast,chambolle2015convergence, attouch2018convergence}, extensions to the stochastic gradient case \cite{ghadimi2016accelerated,allen2018katyusha, hermantgradient}, and acceleration results under relaxed geometrical assumptions \cite{hinder2020near,aujdossrondPL,aujol2019optimal}. To add some nuance, it is worth noting that even in the convex setting, inertial methods do not necessarily ensure accelerated convergence \cite{goujaud2022optimal,goujaud2023provable}.

\subsection{Second-order stationary points}\label{app:second_order_critic}
We provide a rigorous statement of the notion of second-order $\varepsilon$-stationary point, mentioned in our introductory section. We then briefly comment on it.
\begin{definition}[Second-order $\varepsilon$-stationary point, \cite{nesterov2006cubic}]\label{def:second_order}
    For a $L_2$-Lipschitz Hessian function $f$, $x$ is an $\varepsilon$-second-order stationary point if
    $$ \norm{\nabla f(x)} \le \varepsilon, \quad \lambda_{\min} (\nabla^2 f(x)) \ge -\sqrt{\varepsilon L_2}.$$
\end{definition}
A second-order $\varepsilon$-stationary point approximates a second-order stationary point, namely a point $x^\ast$ such that 
    $$ \norm{\nabla f(x^\ast)} = 0, \quad \lambda_{\min} (\nabla^2 f(x^\ast)) \ge 0.$$
In particular, such points are ensured to avoid \textit{strict saddle-points}, that is, stationary points such that $\lambda_{\min} (\nabla^2 f(x^\ast)) < 0$. The 1d function $x \mapsto x^3$ is a simple example of a function that has a second-order stationary point (in $x =0$) that is not a local minimizer, thus called \textit{non-strict saddle-points}. Fortunately, several works indicate that many applications are prevented from such pathological stationary points \cite{choromanska2015loss,kawaguchi2016deep,bandeira2016low,mei2017solving,boumal2016non,bhojanapalli2016global,ge2016matrix,ge2017no}.

\subsection{Detailed formulations of existing safety check and alternative mechanisms}\label{app:alg}
For completeness, we give detailed versions of existing algorithms that achieve the complexity $\bigO(e^{-7/4})$. Algorithm~\ref{alg:jin} is a detailed version of an algorithm that performs a negative curvature step \cite{momentumsaddle}. Algorithm~\ref{alg:ragd-nc} is a detailed version of an algorithm that uses a restart mechanism \cite{li2023restarted}. We make a few remarks:
\begin{itemize}
    \item We chose to state Algorithm~\ref{alg:jin} as the version from \cite{momentumsaddle} rather than the earlier one from \cite{convexguilty}, because of the significant length and trickiness of the latter.
    \item In its original statement, Algorithm~\ref{alg:jin} also involves a "stochastic perturbation step", used to find an $\varepsilon$-second-order stationary point (Definition~\ref{def:second_order}). Because we do not study this concept, we avoided this step in our presentation. 
    \item In the convergence result associated with Algorithm~\ref{alg:ragd-nc}
\cite[Theorem~1]{li2023restarted}, the restart threshold is set to
$B=\sqrt{\varepsilon/L_2}$, where $\varepsilon$ denotes the target accuracy and
$L_2$ is the Lipschitz constant of the Hessian. As noted by the authors, this value
of $B$ can be very small in practice, causing the algorithm to restart at nearly
every iteration. In such cases, the method effectively reduces to standard
gradient descent, which eliminates the practical benefit of momentum. This
behavior is confirmed by numerical experiments, showing little to no empirical
acceleration due to the near absence of momentum steps.

To address this issue, the authors propose an alternative restart scheme
\cite[Algorithm~2]{li2023restarted}, in which restarts are allowed to occur less
frequently. While keeping $B=\sqrt{\varepsilon/L_2}$,
the user is permitted to initialize a larger value $B_0$, potentially satisfying
$B_0 \gg B$. The restart condition is then replaced by
\[
k \sum_{t=0}^{k-1} \|x^{t+1}-x^t\|^2 > \max\{B^2, B_0^2\},
\]
and $B_0$ is gradually decreased until $B_0 \le B$, at which point the algorithm
reduces to Algorithm~\ref{alg:ragd-nc}. It is shown
\cite[Theorem~2]{li2023restarted} that this modification introduces an additional
$\mathcal{O}\!\left(L^{1/2}\varepsilon^{-1/4}L_2^{1/4}\log\!\left(\tfrac{L_2 B_0}{\varepsilon}\right)\right)$
term in the gradient complexity, which does not significantly degrade the overall
rate. This modification shows great empirical behaviour, enabling acceleration over gradient descent. However, this alternative scheme requires additional function evaluations,
leading to a total complexity of $\mathcal{O}(\varepsilon^{-3/2})$ function
queries.

        \end{itemize}

\begin{algorithm}[t]
\caption{Negative curvature exploitation and Nesterov momentum - Detailed version \cite{momentumsaddle}}
\label{alg:jin}
\KwIn{$x_0, \eta, \theta, \gamma, s$}

$v_0 \leftarrow 0\;$

\For{$t = 0,1,\dots$}{

    $y_{t} \leftarrow x_{t} + (1-\theta) v_t$\;
    
    $x_{t+1} \leftarrow y_t - \eta \nabla f (y_t)$\;
    
    $v_{t+1} \leftarrow x_{t+1} - x_t $\;

\If{$f(x_t) \le  f(y_t) + \dotprod{ \nabla f(y_t), x_t - y_t }- \frac{\gamma}{2} \norm{x_t - y_t}^2$}{$(x_{t+1}, v_{t+1}) \leftarrow $ Negative-Curvature-Exploitation($x_t, v_t, s$)}
}

\SetKwBlock{Begin}{Negative-Curvature-Exploitation$\left(x_t, v_t, s\right)$}{end sous-algorithme A}
\Begin{
\If{$\norm{v_t} \ge s$}{
$x_{t+1} \leftarrow x_t$}
\Else{
$\delta = s\cdot v_t/\norm{v_t}$

$x_{t+1} \leftarrow \argmin_{x \in \{x_t + \delta, x_t - \delta\}} f(x)$
}

    \Return{($x_{t+1}, 0)$}\;
}
\end{algorithm}

\begin{algorithm}[t]
\caption{Restarted Nesterov Momentum - Detailed version \cite{li2023restarted}}
\label{alg:ragd-nc}

\DontPrintSemicolon

\KwIn{$x^{-1} = x^{0} $, $B>0$, $K$, $\eta$,  $\theta$.}

$k \leftarrow 0$\;

\While{$k < K$}{
    $y^{k} \leftarrow x^{k} + (1-\theta)(x^{k} - x^{k-1})$\;

    $x^{k+1} \leftarrow y^{k} - \eta \nabla f(y^{k})$\;
    $k \leftarrow k+1$\;

    \If{$
        k \displaystyle\sum_{t=0}^{k-1}
        \|x^{t+1} - x^{t}\|^2 > B^2
    $}{
        $x^{-1} \leftarrow x^{k}$,\;
        $x^{0} \leftarrow x^{k}$,\;
        $k \leftarrow 0$\;
    }
}

$K_0 \leftarrow
{\arg\min}_{\lfloor K/2\rfloor \le k \le K-1}
\|x^{k+1} - x^{k}\|
$\;

$\hat{y} \leftarrow
\dfrac{1}{K_0+1}
\sum_{k=0}^{K_0} y^{k}
$\;

\Return{$\hat{y}$}\;

\end{algorithm}
\section{Proof of Lemma~\ref{lem:l_smooth} and Proposition~\ref{thm:cv_L_smooth}}\label{app:l_smooth}
Under Assumption~\ref{ass:l_smooth}, we have this classical descent lemma \citep{nesterov2004introductory}.
\begin{lemma}\label{lem:descent_sgc}
Under Assumption~\ref{ass:l_smooth}, if $\gamma \le \frac{1}{L }$, we have

\begin{equation*}
    f(x-\gamma \nabla f(x)) - f(x) \leq \gamma\left(\frac{L }{2}\gamma - 1 \right) \lVert  \nabla f(x) \rVert^2.
\end{equation*}
\end{lemma}
To perform Lyapunov analysis, our tool is the following Itô formula, which intuitively allows stochastic derivation.
\begin{proposition}[\cite{even2021continuized}, Proposition 2]\label{prop:sto_calc_abridged}
Let $\varphi : \mathbb{R}^d \to \mathbb{R}$ be a smooth function.
Let $x_t \in \R^d$ be a solution of 
\begin{equation*}
    dx_t = \zeta(x_t)dt + G(x_{t^-})dN_t
\end{equation*}
where $\zeta : \R^d \to \R^d$ is a locally Lipschitz function and $G : \R^d \to \R^d$ is measurable.  Then,
\begin{equation}\label{eq:prop_calc_sto_core_abbriged}
    \varphi(x_t) = \varphi(x_0) + \int_{0}^t \dotprod{\nabla \varphi(x_s),\zeta(x_s)}ds + \int_{[0,t]}  \varphi(x_s + G(x_s)) - \varphi(x_s)ds + M_t,
\end{equation}

where $M_t$ is a martingale such that $\E{M_t}=0$, $\forall t\ge 0$.

\end{proposition}

We first restate Lemma~\ref{lem:l_smooth}.
\lsmooth*
\begin{proof}
 
Let $\overline{x}_t= (x_t,z_t)$ where $\xz$ satisfies \eqref{eq:nest_continuized}. It satisfies $d\overline{x}_t = \zeta(\overline{x}_t)dt + G(\overline{x}_t)dN(t)$, where
\begin{equation*}
    \zeta(\overline{x}_t) = \begin{pmatrix}
     \eta(z_t-x_t) \\ \eta'(x_t-z_t)
    \end{pmatrix},\quad G(\overline{x}_t) = \begin{pmatrix}
    -\gamma_t \nabla f(x_t) \\ -\gamma_t' \nabla f(x_t)
    \end{pmatrix}
\end{equation*}
We apply Proposition~\ref{prop:sto_calc_abridged} to $\varphi(\overline{x}_t)$, where
\begin{equation*}
    \varphi(x,z) = f(x) + \frac{1}{2}\norm{x-z}^2,
\end{equation*}
inducing
\begin{equation}\label{eq:prop_calc_sto_lsmooth}
    \varphi(\overline{x}_t) = \varphi(\overline{x}_0) + \int_{0}^t \dotprod{\nabla \varphi(\overline{x}_s),\zeta(\overline{x}_s)}ds + \int_0^t  \varphi(\overline{x}_s + G(\overline{x}_s)) - \varphi(\overline{x}_s)ds + M_t,
\end{equation}
where $M_t$ is a martingale.

\paragraph{Computations} 
We compute $\dotprod{\nabla \varphi(\overline{x}_s),\zeta(\overline{x}_s)} $ and $ \varphi(\overline{x}_s+ G(\overline{x}_s)) - \varphi(\overline{x}_s)$.
We have 
\begin{equation*}
    \frac{\partial \varphi}{\partial x} = \nabla f(x) + x-z,\quad  \frac{\partial \varphi}{\partial z} = z - x.
\end{equation*}
So, 
\begin{align}
    \dotprod{\nabla \varphi(\overline{x}_s),\zeta(\overline{x}_s)} &=\dotprod{\nabla f(x_{s}) + x_s - z_s,\eta(z_s-x_s)} + \dotprod{z_s-x_s,\eta'(x_s - z_s)} \nonumber\\
    &= \eta \dotprod{\nabla f(x_{s}),z_s-x_s} - \bpar{\eta + \eta'}\norm{x_s-z_s}^2 \label{eq:lem_l_smooth:1}.
\end{align}
Also, 
\begin{equation}
    \begin{aligned}  \label{eq:lem_l_smooth:2}
                 \varphi(\overline{x}_s+ G(\overline{x}_s)) - \varphi(\overline{x}_s) &= f(x_s - \gamma \nabla f(x_s)) - f(x_s)\\
   &+ \frac{1}{2}\bpar{\norm{x_s - \gamma \nabla f(x_s) - (z_s - \gamma'\nabla f(x_s)}^2 - \norm{x_s-z_s}^2}\\
    &=f(x_s - \gamma \nabla f(x_s)) - f(x_s)  + \frac{(\gamma' - \gamma)^2}{2}\norm{\nabla f(x_s)}^2 \\
    &+ (\gamma' - \gamma)\dotprod{\nabla f(x_s),x_s - z_s}\\
    &\le \bpar{(\gamma' - \gamma)^2 - \gamma(2-L \gamma )}\frac{1}{2}\norm{\nabla f(x_{s})}^2\\
    &+ (\gamma' - \gamma)\dotprod{\nabla f(x_{s}),x_s - z_s}.
    \end{aligned}
\end{equation}
The last inequality uses Lemma~\ref{lem:descent_sgc}, assuming $\gamma \le \frac{1}{ L}$.
We combine \eqref{eq:lem_l_smooth:1} and \eqref{eq:lem_l_smooth:2}
\begin{equation}
    \begin{aligned}\label{eq:lem_l_smooth_final}
                &\dotprod{\nabla \varphi(\overline{x}_s),\zeta(\overline{x}_s)}  +    \varphi(\overline{x}_s+ G(\overline{x}_s)) - \varphi(\overline{x}_s) \\
    &\le (\eta - (\gamma' - \gamma))\dotprod{\nabla f(x_{s^-}),z_s-x_s} - (\eta + \eta')\norm{x_s - z_s}^2 \\
    &+  \bpar{(\gamma' - \gamma)^2 - \gamma(2-L \gamma)}\frac{1}{2}\norm{\nabla f(x_{s})}^2 
    \end{aligned}
\end{equation}

\paragraph{Parameter tuning}
We now choose the parameters driving \eqref{eq:nest_continuized} such that we obtain the desired result.
We set $\eta =  \gamma' - \gamma$ to cancel scalar product. Then fixing $ \gamma' = \gamma + \sqrt{\frac{\gamma}{2 }}$ ensures 
$$(\gamma' - \gamma)^2 - \gamma(2-L \gamma) \le \frac{\gamma}{2} - \gamma = -\frac{\gamma}{2}.$$ The latter holds because $\gamma \le \frac{1}{ L} \Rightarrow -\gamma(2-L  \gamma) \le - \gamma$. We thus have
\begin{equation}\label{eq:L_smooth_eq1}
     \dotprod{\nabla \varphi(\overline{x}_s),\zeta(\overline{x}_s)}  +    \varphi(\overline{x}_s+ G(\overline{x}_s)) - \varphi(\overline{x}_s)\le -(\eta + \eta')\norm{x_s - z_s}^2 - \frac{\gamma}{4}\norm{\nabla f(x_{s})}^2.
\end{equation}
\paragraph{Conclusion}
Combining \eqref{eq:L_smooth_eq1}, \eqref{eq:lem_l_smooth_final} and \eqref{eq:prop_calc_sto_lsmooth},
we get
\begin{align}
     &\varphi(\overline{x}_t) \le \varphi(\overline{x}_0)  + \int_{0}^{t} -(\eta + \eta')\norm{x_s - z_s}^2 - \frac{\gamma}{4}\norm{\nabla f(x_{s})}^2 ds + M_{t} \nonumber \\
     \Rightarrow&\int_{0}^{t} (\eta + \eta')\norm{x_s - z_s}^2 + \frac{\gamma}{4}\norm{\nabla f(x_{s})}^2 ds\le f(x_0)- f^\ast + M_t ,\label{eq:l_smooth_lyap_dec_proof}
\end{align} using $\varphi(\overline{x}_0) - \varphi(\overline{x}_t) \le f(x_0)-\min_x f(x)$ and that $x_0 = z_0$. It remains to take expectation on \eqref{eq:l_smooth_lyap_dec_proof}
\begin{align*}
    \mathbb{E}\left[ \int_{0}^{t} (\eta + \eta')\norm{x_s - z_s}^2 + \frac{\gamma}{4}\norm{\nabla f(x_{s})}^2 ds \right] \le \mathbb{E}\left[  f(x_0)- f^\ast \right],
\end{align*}
where we used $\mathbb{E}[M_t] = \mathbb{E}[M_0] = 0$.
\end{proof}

We now prove the discrete version of the result, which we first restate.
\cvLsmooth*

\begin{proof}

From \eqref{eq:l_smooth_lyap_dec_proof}, we have
\begin{equation*}
    \int_{0}^{t}  \frac{\gamma}{4}\norm{\nabla f(x_{s})}^2 ds\le f(x_0) - f^\ast + M_t.
\end{equation*}
Also, we have
\begin{equation}\label{eq:poisson_mart}
\int_0^t \frac{\gamma}{4}\norm{\nabla f(x_{s})}^2 ds  = \int_0^t \frac{\gamma}{4}\norm{\nabla f(x_{s^{-}})}^2 dN_s + \int_0^t \frac{\gamma}{4}\norm{\nabla f(x_{s^{-}})}^2 (ds - dN_s),
    \end{equation}
where $\int_0^t \frac{\gamma}{4}\norm{\nabla f(x_{s^{-}})}^2 (ds - dN_s)$ is a centered martingale.
 Then, there exists a centered martingale $U_t$ such that 
\begin{equation*}
    \int_0^t \frac{\gamma}{4}\norm{\nabla f(x_{s^{-}})}^2 dN_s \le 
    \gapf + U_t.
\end{equation*}
We use the following stopping theorem.
\begin{theorem}[\cite{hermant2025continuized}, Theorem 6]\label{thm:martingal_stopping}
    Let $(\varphi_t)_{t\in \R_+}$ be a nonnegative process with cadlag trajectories, such that it verifies
    $$ \varphi_t \le K_0 + M_t, $$
    for some positive random variable $K_0$, some martingale $\mart$ with $M_0 = 0$. Then, for an almost surely finite stopping time $\tau$, one has 
$$ \E{\varphi_\tau} \le \E{K_0}.$$
\end{theorem}
Applying Theorem~\ref{thm:martingal_stopping} with $\tau = T_k$, $\varphi_t =  \int_{0}^{t}  \frac{\gamma}{4}\norm{\nabla f(x_{s^-})}^2 dN_s$, $K_0 = \gapf$ and $M_t = U_t$, we obtain
\begin{equation*}
    \mathbb{E}\left[\int_0^{T_k} \norm{\nabla f(x_{s^{-}})}^2dN_s\right] \le \frac{4}{\gamma }\gapf.
\end{equation*}
Finally, we have 
\begin{equation*}
    \int_0^{T_k} \norm{\nabla f(x_{s^{-}})}^2dN_s = \sum_{i = 1}^k \norm{\nabla f(x_{T_i^{-}})}^2 = \sum_{i = 1}^k \norm{\nabla f(\tilde{y}_{i-1})}^2 \ge k \min_{0\le i \le k-1}\norm{\nabla f(\tilde{y}_i)}^2,
\end{equation*}
thus the result.
\end{proof}

\section{Proof of Theorem~\ref{thm:cv_hb}}\label{app:proof_hb_cv}

 Recall $\gamma' = \gamma + (\gamma/2)^{\frac{1}{2}}$, $\eta =  (\gamma/2)^{\frac{1}{2}}$, $\eta'_T= \alphavalcontf{(\gamma/2)^{-1/2}T}- (\gamma/2)^{\frac{1}{2}} $. From the proof of Lemma~\ref{lem:l_smooth}, we have
\begin{align*}
    \mathbb{E}\left[ f(x_t)-f^\ast \right]
    \le \mathbb{E}\left[  f(x_0)- f^\ast \right].
\end{align*}
Also, by $L$-smoothness we have $\norm{\nabla f(x)}^2 \le 2L(f(x)-f^\ast)$ for any $x\in \R^d$ \cite[Section 1.2.3]{nesterov2004introductory}, such that for any $t \ge 0$,
\begin{align*}
    \mathbb{E}\left[\norm{\nabla f(x_t)}^2 \right]
    \le 2L(f(x_0)- f^\ast).
\end{align*}
Taking the $\sup$ on $t$, we get
\begin{align}\label{eq:uniform_bound}
    \sup_{t > 0}\mathbb{E}\left[\norm{\nabla f(x_t)}^2 \right]
    \le 2L(f(x_0)- f^\ast).
\end{align}
We define the rescaled process $(\xrescaled,\zrescaled) := (x_{ (\gamma/2)^{-1/2}s},z_{ (\gamma/2)^{-1/2}s})$ and $\Nrescaled := N_{  (\gamma/2)^{-1/2}s}$ on the interval $s \in [0,T]$. This means that the process $\xcont$ has to be define on the interval of time $s \in [0,(\gamma/2)^{-1/2}T]$, which justifies the choice of terminal time in $\alphavalcontf{(\gamma/2)^{-1/2}T} $.  We also have that $\Mrescaled := \Nrescaled -  (\gamma/2)^{-1/2}s$ is a martingale,  whose predictable quadratic variation is
\begin{equation}\label{eq:quad_var}
    \dotprod{\Mrescaled} =  (\gamma/2)^{-1/2}s.
\end{equation}
The rescaled process satisfies
\begin{align*}
    \left\{
    \begin{array}{ll}
        d\xrescaled &= \eta(\gamma/2)^{-1/2}(\zrescaled-\xrescaled)ds -  (\gamma/2)^{1/2}  \nabla f(\xrescaledprev) ds -  \gamma  \nabla f(\xrescaledprev) d\Mrescaled, \\
        d\zrescaled &= \eta'(\gamma/2)^{-1/2}(\xrescaled-\zrescaled)ds -  \gamma' (\gamma/2)^{-1/2} \nabla f(\xrescaledprev)ds-  \gamma'  \nabla f(\xrescaledprev)d\Mrescaled,
    \end{array}
\right.
\end{align*}
which becomes, once the parameters are replaced by their dependence in $\gamma$
\begin{align*}
    \left\{
    \begin{array}{ll}
        d\xrescaled &= (\zrescaled-\xrescaled)ds -  (\gamma/2)^{1/2}  \nabla f(\xrescaledprev) ds -  \gamma  \nabla f(\xrescaledprev) d\Mrescaled, \\
        d\zrescaled &= ( 2^{\frac{3}{7}}K_{\tilde \varepsilon,\cupp}\bpar{\frac{L_2^2 \Delta_f}{T}}^{1/7} -1)(\xrescaled-\zrescaled)ds -  ((2\gamma)^{1/2}  +1) \nabla f(\xrescaledprev)ds-  \gamma'  \nabla f(\xrescaledprev)d\Mrescaled,
    \end{array}
\right.
\end{align*}
By \eqref{eq:uniform_bound}, we deduce
\begin{equation}\label{eq:change_variable}
\begin{aligned}
        \mathbb{E}\left[ \int_{0}^{t}\norm{\nabla f(\xrescaled)}^2 ds \right] &=\int_{0}^{t} \mathbb{E}\left[ \norm{\nabla f(\xrescaled)}^2\right] ds \\ &\le t \gapf.
\end{aligned}
\end{equation}
 In the following step, we show that some of the component converges uniformly to zero in probability as $\gamma$ goes to zero.

\textbf{Step 1.}
We have
\[\E{\sup_{s \in [0,T]} \norm{ \gamma^{1/2}  \int_{0}^s \nabla f(x_r^\gamma)dr}^2} \le\E{\sup_{s \in [0,T]}  \gamma s  \int_{0}^s \norm{\nabla f(x_r^\gamma)}^2dr} \le \E{  \gamma T  \int_{0}^T \norm{\nabla f(x_r^\gamma)}^2dr},  \]
where we used Cauchy-Schwartz. Using \eqref{eq:change_variable}, we obtain 
\[\E{\sup_{s \in [0,T]} \norm{ \gamma^{1/2}  \int_{0}^s \nabla f(x_r^\gamma)dr}^2} \le  \gamma T^2 \gapf \to_{\gamma \to 0} 0,\]
which implies that $\sup_{s \in [0,T]} \norm{ \gamma^{1/2}  \int_{0}^s \nabla f(\xrescaled)ds}$ converges to zero in probability. 

Then, by Doob's inequality, we get
\[\E{\sup_{s \in [0,T]} \norm{ \gamma\int_{0}^s \nabla f(x_r^\gamma)dM_r^\gamma}^2} \le 4\E{ \norm{ \gamma\int_{0}^T \nabla f(x_r^\gamma)dM_r^\gamma}^2}.\]
By an isometry theorem and \eqref{eq:quad_var}, we deduce 
\[4\E{ \norm{ \gamma\int_{0}^T \nabla f(x_r^\gamma)dM_r^\gamma}^2}= 4\E{  \gamma^2  \int_{0}^T \norm{\nabla f(x_r^\gamma)}^2d\dotprod{M_r^\gamma}} = 4\E{  \gamma^{3/2}  \int_{0}^T \norm{\nabla f(x_r^\gamma)}^2dr}.  \]
Using again \eqref{eq:change_variable}, we obtain 
\[\E{\sup_{s \in [0,T]} \norm{ \gamma\int_{0}^s \nabla f(x_r^\gamma)dM_r^\gamma}^2} \le \gamma^{3/2}T\gapf \to_{\gamma \to 0} 0, \]
which implies $\sup_{s \in [0,T]} \norm{ \gamma\int_{0}^s \nabla f(x_r^\gamma)dM_r^\gamma}\to_{\gamma \to 0} 0$ in probability.

There remains one term, for which we use the same arguments.
\[\E{\sup_{s \in [0,T]} \norm{ \gamma'\int_{0}^s \nabla f(x_r^\gamma)dM_r^\gamma}^2} \le 4\E{  \gamma'^2  \int_{0}^s \norm{\nabla f(x_r^\gamma)}^2d\dotprod{M_r^\gamma}} =4 \E{  \gamma'^2 \gamma^{-1/2}  \int_{0}^T \norm{\nabla f(x_r^\gamma)}^2dr}.  \]
Using \eqref{eq:change_variable}, it becomes
\[\E{\sup_{s \in [0,T]} \norm{ \gamma'\int_{0}^s \nabla f(x_r^\gamma)dM_r^\gamma}^2} \le 4\gamma'^2 \gamma^{-1/2}T\gapf \le 4\gamma^{1/2}\gapf \to_{\gamma \to 0} 0.\]
Which implies $\sup_{s \in [0,T]} \norm{ \gamma'\int_{0}^s \nabla f(\xrescaled)d\Mrescaled} \to 0$ in probability.

\textbf{Step 2.} Now, we define
\begin{align}\left\{
    \begin{array}{ll}
        dX_s &= (Z_s-X_s)ds, \\
        dZ_s &= (2^{\frac{3}{7}}K_{\tilde \varepsilon,\cupp}\bpar{\frac{L_2^2 \Delta_f}{T}}^{1/7}-1)(X_s-Z_s)ds - \nabla f(X_{s}) ds,
    \end{array}
\right.
\end{align}
We want to show that the rescaled process converges in probability to this system with $\gamma \to 0$. We note $\Delta X_s^{\gamma} := X_s-\xrescaled$ and $\Delta Z_s^{\gamma} := Z_s-\zrescaled$. We have
\[\Delta X_s^{\gamma} = \int_0^s (\Delta Z_r^{\gamma}-\Delta X_r^{\gamma})dr + R_{X,\gamma}(s),  \]
where $R_{X,\gamma}(s) := -(\gamma/2)^{1/2}  \int_{0}^s \nabla f(\xrescaled)ds-\gamma\int_{0}^s \nabla f(\xrescaled)d\Mrescaled$, and 
\[\Delta Z_s^{\gamma} =A\int_0^s (\Delta X_r^{\gamma}-\Delta Z_r^{\gamma})dr - \int_0^s (\nabla f(X_r)-\nabla f(x_r^\gamma))dr + R_{Z,\gamma}(s),  \]
denoting $A:=(2^{\frac{3}{7}}K_{\tilde \varepsilon,\cupp}\bpar{\frac{L_2^2 \Delta_f}{T}}^{1/7}-1)$ and $ R_{(Z,\gamma)(s)}=-(2\gamma)^{1/2}  \int_{0}^s \nabla f(\xrescaled)ds - \gamma'\int_{0}^s \nabla f(\xrescaled)d\Mrescaled$. Using several time the triangular inequality, and $\Vert \int \cdot ds \Vert \le \int \Vert \cdot \Vert ds$, we get
\[\Vert \Delta X_s^{\gamma}\Vert \le\int_0^s (\Vert \Delta Z_r^{\gamma} \Vert + \Vert\Delta X_r^{\gamma})\Vert)dr + \Vert R_{X,\gamma}(s) \Vert. \]
Similarly, 
\[\Vert \Delta Z_s^{\gamma}\Vert \le A\int_0^s (\Vert \Delta Z_r^{\gamma} \Vert + \Vert\Delta X_r^{\gamma})\Vert)dr +\int_0^s \Vert \nabla f(X_r)-\nabla f(x_r^\gamma)\Vert dr + \Vert R_{Z,\gamma}(s) \Vert, \]
and using also that $\nabla f$ is $L$-Lipschitz
\[\Vert \Delta Z_s^{\gamma}\Vert \le A\int_0^s (\Vert \Delta Z_r^{\gamma} \Vert + \Vert\Delta X_r^{\gamma})\Vert)dr + L\int_0^s \Vert  \Delta X_r^\gamma\Vert dr + \Vert R_{Z,\gamma}(s) \Vert. \]
Using the rough inequality $\Vert \Delta X_r^\gamma\Vert \le  \Vert \Delta Z_r^\gamma\Vert+\Vert \Delta X_r^\gamma\Vert$, we get
\begin{equation}
    \begin{aligned}
        \Vert \Delta X_s^\gamma\Vert + \Vert \Delta Z_s^{\gamma}\Vert \le (A + L) \int_0^s (\Vert \Delta Z_r^{\gamma} \Vert + \Vert\Delta X_r^{\gamma})\Vert)dr  + \Vert R_{X,\gamma}(s) \Vert + \Vert R_{Z,\gamma}(s) \Vert.
    \end{aligned}
\end{equation}
We take the $\sup$ on $[0,s]$, denoting $D_\gamma(r) := \sup_{s \in [0,r]} (  \Vert \Delta X_s^\gamma\Vert + \Vert \Delta Z_s^{\gamma}\Vert)$, and use Fatou's Lemma to obtain
\[D_\gamma(s) \le (A+L)\int_0^s D_\gamma(r)dr + \sup_{s \in [0,T]}\Vert R_{X,\gamma}(s) \Vert + \sup_{s \in [0,T]}\Vert R_{Z,\gamma}(s) \Vert.\]
Using the Grönwall Lemma, we deduce
\[D_\gamma(T) \le e^{(A+L)T}\bpar{\sup_{s \in [0,T]}\Vert R_{X,\gamma}(s) \Vert + \sup_{s \in [0,T]}\Vert R_{Z,\gamma}(s) \Vert}. \]
Yet, from Step 1 we can deduce
\[\sup_{s \in [0,T]}\Vert R_{X,\gamma}(s) \Vert + \sup_{s \in [0,T]}\Vert R_{Z,\gamma}(s) \Vert \to_{\gamma \to 0} 0\]
in probability. We can conclude that in probability, $D_\gamma(T) \to_{\gamma \to 0} 0$, or 
\[\sup_{s \in [0,T]} \norm{X_s-\xrescaled} \overset{\mathbb{P}}{\to}_{\gamma \to 0} 0, \quad \sup_{s \in [0,T]} \norm{Z_s-\zrescaled} \overset{\mathbb{P}}{\to}_{\gamma \to 0} 0.\]
To conclude, we use \cite[Proposition 27]{hermant2025continuized}, to deduce that $X_s$ can be rewritten as the solution of the following ODE
 \[\ddot X_s +  2^{\frac{3}{7}}K_{\tilde \varepsilon,\cupp}\bpar{\frac{L_2^2 \Delta_f}{T}}^{1/7}\dot X_s + \nabla f(X_s) = 0.\]

\section{Proof of the Lemmas stated in Section~\ref{sec:proof}}
We prove formally all the Lemmas used in Section~\ref{sec:proof} to prove Theorem~\ref{thm:hess_lip}. 

\subsection{Proof of Lemma~\ref{lem:hess_lip_2}}\label{app:lem_1}
  The proof follows similar steps as for Lemma~1 in \cite{okamura2024primitive}. We show the following inequalities
    \begin{align}
        &\norm{\int_0^t w_t(s)\nabla f(x_{s^-})dN_s - \nabla f(\overline{x}_t)} \nonumber\\
        &\le \frac{L_2}{2}\int_0^t w_t(s)\norm{x_{s^-}-\overline{x}_{t}}^2 dN_s\label{eq:hess:lem1:1}\\
        &= \frac{L_2}{2}\int\int_{0\le \sigma \le \tau \le t} w_t(\sigma)w_t(\tau)\norm{x_{\tau^-} - x_{\sigma^-}}^2 dN_\sigma dN_\tau\label{eq:hess:lem1:2}\\
        &\le  L_2\int_0^t \eta^2 \norm{z_s-x_s}^2 \bpar{\int_s^t \int_0^s    w_t(\sigma)w_t(\tau)(\tau-\sigma)dN_\sigma dN_\tau}ds\nonumber\\
    &+L_2\int_0^t \gamma^2 \norm{\nabla f(x_{s^{-}})}^2 \bpar{\int_s^t \int_0^s    w_t(\sigma)w_t(\tau)(N_{\tau^-}-N_{\sigma^-})dN_\sigma dN_\tau}dN_s.\label{eq:hess:lem1:3}
    \end{align}

\paragraph{Inequality \eqref{eq:hess:lem1:1}.}
Using Taylor expansion of the gradient, we have

\begin{align*}
   & \nabla f(x_{s^-})-\nabla f(\overline{x}_t) \\
   &= \int_0^1 \nabla^2 f((1-\sigma)\overline{x}_t + \sigma x_{s^-})(x_{s^-}-\overline{x}_t)d\sigma\\
    &= \nabla^2 f(\overline{x}_t)(x_{s^-}-\overline{x}_t) + \int_0^1 (\nabla^2 f((1-\sigma)\overline{x}_t + \sigma x_{s^-})- \nabla^2 f(\overline{x}_t))(x_{s^-}-\overline{x}_t)d\sigma.
\end{align*}
Multiply both sides by $w_t(s)$ and integrate with respect to $dN(\cdot)$ on $s\in [0,t]$
\begin{equation}
    \begin{aligned}\label{eq:hess:lem:1_bis}
         &   \int_0^t w_t(s)(\nabla f(x_{s^-})-\nabla f(\overline{x}_t))dN_s\\
         &= \int_0^t w_t(s)\nabla^2 f(\overline{x}_t)(x_{s^-} - \overline{x}_t)dN_s\\
    &+ \int_0^t w_t(s)\int_0^1 (\nabla^2 f((1-\sigma)\overline{x}_t + \sigma x_{s^-})-\nabla^2 f(\overline{x}_t))(x_{s^-}-\overline{x}_t)d\sigma dN_s.
    \end{aligned}
\end{equation}
Because $ \int_0^t w_t(s)dN_s = 1$, we have
\begin{equation*}
    \int_0^t w_t(s)(\nabla f(x_{s^-})-\nabla f(\overline{x}_t))dN_s =  \int_0^t w_t(s)\nabla f(x_{s^-})dN_s - \nabla f(\overline{x}_t),
\end{equation*}
and because $\overline{x}_t = \int_0^t w_t(s)x_{s^-}dN_s$, we get
\begin{equation*}
\int_0^t w_t(s)\nabla^2 f(\overline{x}_t)(x_{s^-} - \overline{x}_t)dN_s = 0.
\end{equation*}
Also,
\begin{align*}
    &\norm{\int_0^t w_t(s)\int_0^1 (\nabla^2 f((1-\sigma)\overline{x}_t + \sigma x_{s^-})-\nabla^2 f(\overline{x}_t))(x_{s^-}-\overline{x}_t)d\sigma dN_s}\\ 
    &\le \int_0^t w_t(s)\int_0^1 \norm{\nabla^2 f((1-\sigma)\overline{x}_t + \sigma x_{s^-})-\nabla^2 f(\overline{x}_t)}_2\norm{x_{s^-}-\overline{x}_t}d\sigma dN_s\\
    &\le \int_0^t w_t(s)\int_0^1 \sigma L_2\norm{x_{s^-}-\overline{x}_t}^2d\sigma dN_s\quad (\text{Assumption}~\ref{ass:hess_lip})\\
    &=\frac{L_2}{2}\int_0^t w_t(s)\norm{x_{s^-}-\overline{x}_t}^2 dN_s.
\end{align*}
Finally, taking the norm, \eqref{eq:hess:lem:1_bis} becomes
\begin{equation*}
    \norm{\int_0^t w_t(s)\nabla f(x_{s^-})dN_s - \nabla f(\overline{x}_t)} \le \frac{L_2}{2}\int_0^t w_t(s)\norm{x_{s^-}-\overline{x}_t}^2 dN_s.
\end{equation*}
\paragraph{Equality \eqref{eq:hess:lem1:2}.}

We have

    \begin{align}
    \norm{\int_0^t w_t(\sigma)x_{\sigma^-} dN_\sigma}^2 &= \dotprod{\int_0^t w_t(\sigma)x_{\sigma^-} dN_\sigma,\int_{0}^t w_t(\tau)x_{\tau^-} dN_\tau}\nonumber\\
    &= \int \int_{[0,t]^2} w_t(\sigma)w_t(\tau)\dotprod{x_{\sigma^-},x_{\tau^-}}dN_\sigma dN_\tau. \label{eq:adapt_lemma}
\end{align}

Also
\begin{equation*}
\begin{aligned}
    &\int\int_{0\le \sigma \le \tau \le t} w_t(\sigma)w_t(\tau)\norm{x_{\tau^-} - x_{\sigma^-}}^2 dN_\sigma dN_\tau\\
    &= \frac{1}{2}\int\int_{[0,t]^2} w_t(\sigma)w_t(\tau)\norm{x_{\tau^-} - x_{\sigma^-}}^2 dN_\sigma dN_\tau\\
    &= \int_0^t w_t(\sigma)\norm{x_{\sigma^-}}^2 dN_\sigma - \int\int_{[0,t]^2} w_t(\sigma)w_t(\tau)\dotprod{x_{\sigma^-},x_{\tau^-}} dN_\sigma dN_\tau
    \end{aligned}
    \end{equation*}    
    where the last equality is obtained by developing the squared norm, and because $$\int\int_{[0,t]^2} w_t(\sigma)w_t(\tau)\norm{x_{\tau^-}}^2 dN_\sigma dN_\tau =\int_0^tw_t(\tau)\norm{x_{\tau^-}}^2dN_\tau\int_0^tw_t(\sigma)dN_\sigma = \int_0^tw_t(\tau)\norm{x_{\tau^-}}^2dN_\tau.$$ Then,
    \begin{align*}
    &\int_0^t w_t(\sigma)\norm{x_{\sigma^-}}^2 dN_\sigma - \int\int_{[0,t]^2} w_t(\sigma)w_t(\tau)\dotprod{x_{\sigma^-},x_{\tau^-}} dN_\sigma dN_\tau\\
    &=\int_0^t w_t(\sigma)\norm{x_{\sigma^-}}^2 dN_\sigma  - \norm{\int_0^t w_t(\sigma)x_{\sigma^-} dN_\sigma}^2 , \quad \text{using (\ref{eq:adapt_lemma})}\\
    &=\int_0^t w_t(\sigma)\norm{x_{\sigma^-}}^2 dN_\sigma - \norm{\overline{x}_t}^2\\
    &=\int_0^t w_t(\sigma)\norm{x_{\sigma^-}}^2 dN_\sigma - 2\norm{\overline{x}_t}^2 + \norm{\overline{x}_t}^2\\
    &=\int_0^t w_t(\sigma)\norm{x_{\sigma^-}}^2 dN_\sigma  - 2\dotprod{\int_0^t w_t(\sigma)x_{\sigma^-} dN_\sigma,\underbrace{\int_0^t w_t(\sigma)x_{\sigma^-} dN_\sigma}_{=\overline{x}_t} } \\
    &+ \norm{\overline{x}_t}^2\underbrace{\int_0^t w_t(\sigma)dN_\sigma}_{=1 } \\
    &=\int_0^t w_t(\sigma)\norm{x_{\sigma^-}}^2 dN_\sigma -2\int_0^t w_t(\sigma)\dotprod{x_{\sigma^-},\overline{x}_t}dN_\sigma +\int_0^t w_t(\sigma)\norm{\overline{x}_t}^2dN_\sigma \\
    &=\int_0^t w_t(\sigma)\norm{x_{\sigma^-} - \overline{x}_t}^2dN_\sigma.
\end{align*}
\paragraph{Inequality \eqref{eq:hess:lem1:3}.}

By definition of $\xcont$, for all $\sigma \in [0,t]$ we have
\begin{equation*}
    x_t = x_{\sigma^-} + \int_{\sigma}^t \eta(z_s-x_s)ds - \int_{[\sigma,t]} \gamma \nabla f(x_{s^{-}})dN_s
\end{equation*}
Taking the left-limit at $t$, it becomes
\begin{equation*}
    x_{t^-} = x_{\sigma^{-}} + \int_\sigma^t \eta(z_s-x_s)ds - \int_{[\sigma,t)} \gamma \nabla f(x_{s^{-}})dN_s
\end{equation*}
For $0 \le \sigma \le \tau \le t$, using $\norm{a-b}^2 \le 2\norm{a}^2 + 2\norm{b}^2$ for $a,b\in\R^d$ and Cauchy-Schwarz inequality, we have
\begin{equation}
\begin{aligned}\label{eq:hess:lem1:3_bis}
    \norm{x_{\tau^-} - x_{\sigma^-}}^2 &= \norm{\int_\sigma^\tau  \eta(z_s-x_s)ds - \int_{[\sigma,\tau)} \gamma \nabla f(x_{s^{-}})dN_s}^2\\
    &\le 2\norm{\int_\sigma^\tau  \eta(z_s-x_s)ds}^2+ 2 \norm{\int_{[\sigma,\tau)} \gamma \nabla f(x_{s^{-}})dN_s}^2\\
    &\le 2 \bpar{\int_\sigma^\tau \norm{ \eta(z_s-x_s)}ds}^2 + 2 \bpar{\int_{[\sigma,\tau)}\norm{ \gamma \nabla f(x_{s^{-}})}dN_s}^2\\
    &\le 2 \bpar{\int_\sigma^\tau \norm{ \eta(z_s-x_s)}^2ds}\bpar{\int_\sigma^\tau ds} + 2\bpar{\int_{[\sigma,\tau)} dN_s}\bpar{\int_{[\sigma,\tau)}\norm{ \gamma \nabla f(x_{s^{-}})}^2dN_s}\\
    &\le 2(\tau - \sigma)\bpar{\int_\sigma^\tau \eta^2\norm{ (z_s-x_s)}^2ds} + 2(N_{\tau^-}-N_{\sigma^-})\int_{\sigma}^\tau \gamma^2\norm{\nabla f(x_{s^{-}})}^2dN_s.
\end{aligned}
\end{equation}
In the last inequality, we used $\int_{[\sigma,\tau)}\norm{ \gamma \nabla f(x_{s^{-}})}^2dN_s \le \int_{[\sigma,\tau]}\norm{ \gamma \nabla f(x_{s^{-}})}^2dN_s$, recalling our notation  $\int_{\sigma}^{\tau}dN_s = \int_{[\sigma,\tau]}$.
Then,  we have
\begin{align*}
    &\int\int_{0\le \sigma \le \tau \le t}w_t(\sigma)w_t(\tau)\norm{x_{\tau^-} - x_{\sigma^-}}^2 dN_\sigma dN_\tau\\
    &\overset{\eqref{eq:hess:lem1:3_bis}}{\le} \int\int_{0\le \sigma \le \tau \le t} w_t(\sigma)w_t(\tau)2(\tau - \sigma)\bpar{\int_\sigma^\tau \eta^2\norm{ z_s-x_s}^2ds} dN_\sigma dN_\tau\\
    &+\int\int_{0\le \sigma \le \tau \le t} w_t(\sigma)w_t(\tau)2(N_{\tau^-}-N_{\sigma^-})\bpar{\int_\sigma^\tau \gamma^2\norm{\nabla f(x_{s^{-}})}^2dN_s} dN_\sigma dN_\tau\\
    &= \int\int_{0\le \sigma \le s \le \tau \le t} w_t(\sigma)w_t(\tau)2(\tau - \sigma)\eta^2\norm{ z_s-x_s}^2dsdN_\sigma dN_\tau\\
    &+\int\int_{0\le \sigma \le s \le \tau \le t} w_t(\sigma)w_t(\tau)2(N_{\tau^-}-N_{\sigma^-}) \gamma^2\norm{\nabla f(x_{s^{-}})}^2dN_s dN_\sigma dN_\tau\\
    &=2\int_0^t \eta^2 \norm{z_s-x_s}^2 \bpar{\int_s^t \int_0^s   w_t(\sigma)w_t(\tau)(\tau-\sigma)dN_\sigma dN_\tau}ds\\
    &+2\int_0^t \gamma^2 \norm{\nabla f(x_{s^{-}})}^2 \bpar{\int_s^t \int_0^s   w_t(\sigma)w_t(\tau)(N_{\tau^-}-N_{\sigma^-})  dN_\sigma dN_\tau}dN_s.
\end{align*}

\subsection{Proof of Lemma~\ref{lem:transfer_grad_derivative}}\label{app:by_part_int}
We set   $\tilde w(s) = \alphaval e^{\alphaval s}$, such that $s \mapsto \tilde w(s)$ is differentiable, satisfying
\begin{equation*}
    \frac{d \tilde w(s)}{ds} = \alphaval \tilde w(s).
\end{equation*}
     Summing the two lines of \eqref{eq:nest_continuized} and using $\eta+\eta' = \alphaval$, we get
\begin{align}\label{lem:eq:by_part_1}
    dx_s - dz_s =  \alphaval(z_s-x_s)ds - (\gamma - \gamma')\nabla f(x_{s^-})dN_s.
    \end{align}
    We multiply each side of \eqref{lem:eq:by_part_1} by $\tilde w(s)$ and use $d \tilde w(s) = \alphaval \tilde w(s)ds$ to get
    \begin{align*}
 \tilde w(s)(dx_s - dz_s) &=\tilde  \alphaval w(s)(z_s-x_s)ds -\tilde w(s)(\gamma - \gamma')\nabla f(x_{s^-})dN_s\\
 &= (z_s-x_s)d \tilde w(s) -\tilde w(s)(\gamma - \gamma')\nabla f(x_{s^-})dN_s.
\end{align*}
By rearranging, we obtain
\begin{align}\label{lem:eq:by_part_2}
 \tilde w(s)(dx_s - dz_s) + (x_s-z_s)d\tilde w(s) =-\tilde w(s)(\gamma - \gamma')\nabla f(x_{s^-})dN_s.
\end{align}
We write \eqref{lem:eq:by_part_2} in an integral form, such that
\begin{equation}\label{lem:eq:by_part_3}
    \int_0^t \tilde w(s)(dx_s - dz_s) + \int_0^t (x_s-z_s)d\tilde w(s) =-(\gamma-\gamma') \int_0^t \tilde w(s)\nabla f(x_{s^{-}})dN_s
\end{equation}
We use the following integration by part formula for stochastic integrals, see a precise statement in \cite[p.83]{protter2012stochastic}.
\begin{corollary}[\cite{protter2012stochastic}]\label{cor:protter}
    Let $X$, $Y$ be semimartingales with at least one of $X$ or $Y$ continuous. Then
    \begin{equation*}
        X_t Y_t = X_0Y_0 + \int_0^t X_{s^-}dY_s + \int_0^t Y_{s^-}dX_s. 
    \end{equation*}
\end{corollary}
We apply Corollary~\ref{cor:protter} to $X_t = x_t - z_t$ and $Y_t =\tilde w(t)$, such that $X_0Y_0 = 0$, and so
\begin{equation}\label{lem:eq:by_part_4}
     \int_0^t \tilde w(s)(dx_s - dz_s) + \int_0^t (x_s-z_s)d\tilde w(s) =\tilde w(t)(x_t-z_t).
\end{equation}
We combine \eqref{lem:eq:by_part_3} with \eqref{lem:eq:by_part_4}, and multiply both sides by $\Big(\int_0^t \alphaval e^{\alphaval s}dN_s\Big)^{-1}$ to conclude. Note that we can use Corollary~\ref{cor:protter} as $t \mapsto \tilde w(t)$ is deterministic and continuous, and $x_t - z_t$ is a semimartingale as a difference of semimartingales. To see that $\xcont$ is a semi martingale, remark that it writes
\begin{align*}
    x_t &= x_0 + \eta \int_0^t(z_s-x_s)ds - \gamma\int_0^t  \nabla f(x_{s^-})dN_s\\
    &=x_0 + \underbrace{\int_0^t(\eta (z_s-x_s) - \gamma   \nabla f(x_{s}) )ds}_{\text{adapted + continuous + bounded variation}} - \underbrace{\gamma\int_0^t  \nabla f(x_{s^-})(dN_s-ds)}_{\text{martingale}},
\end{align*}
which is a semimartingale by definition. The same argument holds for $(z_t)_{t\ge 0}$.
\subsection{Proof of Lemma~\ref{lem:hess_lip_3}}\label{app:lem_2}
From Lemma~\ref{lem:hess_lip_2}, we have
\begin{align}
     \norm{\nabla f(\overline{x}_t)} &\le \norm{\int_0^t w_t(s)\nabla f(x_{s^-})dN_s}\nonumber\\
     &+  L_2\int_0^t \eta^2 \norm{z_s-x_s}^2 \bpar{\int_s^t \int_0^s    w_t(\sigma)w_t(\tau)(\tau-\sigma)dN_\sigma dN_\tau}ds \nonumber\\
     &+L_2\int_0^t \gamma^2 \norm{\nabla f(x_{s^{-}})}^2 \bpar{\int_s^t \int_0^s    w_t(\sigma)w_t(\tau)(N_{\tau^-}-N_{\sigma^-})dN_\sigma dN_\tau}dN_s \nonumber\\
     &= \norm{\frac{w_t(t)(z_t-x_t)}{\gamma' - \gamma}}+  L_2\int_0^t \eta^2 \norm{z_s-x_s}^2 \bpar{\int_s^t \int_0^s    w_t(\sigma)w_t(\tau)(\tau-\sigma)dN_\sigma dN_\tau}ds \nonumber\\
     &+L_2\int_0^t \gamma^2 \norm{\nabla f(x_{s^{-}})}^2 \bpar{\int_s^t \int_0^s    w_t(\sigma)w_t(\tau)(N_{\tau^-}-N_{\sigma^-})dN_\sigma dN_\tau}dN_s, \nonumber
\end{align}
where the equality is because of Lemma~\ref{lem:transfer_grad_derivative}.
Multiplying both sides by $\Lambda_t := \bpar{\int_0^t \alphaval e^{\alphaval(s-t)}dN_s}^2$ and integrating on $[0,T_n]$ with respect to $dN_t$, we get 
    \begin{align*}
    &\int_0^{T_n}\Lambda_{t} \norm{\nabla f(\overline{x}_{t})}dN_t \le \int_0^{T_n}\Lambda_{t}\norm{\frac{w_{t}(t)(z_{t}-x_{t})}{\gamma' - \gamma}}dN_t \\
    &+L_2\int_0^{T_n}\Lambda_{t}\int_0^t \eta^2 \norm{z_s-x_s}^2 \bpar{\int_s^t \int_0^s    w_{t}(\sigma)w_{t}(\tau)(\tau-\sigma)dN_\sigma dN_\tau}ds dN_t\\
    &+L_2\int_0^{T_n}\Lambda_{t}\int_0^t \gamma^2 \norm{\nabla f(x_{s^-})}^2 \bpar{\int_s^t \int_0^s   w_{t}(\sigma)w_{t}(\tau)(N_{\tau^-}-N_{\sigma^-})dN_\sigma dN_\tau }dN_s dN_t\\
    &\overset{\text{(i)}}{=} \int_0^{T_n}\Lambda_{t}\norm{\frac{w_{t}(t)(z_{t}-x_{t})}{\gamma' - \gamma}}dN_{t} \\
    &+\alphaval^2 L_2\int_0^{T_n}\int_0^{t} \eta^2 \norm{z_s-x_s}^2 \bpar{\int_s^t \int_0^s    e^{\alphaval(\tau-t)}e^{\alphaval(\sigma-t)}(\tau-\sigma)dN_\sigma dN_\tau}ds dN_t\\
    &+\alphaval^2L_2\int_0^{T_n}\int_0^t \gamma^2 \norm{\nabla f(x_{s^-})}^2 \bpar{\int_s^t \int_0^s    e^{\alphaval(\tau-t)}e^{\alphaval(\sigma-t)}(N_{\tau^-}-N_{\sigma^-})dN_\sigma dN_\tau}dN_s dN_t
    \\
    &\overset{\text{(ii)}}{=} \int_0^{T_n}\Lambda_{t}\norm{\frac{w_{t}(t)(z_{t}-x_{t})}{\gamma' - \gamma}}dN_t \\
    &+\alphaval^2L_2\int_0^{T_n} \eta^2 \norm{z_s-x_s}^2 \bpar{\int_s^{T_n}\int_s^t \int_0^s    e^{\alphaval(\tau-t)}e^{\alphaval(\sigma-t)}(\tau-\sigma) dN_\sigma dN_\tau dN_t}ds\\
    &+\alphaval^2L_2\int_0^{T_n} \gamma^2 \norm{\nabla f(x_{s^-})}^2  \bpar{\int_s^{T_n}\int_s^t \int_0^s    e^{\alphaval(\tau-t)}e^{\alphaval(\sigma-t)}(N_{\tau^-}-N_{\sigma^-})dN_\sigma  dN_\tau dN_t}dN_s.
\end{align*}

 Equality (i) holds as $\Lambda_{t} = \bpar{\int_0^{t} \alphaval e^{\alphaval(s-t)}dN_s}^2 = \bpar{\int_0^{t} \alphaval e^{\alphaval s}dN_s}^2 e^{-2\alphaval t}$, such that 
 $$\Lambda_{t} w_{t}(\sigma)w_{t}(\tau)= \bpar{\int_0^{t} \alphaval e^{\alphaval s}dN_s}^2 e^{-2\alphaval t}\frac{\alphaval e^{\alphaval \sigma}}{\int_0^{t} \alphaval e^{\alphaval s}dN_s}\frac{\alphaval e^{\alphaval \tau}}{\int_0^{t} \alphaval e^{\alphaval s}dN_s}   =\alphaval^2 e^{-2\alphaval t} e^{\alphaval \sigma}e^{\alphaval \tau}.$$ Equality (ii) holds by Tonelli, enabling to swap the integration order of $s$ and $t$. We conclude by taking the expectation.

\subsection{Proof of Lemma~\ref{lem:term_1_control}}\label{app:term_1}

Recalling $w_t(t) = \frac{ \alphaval e^{\alphaval t}}{\int_0^t \alphaval e^{\alphaval s}dN_s}$, one has
\begin{align*}
    \norm{\frac{w_{t}(t)(z_{t}-x_{t})}{\gamma' - \gamma}}  &= \frac{\alphaval e^{\alphaval t}}{(\gamma' - \gamma)}\frac{1}{ \int_0^{t} \alphaval e^{\alphaval s}dN_s}\norm{z_{t} - x_{t}}\\
   &=\frac{\alphaval} {(\gamma' - \gamma)}\frac{1}{ \int_0^{t} \alphaval e^{\alphaval (s-t)}dN_s}\norm{z_{t} - x_{t}}
\end{align*}
So, using the notation $\Lambda_t := \bpar{\int_0^t \alphaval e^{\alphaval(s-t)}dN_s}^2 $, we have
\begin{equation}
    \begin{aligned}\label{eq:hess_bound_1_bis}
              \mathbb{E}\left[\int_0^{T_n}\Lambda_{t}\norm{\frac{w_t(t)(z_{t} - x_{t})}{\gamma' - \gamma}}dN_t \right]&=\mathbb{E}\left[\int_0^{T_n} \frac{\alphaval}{\gamma' - \gamma}\bpar{\int_0^{t} \alphaval e^{\alphaval(s-t)}dN_s}\norm{z_{t} - x_{t}}dN_t\right]\\
    &\le \frac{\alphaval}{\gamma' - \gamma}\sqrt{\mathbb{E}\left[ \int_0^{T_n} \Lambda_{t} dN_t\right]}\sqrt{\mathbb{E}\left[\int_0^{T_n}\norm{z_{t} - x_{t}}^2dN_t \right]},
    \end{aligned}
\end{equation}
where we used the Cauchy-Schwarz inequality. 
We have
\begin{equation*}
    \begin{aligned}
        \int_0^{T_n}\norm{z_{t} - x_{t}}^2dN_t  &= \int_0^{T_n}\norm{z_{t^-} - x_{t^-} - (\gamma'-\gamma)\nabla f(x_{t^-})}^2dN_t \\
        &\le 2\int_0^{T_n}\norm{z_{t^-} - x_{t^-}}^2dN_t + 2(\gamma'-\gamma)^2\int_0^{T_n}\norm{\nabla f(x_{t^-})}^2dN_t  
    \end{aligned}
\end{equation*}
Then, from \eqref{eq:lyap_control_2}, we have
\begin{equation}\label{eq:lem:hess_lip_3:3_bisbis}
\E{\int_0^{T_n}\norm{\nabla f(x_{t^-})}^2dN_t} \le \frac{4}{\gamma}\gapf,
\end{equation} and from \eqref{eq:lyap_control_3}, we have
\begin{equation}\label{eq:lem:hess_lip_3:3_bis}
     \E{\int_0^{T_n}  \norm{z_{t^-}-x_{t^-}}^2 dN_t} \le  \frac{1}{\eta + \eta'}\gapf =\frac{1}{\alphaval}\gapf  .
\end{equation}
Combining \eqref{eq:lem:hess_lip_3:3_bisbis} and \eqref{eq:lem:hess_lip_3:3_bis}, we obtain
\begin{equation}
    \begin{aligned}\label{eq:correct}
\sqrt{\mathbb{E}\left[\int_0^{T_n}\norm{z_{t} - x_{t}}^2dN_t \right]} &\le \sqrt{2\gapf} \sqrt{\frac{1}{\alphaval} + \frac{4(\gamma'-\gamma)^2}{\gamma}} \\
&\le  \sqrt{2\alphaval^{-1}\gapf} \sqrt{1+ \frac{4(\gamma'-\gamma)^2}{\gamma}},
    \end{aligned}
\end{equation}
the last inequality using $\alphaval^{-1}\ge 1$. Note that we assume $\gamma' = \gamma + \sqrt{\gamma/2}$, such that $\frac{(\gamma'-\gamma)^2}{\gamma} = \frac{1}{2}$, and $\sqrt{1+ \frac{4(\gamma'-\gamma)^2}{\gamma}} = \sqrt{1+2} = \sqrt{3}$.
Combining this with \eqref{eq:hess_bound_1_bis}, Lemma~\ref{prop:calc_sto_comput_1}-(i) and \eqref{eq:correct}, we deduce

\begin{align}
     \mathbb{E}\left[\int_0^{T_n}\Lambda_t\norm{\frac{w_t(t)(z_{t^-} - x_{t^-})}{\gamma' - \gamma}}dN_t \right] &\le \frac{\alphaval}{\gamma'-\gamma}\sqrt{3\bpar{1+\frac{3}{2}\alphaval}(1+2\alphaval)n}\sqrt{2\frac{\gapf}{\alphaval}}\nonumber\\
     &=\sqrt{2\alphaval n}\sqrt{\frac{6\bpar{1+\frac{3}{2}\alphaval}(1+2\alphaval)}{\gamma}}\sqrt{\gapf}\nonumber\\
     &=\sqrt{\alphaval n}\sqrt{\frac{A_n}{\gamma}}\sqrt{\gapf},\label{eq:lem:hess_lip_3:a}
\end{align}
where we used $\gamma' = \gamma + \sqrt{\frac{\gamma}{2}}$ and defined $A_n = 12\bpar{1+\frac{3}{2}\alphaval}(1+2\alphaval)$.

 \subsection{Proof of Lemma~\ref{lem:hess_lip_4}}\label{app:lem_3}
 We prove each statement separately.
 \paragraph{Equation \eqref{eq:exp_separate_1}}
 We have
\begin{align*}
    &\int_0^{T_n} \eta^2 \norm{z_s -x_s}^2  \bpar{\int_s^{T_n}\int_s^t \int_0^s   e^{\alphaval(\tau-t)}e^{\alphaval(\sigma-t)}(\tau-\sigma) dN_\sigma dN_\tau  dN_t}ds \\
&= \sum_{i = 0}^{n-1}\int_{(T_i,T_{i+1})} \eta^2 \norm{z_s -x_s}^2  \bpar{\int_s^{T_n}\int_s^t \int_0^s    e^{\alphaval(\tau-t)}e^{\alphaval(\sigma-t)}(\tau-\sigma) dN_\sigma dN_\tau dN_t}ds,
\end{align*}
where we used that integration with respect to $ds$ on a singleton is zero. 

Then, noting $U_i := \int_{(T_i,T_{i+1})} \eta^2 \norm{z_s -x_s}^2  ds$, we have
\begin{align}
   & \sum_{i = 0}^{n-1}\int_{(T_i,T_{i+1})} \eta^2 \norm{z_s -x_s}^2  \bpar{\int_s^{T_n}\int_s^t \int_0^s    e^{\alphaval(\tau-t)}e^{\alphaval(\sigma-t)}(\tau-\sigma) dN_\sigma dN_\tau dN_t}ds \nonumber\\
   &\overset{(i)}{=}\sum_{i = 0}^{n-1}\left(\int_{(T_i,T_{i+1})} \eta^2 \norm{z_s -x_s}^2  ds\right)\bpar{\int_{T_{i+1}}^{T_n} \int_{T_{i+1}}^t  \int_0^{T_i} e^{\alphaval(\tau-t)}e^{\alphaval(\sigma-t)}(\tau-\sigma)dN_\sigma dN_\tau dN_t} \nonumber\\
      &\overset{(ii)}{=}\sum_{i = 1}^{n-1}\left(\int_{(T_i,T_{i+1})} \eta^2 \norm{z_s -x_s}^2  ds\right)\bpar{\int_{T_{i+1}}^{T_n} \int_{T_{i+1}}^t  \int_0^{T_i} e^{\alphaval(\tau-t)}e^{\alphaval(\sigma-t)}(\tau-\sigma)dN_\sigma dN_\tau dN_t} \nonumber\\
      &=\sum_{i = 1}^{n-1}U_i  \bpar{\int_{T_{i+1}}^{T_n} \int_{T_{i+1}}^t  \int_0^{T_i} e^{\alphaval(\tau-t)}e^{\alphaval(\sigma-t)}(\tau-\sigma)dN_\sigma dN_\tau dN_t} \nonumber\\
   &= \sum_{i = 1}^{n-1}\bpar{\int_{T_{i+1}}^{T_n} \int_{T_{i+1}}^t \int_0^{T_i}  e^{\alphaval(\tau-t)}e^{\alphaval(\sigma-t)}(\tau-\sigma) U_i dN_\sigma dN_\tau dN_t} ,\label{eq:exp_separate_1:1}
\end{align}
where (i) uses the fact that $\int_s^{T_n}\int_s^t \int_0^s    e^{\alphaval(\tau-t)}e^{\alphaval(\sigma-t)}(\tau-\sigma) dN_\sigma dN_\tau dN_t$ is constant for $s \in (T_i,T_{i+1})$, as no new jump occurs on this interval. (ii) uses that $\int_0^{T_0} \cdot ~ dN_t = \int_0^0\cdot~ dN_t = 0$, as almost surely no jumps occur at $t = 0$.
Because we integrate over $[0,T_n]$ with Poisson integrals, we can write \eqref{eq:exp_separate_1:1} as a sum
\begin{equation}\label{eq:exp_separate_1:2}
    \eqref{eq:exp_separate_1:1} = \sum_{i = 1}^{n-1} \sum_{l=i+1}^n \sum_{k=i+1}^l \sum_{j=1}^ie^{\alphaval(T_k-T_l)}  e^{\alphaval (T_j-T_l)}(T_k-T_j)U_i.
\end{equation}
We multiply each side of \eqref{eq:exp_separate_1:2} by $\1_{\Abar}$, and take the expectation
\begin{equation*}
    \begin{aligned}
        \E{ \1_{\Abar} \cdot \eqref{eq:exp_separate_1:2}} &=  \E{ \1_{\Abar}   \sum_{i = 1}^{n-1} \sum_{l=i+1}^n \sum_{k=i+1}^l \sum_{j=1}^i e^{\alphaval(T_k-T_l)}  e^{\alphaval (T_j-T_l)}(T_k-T_j)U_i}\\
        &=\sum_{i = 1}^{n-1} \sum_{l=i+1}^n \sum_{k=i+1}^l  \E{ \1_{\Abar} e^{\alphaval(T_k-T_l)} e^{\alphaval(T_{i+1}-T_l)}\sum_{j=1}^ie^{\alphaval (T_j-T_{i+1})}(T_k-T_j)U_i}\\
        &\overset{(i)}\le \sum_{i = 1}^{n-1} \sum_{l=i+1}^n \sum_{k=i+1}^l  \E{ \1_{\Abar} e^{\alphaval(T_k-T_l)} e^{\alphaval(T_{i+1}-T_l)}\sum_{j=1}^{i+1}e^{\alphaval (T_j-T_{i+1})}(T_k-T_j)U_i}\\
        &= \sum_{i = 1}^{n-1} \sum_{l=i+1}^n \sum_{k=i+1}^l  \E{e^{\alphaval(T_k-T_l)}e^{\alphaval(T_{i+1}-T_l)} U_i \1_{\Abar} \sum_{j=1}^{i+1}e^{\alphaval (T_j-T_{i+1})}(T_k-T_{i+1})}\\
        &+ \sum_{i = \Kval }^{n-1} \sum_{l=i+1}^n \sum_{k=i+1}^l  \E{ e^{\alphaval(T_k-T_l)}e^{\alphaval(T_{i+1}-T_l)} U_i \1_{\Abar} \sum_{j=1}^{i+1}e^{\alphaval (T_j-T_{i+1})}(T_{i+1}-T_j)}\\
        &+\sum_{i = 1}^{\Kval -1} \sum_{l=i+1}^n \sum_{k=i+1}^l  \E{ e^{\alphaval(T_k-T_l)}e^{\alphaval(T_{i+1}-T_l)} U_i \sum_{j=1}^{i+1}e^{\alphaval (T_j-T_{i+1})}(T_{i+1}-T_j)}.
    \end{aligned}
\end{equation*}
(i) uses that $\sum_{j=1}^i a_j \le \sum_{j=1}^{i+1}a_j$, as long as $a_j \ge0$ for any $j\in \{1,\dots,i+1 \}$. The last line uses $\1_{\Abar} \le 1$.
Now, we use Proposition~\ref{prop:maj_H_0} to deal with the third term, and Theorem~\ref{thm:high_proba_bound} to deal with the first and second term, such that
\begin{equation}
    \begin{aligned}\label{eq:exp_separate_1:4_pre}
         \E{ \1_{\Abar} \cdot \eqref{eq:exp_separate_1:2}} & \le  \sum_{i = 1}^{n-1} \sum_{l=i+1}^n \sum_{k=i+1}^l  \E{ e^{\alphaval(T_{i+1}-T_l)}e^{\alphaval(T_k-T_l)} (T_k-T_{i+1})U_i \Ccont\E{\sum_{j=1}^{i+1}e^{\alphaval (T_j-T_{i+1})}}}\\
        &+\sum_{i = \Kval }^{n-1} \sum_{l=i+1}^n \sum_{k=i+1}^l  \E{ e^{\alphaval(T_{i+1}-T_l)}e^{\alphaval(T_k-T_l)}U_i \Ccont\E{\sum_{j=1}^{i+1}e^{\alphaval (T_j-T_{i+1})}(T_{i+1}-T_j)}}\\
        &+\sum_{i = 1}^{\Kval -1} \sum_{l=i+1}^n \sum_{k=i+1}^l  \E{ e^{\alphaval(T_k-T_l)}e^{\alphaval(T_{i+1}-T_l)} U_i \cupp \Kval ^2}\\
        &\le \Ccont \sum_{i = 1}^{n-1} \sum_{l=i+1}^n \sum_{k=i+1}^l  \E{ e^{\alphaval(T_{i+1}-T_l)}e^{\alphaval(T_k-T_l)} (T_k-T_{i+1})U_i }\E{\sum_{j=1}^{i+1}e^{\alphaval (T_j-T_{i+1})}}\\
        &+ \Ccont\sum_{i = \Kval }^{n-1} \sum_{l=i+1}^n \sum_{k=i+1}^l  \E{ e^{\alphaval(T_{i+1}-T_l)}e^{\alphaval(T_k-T_l)}U_i }\E{\sum_{j=1}^{i+1}e^{\alphaval (T_j-T_{i+1})}(T_{i+1}-T_j)}\\
        &+\sum_{i = 1}^{\Kval -1} \sum_{l=i+1}^n \sum_{k=i+1}^l  \E{ e^{\alphaval(T_k-T_l)}e^{\alphaval(T_{i+1}-T_l)} U_i }\cupp \Kval ^2\\
        &\overset{(i)}{\le} \Ccont \sum_{i = 1}^{n-1} \sum_{l=i+1}^n \sum_{k=i+1}^l  \E{ e^{\alphaval(T_{i+1}-T_l)}e^{\alphaval(T_k-T_l)} (T_k-T_{i+1})U_i }\E{\sum_{j=1}^{i+1}e^{\alphaval (T_j-T_{i+1})}}\\
        &+ \Ccont\sum_{i = 1}^{n-1} \sum_{l=i+1}^n \sum_{k=i+1}^l  \E{ e^{\alphaval(T_{i+1}-T_l)}e^{\alphaval(T_k-T_l)}U_i }\E{\sum_{j=1}^{i+1}e^{\alphaval (T_j-T_{i+1})}(T_{i+1}-T_j)}.    \end{aligned}
\end{equation}
 (i) uses that for $i\ge \Kval $ and by Lemma~\ref{lem:exp_low_bound_H_0}, we have
\[
\E{\sum_{j=1}^{i+1}e^{\alphaval (T_j-T_{i+1})}(T_{i+1}-T_j)} 
\ge
\frac{e^{-1}}{32}
\min\!\left\{(i+1)^2,\frac1{\alphaval^2} \right\} \ge \frac{e^{-1}}{32}
\min\!\left\{ \Kval ^2,\frac1{\alphaval^2} \right\}  \ge \frac{\cinf^2e^{-1}}{32}\Kval ^2,
\]
the last inequality using the assumption $\alphaval^{-1} \ge \cinf \Kval $. So, we deduce 
\[\E{\sum_{j=1}^{i+1}e^{\alphaval (T_j-T_{i+1})}(T_{i+1}-T_j)} \ge \frac{\cinf^2e^{-1}}{32\cupp } \cupp \Kval ^2 \ge \Ccont^{-1} \cupp \Kval ^2,\]
which implies
\[\cupp \Kval ^2 \le \Ccont \E{\sum_{j=1}^{i+1}e^{\alphaval (T_j-T_{i+1})}(T_{i+1}-T_j)} .\]
Then, $U_i$ depends on $\tkvar{i}$, as it depends on $x_s$ and $z_s$ for $s\in(T_i,T_{i+1})$, which depend on $\tkvar{i}$. Also, $U_i$ depends on the length of the interval $(T_i,T_{i+1})$, such that $U_i$ also depends on $T_{i+1}$. So, we have that $U_i$ only depends on $\tkvar{i+1}$. Because the increments of $\{ T_k\}_{k\in \N}$ are independent, because $k,l \ge i+1$, we have $T_{i+1}-T_l \ind U_i$ and $T_k-T_l \ind  U_i$. So, \eqref{eq:exp_separate_1:4_pre} becomes 

\begin{equation}
    \begin{aligned}\label{eq:exp_separate_1:4}
         \E{\1_{\Abar} \cdot \eqref{eq:exp_separate_1:2}}
        &\le \Ccont \sum_{i = 1}^{n-1} \sum_{l=i+1}^n \sum_{k=i+1}^l  \E{ e^{\alphaval(T_{i+1}-T_l)}e^{\alphaval(T_k-T_l)} (T_k-T_{i+1})}\E{U_i }\E{\sum_{j=1}^{i+1}e^{\alphaval (T_j-T_{i+1})}}\\
        &+ \Ccont\sum_{i = \Kval }^{n-1} \sum_{l=i+1}^n \sum_{k=i+1}^l  \E{ e^{\alphaval(T_{i+1}-T_l)}e^{\alphaval(T_k-T_l)}}\E{U_i }\E{\sum_{j=1}^{i+1}e^{\alphaval (T_j-T_{i+1})}(T_{i+1}-T_j)}\\
        &= \Ccont \sum_{i = 1}^{n-1} \E{U_i } \sum_{l=i+1}^n \sum_{k=i+1}^l  \E{ e^{\alphaval(T_{i+1}-T_l)}e^{\alphaval(T_k-T_l)} (T_k-T_{i+1})}\E{\sum_{j=1}^{i+1}e^{\alphaval (T_j-T_{i+1})}}\\
        &+ \Ccont\sum_{i = \Kval }^{n-1} \E{U_i } \sum_{l=i+1}^n \sum_{k=i+1}^l  \E{ e^{\alphaval(T_{i+1}-T_l)}e^{\alphaval(T_k-T_l)}}\E{\sum_{j=1}^{i+1}e^{\alphaval (T_j-T_{i+1})}(T_{i+1}-T_j)}\\
        &\overset{(i)}{=} \Ccont \sum_{i = 1}^{n-1} \E{U_i } \sum_{l=i+1}^n \sum_{k=i+1}^l  \E{ e^{\alphaval(T_{i+1}-T_l)}e^{\alphaval(T_k-T_l)} (T_k-T_{i+1})\sum_{j=1}^{i+1}e^{\alphaval (T_j-T_{i+1})}}\\
        &+ \Ccont\sum_{i = 1}^{n-1} \E{U_i } \sum_{l=i+1}^n \sum_{k=i+1}^l  \E{ e^{\alphaval(T_{i+1}-T_l)}e^{\alphaval(T_k-T_l)}\sum_{j=1}^{i+1}e^{\alphaval (T_j-T_{i+1})}(T_{i+1}-T_j)}\\
        &=\Ccont\sum_{i = 1 }^{n-1} \E{U_i } \E{\sum_{l=i+1}^n \sum_{k=i+1}^l   e^{\alphaval(T_{i+1}-T_l)}e^{\alphaval(T_k-T_l)}\sum_{j=1}^{i+1}e^{\alphaval (T_j-T_{i+1})}(T_k-T_j)}\\
    \end{aligned}
\end{equation}
where in (i) we used the independence of the increments of the $\{ T_k\}_{k\in \N}$.
Finally, writing the sums as Poisson integrals, noting 
\begin{align*}
    A_{i+1} &:= \sum_{l=i+1}^n \sum_{k=i+1}^l   e^{\alphaval(T_{i+1}-T_l)}e^{\alphaval(T_k-T_l)}\sum_{j=1}^{i+1}e^{\alphaval (T_j-T_{i+1})}(T_k-T_j) \\
    &= \int_{[T_{i+1},T_n]}\int_{[T_{i+1}, t]} e^{\alphaval(T_{i+1}-t)}e^{\alphaval(\tau-t)} \int_{[0,T_{i+1}]}e^{\alphaval(\sigma - T_{i+1})}(\tau - \sigma)dN_\sigma dN_\tau dN_t ,
\end{align*} and recalling  $U_i := \int_{(T_i,T_{i+1})} \eta^2 \norm{z_s -x_s}^2  ds$, \eqref{eq:exp_separate_1:4} becomes

\begin{equation*}
    \begin{aligned}
         \E{\1_{\Abar} \cdot \eqref{eq:exp_separate_1:2}}
       &\le \Ccont \sum_{i = 1}^{n-1} \E{\int_{(T_i,T_{i+1})} \eta^2 \norm{z_s -x_s}^2  ds } \E{A_{i+1}}\\
       &\le \Ccont\max_{i \in \{1,\dots,n-1 \}} \E{A_{i+1}}\sum_{i = 1}^{n-1} \E{\int_{(T_i,T_{i+1})} \eta^2 \norm{z_s -x_s}^2  ds }\\
       &\le \Ccont\max_{i \in \{1,\dots,n \}} \E{A_i}\E{\int_0^{T_n} \eta^2 \norm{z_s -x_s}^2  ds}.
    \end{aligned}
\end{equation*}

\paragraph{Equation \eqref{eq:exp_separate_2}}
\begin{equation}
\begin{aligned}\label{eq:exp_separate_2:1}
    &\int_0^{T_n} \gamma^2 \norm{\nabla f(x_{s^-})}^2  \bpar{\int_s^{T_n}\int_s^t \int_0^s    e^{\alphaval(\tau-t)}e^{\alphaval(\sigma-t)}(N_{\tau^-}-N_{\sigma^-})dN_\sigma  dN_\tau dN_t}dN_s \\
        &= \int_0^{T_n}  \bpar{\int_s^{T_n}\int_s^t  e^{\alphaval(\tau-t)}e^{\alphaval(s-t)}\gamma^2 \norm{\nabla f(x_{s^-})}^2 \int_0^s  e^{\alphaval(\sigma-s)}(N_{\tau^-}-N_{\sigma^-}) dN_\sigma dN_\tau dN_t}dN_s
\end{aligned}
\end{equation}
Because we integrate over $[0,T_n]$, \eqref{eq:exp_separate_2:1} can be written as a sum
\begin{align}
(\ref{eq:exp_separate_2:1}) = \sum_{i=1}^n \sum_{l=i}^n \sum_{k=i}^l e^{\alphaval(T_k-T_l)}e^{\alphaval(T_i-T_l)}\gamma^2 \norm{\nabla f(x_{T_i^-})}^2\sum_{j=1}^i e^{\alphaval (T_j-T_i)}(k-j).
\end{align}
We used that $N_{T_k^-} - N_{T_j^-} = (k-1)-(j-1)$.
Multiplying by $\1_{\Abar}$ and taking expectation, we have
\begin{equation*}
    \begin{aligned}
        \E{\1_{\Abar} \cdot \eqref{eq:exp_separate_2:1}} &= \E{\1_{\Abar}\sum_{i=1}^n \sum_{l=i}^n \sum_{k=i}^l e^{\alphaval(T_k-T_l)}e^{\alphaval(T_i-T_l)}\gamma^2 \norm{\nabla f(x_{T_i^-})}^2\sum_{j=1}^i e^{\alphaval (T_j-T_i)}(k-j)}\\
        &= \sum_{i=1}^n \sum_{l=i}^n \sum_{k=i}^l \E{\1_{\Abar}e^{\alphaval(T_k-T_l)}e^{\alphaval(T_i-T_l)}\gamma^2 \norm{\nabla f(x_{T_i^-})}^2\sum_{j=1}^i e^{\alphaval (T_j-T_i)}(k-j)}\\
        &=\sum_{i=1}^n \sum_{l=i}^n \sum_{k=i}^l \E{e^{\alphaval(T_k-T_l)}e^{\alphaval(T_i-T_l)}\gamma^2 \norm{\nabla f(x_{T_i^-})}^2(k-i)\1_{\Abar}\sum_{j=1}^i e^{\alphaval (T_j-T_i)}}\\
        &+\sum_{i=1}^n \sum_{l=i}^n \sum_{k=i}^l \E{e^{\alphaval(T_k-T_l)}e^{\alphaval(T_i-T_l)}\gamma^2 \norm{\nabla f(x_{T_i^-})}^2\1_{\Abar}\sum_{j=1}^i e^{\alphaval (T_j-T_i)}(i-j)}
    \end{aligned}
\end{equation*}
Now, Theorem~\ref{thm:high_proba_bound} ensures that there exists a constant $\Ccont$, such that we have 
\begin{equation}
    \begin{aligned}\label{eq:exp_separate_2:5}
         \E{\1_{\Abar} \cdot \eqref{eq:exp_separate_2:1}}  &\le\sum_{i=1}^n \sum_{l=i}^n \sum_{k=i}^l \E{e^{\alphaval(T_k-T_l)}e^{\alphaval(T_i-T_l)}\gamma^2 \norm{\nabla f(x_{T_i^-})}^2 \Ccont(k-i)\E{\sum_{j=1}^i e^{\alphaval (T_j-T_i)}}}\\
         &+ \sum_{i=1}^n \sum_{l=i}^n \sum_{k=i}^l \E{e^{\alphaval(T_k-T_l)}e^{\alphaval(T_i-T_l)}\gamma^2 \norm{\nabla f(x_{T_i^-})}^2\Ccont \E{\sum_{j=1}^i e^{\alphaval (T_j-T_i)}(i-j)}}\\
         &\le \Ccont \sum_{i=1}^n \sum_{l=i}^n \sum_{k=i}^l \E{e^{\alphaval(T_k-T_l)}e^{\alphaval(T_i-T_l)}\gamma^2 \norm{\nabla f(x_{T_i^-})}^2 }(k-i)\E{\sum_{j=1}^i e^{\alphaval (T_j-T_i)}}\\
         &+  \Ccont\sum_{i=1}^n \sum_{l=i}^n \sum_{k=i}^l \E{e^{\alphaval(T_k-T_l)}e^{\alphaval(T_i-T_l)}\gamma^2 \norm{\nabla f(x_{T_i^-})}^2}\E{\sum_{j=1}^i e^{\alphaval (T_j-T_i)}(i-j)}.
    \end{aligned}
\end{equation}

Now, because the increments of $\{ T_k\}_{k\in \N}$ are independent, as $l,k \ge i$, we have $T_k-T_l \ind \norm{\nabla f(x_{T_i^-})}$ and $T_i-T_l \ind  \norm{\nabla f(x_{T_i^-})}$. So, \eqref{eq:exp_separate_2:5} becomes
\begin{equation}
    \begin{aligned}\label{eq:exp_separate_2:6}
         \E{\1_{\Abar} \cdot\eqref{eq:exp_separate_2:1}} &\le  \Ccont \sum_{i=1}^n \sum_{l=i}^n \sum_{k=i}^l \E{e^{\alphaval(T_k-T_l)}e^{\alphaval(T_i-T_l)}}\gamma^2\E{ \norm{\nabla f(x_{T_i^-})}^2 }(k-i)\E{\sum_{j=1}^i e^{\alphaval (T_j-T_i)}}\\
           &+\Ccont\sum_{i=1}^n \sum_{l=i}^n \sum_{k=i}^l \E{e^{\alphaval(T_k-T_l)}e^{\alphaval(T_i-T_l)}\gamma^2} \E{ \norm{\nabla f(x_{T_i^-})}^2}\E{\sum_{j=1}^i e^{\alphaval (T_j-T_i)}(i-j)}\\
           &=\Ccont \sum_{i=1}^n \gamma^2\E{ \norm{\nabla f(x_{T_i^-})}^2 } \sum_{l=i}^n \sum_{k=i}^l \E{e^{\alphaval(T_k-T_l)}e^{\alphaval(T_i-T_l)}}(k-i)\E{\sum_{j=1}^i e^{\alphaval (T_j-T_i)}}\\
           &+\Ccont\sum_{i=1}^n \gamma^2 \E{ \norm{\nabla f(x_{T_i^-})}^2}\sum_{l=i}^n \sum_{k=i}^l \E{e^{\alphaval(T_k-T_l)}e^{\alphaval(T_i-T_l)}}\E{\sum_{j=1}^i e^{\alphaval (T_j-T_i)}(i-j)}\\
           &\overset{(i)}=\Ccont\sum_{i=1}^n \gamma^2 \E{ \norm{\nabla f(x_{T_i^-})}^2} \E{\sum_{l=i}^n \sum_{k=i}^l e^{\alphaval(T_k-T_l)}e^{\alphaval(T_i-T_l)}\sum_{j=1}^i e^{\alphaval (T_j-T_i)}(k-j)},
    \end{aligned}
\end{equation}
where (i) used 
the independence property of the increments of $\tk$.
Rewriting \eqref{eq:exp_separate_2:6} as Poisson integrals, we note 
\begin{align*}
    B_i &:= \sum_{l=i}^n \sum_{k=i}^l e^{\alphaval(T_k-T_l)}e^{\alphaval(T_i-T_l)}\sum_{j=1}^i e^{\alphaval (T_j-T_i)}(k-j)\\
    &= \int_{T_i}^{T_n}\int_{T_i}^t     e^{\alphaval(\tau-t)}e^{\alphaval(T_i-t)}\int_0^{T_i}(N_{\tau^-}-N_{\sigma^-})e^{\alphaval(\sigma - T_i)}dN_\sigma  dN_\tau dN_t,
\end{align*}
 such that
\begin{equation*}
    \begin{aligned}
         \E{\1_{\Abar} \cdot\eqref{eq:exp_separate_2:1}} &\le \Ccont\sum_{i=1}^n \gamma^2 \E{ \norm{\nabla f(x_{T_i^-})}^2}\E{B_i}\\
            &\le \Ccont \max_{i\in \{1,\dots,n\}}\E{B_i}\sum_{i=1}^n\gamma^2 \E{ \norm{\nabla f(x_{T_i^-})}^2}\\
         &= \Ccont  \max_{i\in \{1,\dots,n\}}\E{B_i}\E{\int_0^{T_n}\gamma^2 \norm{\nabla f(x_{s^-})}^2dN_s}
    \end{aligned}
\end{equation*}

\subsection{Proof of Lemma~\ref{lem:hess_lip_5}}\label{app:lem_5}
We recall 
\[A_k := \int_{[T_k,T_n]}\int_{[T_k, t]}  \int_{[0,T_k]}e^{\alphaval(T_k-t)}e^{\alphaval(\tau-t)}e^{\alphaval(\sigma - T_k)}(\tau - \sigma)dN_\sigma dN_\tau dN_t, \]
\[ B_k := \int_{T_k}^{T_n}\int_{T_k}^t     \int_0^{T_k}e^{\alphaval(\tau-t)}e^{\alphaval(T_k-t)}e^{\alphaval(\sigma - T_k)}(N_{\tau^-}-N_{\sigma^-})dN_\sigma  dN_\tau dN_t.\]
We will split the proof in several sub lemmas. We first bound $\max_{k\in \{1,\dots,n\}}\E{B_k}$.
\begin{lemma}\label{lem:comput_sto_3}
We have
\begin{align*}
     &\max_{k \in \{1,\dots,n\}} \E{B_k} \le  \frac{(1+\alphaval)^2(1+2\alphaval)}{\alphaval^4}
\end{align*}
\end{lemma}
\begin{proof}

Let $k \le n$. 
\begin{align*}
B_k&= \int_{T_k}^{T_n}\int_{T_k}^t \int_0^{T_k}    e^{\alphaval(\tau-t)}e^{\alphaval(\sigma-t)}(N_{\tau^-}-N_{\sigma^-})dN_\sigma  dN_\tau dN_t\\
       &=\sum_{i=k}^n \sum_{j=1}^k\sum_{l=k}^ie^{\alphaval(T_l - T_i)}e^{\alphaval(T_j-T_i)}(l-j)\\
       &=\sum_{i=k}^n \sum_{j=1}^k\sum_{l=k}^i e^{2\alphaval(T_l-T_i)}e^{\alphaval(T_j-T_l)}(l-j).
\end{align*}

We take expectation, use $T_l-T_i \ind T_j-T_l$ and Lemma~\ref{lem:exp_stopping_time}
\begin{align}
   \E{B_k} = & \E{\sum_{i=k}^n\sum_{j=1}^k\sum_{l=k}^i e^{2\alphaval(T_l-T_i)}e^{\alphaval(T_j-T_l)}(l-j)}\\
    &= \sum_{i=k}^n \sum_{j=1}^k\sum_{l=k}^i
    \E{e^{2\alphaval(T_l-T_i)}}\E{e^{\alphaval(T_j-T_l)}}(l-j)\label{eq:divergence_bis} \\
    &=\sum_{i=k}^n \sum_{j=1}^k\sum_{l=k}^i
    (1+2\alphaval)^{l-i}(1+\alphaval)^{j-l}(l-j)\nonumber\\
    &=\sum_{\substack{i,j,l \text{ s.t.}\\ 1\le j\le k\le l \le i \le n}}
    (1+2\alphaval)^{l-i}(1+\alphaval)^{j-l}(l-j)\nonumber\\
    &=\sum_{\substack{j,l \text{ s.t.}\\ 1\le j\le k\le l \le n}}
   (1+\alphaval)^{j-l}(l-j)\sum_{i=l}^n (1+2\alphaval)^{l-i}\nonumber
\end{align}
We have $\sum_{i=l}^n (1+2\alphaval)^{l-i}\le \sum_{i=l}^{+\infty} (1+2\alphaval)^{l-i} = 1+\frac{1}{2\alphaval} $. We plug this inequality in the above, such that
    \begin{equation}
        \begin{aligned}
\E{B_k} \le \bpar{1+\frac{1}{2\alphaval}} \sum_{\substack{j,l \text{ s.t.}\\ 1\le j\le k\le l \le n}}
   (1+\alphaval)^{j-l}(l-j)
        \end{aligned}
    \end{equation}
Now, we do the change of variable $(j,l) \to (j,l-j):=(j,d)$, such that
    \begin{equation}
        \begin{aligned}
\E{B_k} \le \bpar{1+\frac{1}{2\alphaval}} \sum_{d=0}^{n-1}\sum_{j=\max{\{1, k-d \}}}^{\min \{k,n-d\}}
   (1+\alphaval)^{-d}d.
        \end{aligned}
    \end{equation}
Indeed, for a fixed $k$, we both have $1\le j\le k$, and $k\le l \le n \Rightarrow k-d\le l-d\overset{(i)}{=} j \le n-d$, where $(i)$ is due to the change of variable. Then, as $\sum_{j=\max{\{1, k-d \}}}^{\min \{k,n-d\}}
   (1+\alphaval)^{-d}d \le \sum_{j=k-d}^{k}
   (1+\alphaval)^{-d}d$, we deduce
\begin{equation}
    \begin{aligned}
        \E{B_k} &\le \bpar{1+\frac{1}{2\alphaval}} \sum_{d=0}^{n-1}\sum_{j=k-d}^{k}
   (1+\alphaval)^{-d}d\\
   &\le \bpar{1+\frac{1}{2\alphaval}} \sum_{d=0}^{n-1}
   (1+\alphaval)^{-d}d(d+1)\\
   &\le \bpar{1+\frac{1}{2\alphaval}} \sum_{d=0}^{+\infty}
   (1+\alphaval)^{-d}d(d+1)\\
   &= \bpar{1+\frac{1}{2\alphaval}}\frac{2}{\alphaval^3}(1+\alphaval)^2\\
   &= \frac{(1+\alphaval)^2(1+2\alphaval)}{\alphaval^4}.
    \end{aligned}
\end{equation}

\end{proof}
We now turn to $\max_{k\in \{1,\dots,n\}}\E{A_k}$.
The computations are similar, thanks to the following elementary result.
\begin{lemma}\label{lem:gamma_law_deriv}
    If $T_n \sim \Gamma(n,1)$, then
\begin{equation*}
    \E{T_n e^{-\alphaval T_n}} =  \frac{n}{(1+\alphaval)^{n+1}}
\end{equation*}
\end{lemma}

\begin{proof}
    Let $\varphi(\alphaval) = \E{e^{-\alphaval T_n}}$ the moment-generating function of $T_n$. As $T_n \sim \Gamma(n,1)$, we have $\phi(\alphaval) =(1+\alphaval)^{-n} \Rightarrow \phi'(\alphaval) =- n(1+\alphaval)^{-n-1}$. 
    So
\begin{equation*}
    \phi (\alphaval) = -\E{T_ne^{-\alphaval T_n}} \Rightarrow \E{T_ne^{-\alphaval T_n}} =  \frac{n}{(1+\alphaval)^{n+1}}.
\end{equation*}
\end{proof}

\begin{lemma}\label{lem:comput_sto_4}
We have
\begin{align*}
     &\max_{k \in \{1,\dots,n\}} \E{A_k} \le \frac{(1+\alphaval)(1+2\alphaval)}{\alphaval^4}
\end{align*}
\end{lemma}
\begin{proof}
The proof is almost identical to the one of Lemma~\ref{lem:comput_sto_3}. The only difference lies in \eqref{eq:divergence_bis}, where $l-j$ is replaced by $T_l - T_j$.  
    Then, a direct application of Lemma~\ref{lem:gamma_law_deriv} gives 
$\E{T_n e^{-\alphaval T_n}}  = (1+\alphaval)^{-1}\E{ne^{-\alphaval T_n}}$ for some $T_n\sim \Gamma(n,1)$.
\end{proof}

\subsection{Proof of Lemma~\ref{lem:term_2_3_control}}\label{app:term_2_3}

\paragraph{1. Bounding Term 2.}
Using Lemma~\ref{lem:hess_lip_4} and Lemma~\ref{lem:hess_lip_5}, we have
\begin{equation}
    \begin{aligned}\label{eq:lem:hess_lip_3:2}
        &\E{\1_{\Abar} \alphaval^2L_2\int_0^{T_n} \eta^2 \norm{z_s-x_s}^2 \bpar{\int_s^{T_n}\int_s^t \int_0^s    e^{\alphaval(\tau-t)}e^{\alphaval(\sigma-t)}(\tau-\sigma)dN_\sigma dN_\tau  dN_t}ds} \\
        &\le \Ccont L_2 \frac{(1+\alphaval)(1+2\alphaval)}{\alphaval^2}  \E{\int_0^{T_n} \eta^2 \norm{z_s-x_s}^2 ds}.
    \end{aligned}
\end{equation}
From \eqref{eq:lyap_control_1}, we have
\begin{equation}\label{eq:lem:hess_lip_3:3}
     \E{\int_0^{T_n}  \norm{z_s-x_s}^2 ds}  \le  \frac{1}{\eta + \eta'}\gapf =\frac{1}{\alphaval}\gapf  .
\end{equation}
So, combining \eqref{eq:lem:hess_lip_3:2} and \eqref{eq:lem:hess_lip_3:3}, we deduce 
\begin{equation}
    \begin{aligned}\label{eq:lem:hess_lip_3:b}
        &\E{\1_{\Abar} \alphaval^2L_2\int_0^{T_n} \eta^2 \norm{z_s-x_s}^2 \bpar{\int_s^{T_n}\int_s^t \int_0^s    e^{\alphaval(\tau-t)}e^{\alphaval(\sigma-t)}(\tau-\sigma)dN_\sigma dN_\tau  dN_t}ds} \\
        &\le \Ccont L_2 \eta^2\frac{(1+\alphaval)(1+2\alphaval)}{\alphaval^3} \gapf\\
        &=\Ccont L_2 \gamma\frac{(1+\alphaval)(1+2\alphaval)}{2\alphaval^3} \gapf,
    \end{aligned}
\end{equation}
where the last equality uses $\eta = \sqrt{\frac{\gamma}{2}}$.
\paragraph{2. Bounding Term 3.}
Using Lemma~\ref{lem:hess_lip_4} and Lemma~\ref{lem:hess_lip_5}, we have
\begin{equation}
     \begin{aligned}\label{eq:lem:hess_lip_3:4}
          &\E{\1_{\Abar}\alphaval^2L_2\int_0^{T_n} \gamma^2 \norm{\nabla f(x_{s^-})}^2  \bpar{\int_s^{T_n}\int_s^t \int_0^s    e^{\alphaval(\tau-t)}e^{\alphaval(\sigma-t)}(N_{\tau^-}-N_{\sigma^-})dN_\sigma  dN_\tau dN_t}dN_s} \\
        &\le \Ccont L_2\frac{(1+\alphaval)^2(1+2\alphaval)}{\alphaval^2}\E{\int_0^{T_n} \gamma^2 \norm{\nabla f(x_{s^-})}^2  dN_s}
     \end{aligned}
\end{equation}
From \eqref{eq:lyap_control_2}, we have
\begin{equation}\label{eq:lem:hess_lip_3:5}
    \mathbb{E}\left[\int_0^{T_n} \norm{\nabla f(x_{s^{-}})}^2dN_s\right] \le \frac{4}{\gamma }\gapf.
\end{equation}
So, combining \eqref{eq:lem:hess_lip_3:4} and \eqref{eq:lem:hess_lip_3:5}, we deduce 
\begin{equation}
     \begin{aligned}\label{eq:lem:hess_lip_3:c}
          &\E{\1_{\Abar}\alphaval^2L_2\int_0^{T_n} \gamma^2 \norm{\nabla f(x_{s^-})}^2  \bpar{\int_s^{T_n}\int_s^t \int_0^s    e^{\alphaval(\tau-t)}e^{\alphaval(\sigma-t)}(N_{\tau^-}-N_{\sigma^-})dN_\sigma  dN_\tau dN_t}dN_s} \\
        &\le 4\Ccont L_2 \gamma \frac{(1+\alphaval)^2(1+2\alphaval)}{\alphaval^2}\gapf.
     \end{aligned}
\end{equation}
Combining \eqref{eq:lem:hess_lip_3:b} and \eqref{eq:lem:hess_lip_3:c}, we obtain the desired result.
\subsection{Proofs of Section~\ref{sec:part_5}}\label{app:part_5}
    
\paragraph{Proof of Corollary~\ref{cor:choice1}}
    From \eqref{eq:almost_final_bis}
                   \begin{equation}
           \begin{aligned}\label{eq:corr_1}
            & \E{\1_{\Abar}\min_{t \in \{T_1,\dots,T_n \}} \norm{\nabla f(\overline{x}_t)} \int_0^{T_n}\bpar{\int_0^t \alphaval e^{\alphaval(s-t)}dN_s}^2 dN_t}\\
                  &=n^{3/7}\bpar{\sqrt{A_n L\gapf} + B_n \Ccont L^{-1}L_2 \gapf}\\
                  &\le 4\sqrt{3} B_n \Ccont\bpar{\sqrt{ L\gapf} + L^{-1}L_2 \gapf} n^{3/7}
    \end{aligned}
       \end{equation}
where we use $\Ccont \ge 1$ and $\sqrt{A_n} = 2\sqrt{3}\sqrt{(1+\frac{3}{2}\alphaval)(1+2\alphaval)} \le2\sqrt{3}(1+2\alphaval) \le 2\sqrt{3}(1+2\alphaval)(1+\alphaval)\bpar{1 + \alphaval8(1+\alphaval)}=4\sqrt{3}B_n$.  Using $\alphaval = n^{-1/7}$, we have
\[ B_n \Ccont= 2\sqrt{3}(1+2n^{-1/7})(1+n^{-1/7})\bpar{1 + n^{-1/7}8(1+n^{-1/7})} 32(4+4C_n+C_n^2)\frac{\cupp e}{\cinf^2},\]
with $C_n =  \cinf n^{-1/7} \Kval $.

    \paragraph{Proof of Corollary~\ref{cor:choice2} }

From \eqref{eq:almost_final_bis}
           \begin{equation}
           \begin{aligned}\label{eq:corr_2:1}
            & \E{\1_{\Abar}\min_{t \in \{T_1,\dots,T_n \}} \norm{\nabla f(\overline{x}_t)} \int_0^{T_n}\bpar{\int_0^t \alphaval e^{\alphaval(s-t)}dN_s}^2 dN_t}\\
                  &\le n^{3/7}\bpar{\sqrt{A_n\constalpha L^{1+a_1}L_2^{a_2}\gapf^{1+a_3}} + B_n\constalpha^{-3} \Ccont L^{-1-3a_1}L_2^{1-3a_2} \gapf^{1-3a_3}},
    \end{aligned}
       \end{equation}
        with, $A_n := 12\bpar{1+\frac{3}{2}\alphaval}(1+2\alphaval)$ and $B_n := \frac{1}{2}(1+\alphaval)(1+2\alphaval)\bpar{1 + \alphaval8(1+\alphaval)}$. 
    We want to find $a_1,a_2,a_3$ such that the exponents in the left and right term are the same.  Therefore, we solve
\[
\left\{
\begin{aligned}
\frac{1+a_1}{2} &= -1-3a_1,\\
\frac{a_2}{2} &= 1-3a_2,\\
\frac{1+a_3}{2} &= 1-3a_3,
\end{aligned}
\right.
\qquad \Longrightarrow \qquad
(a_1,a_2,a_3)=\left(-\frac{3}{7},\frac{2}{7},\frac{1}{7}\right).
\]
Then, \eqref{eq:corr_2:1} becomes
                  \begin{equation}
           \begin{aligned}\label{eq:corr_2:2}
            & \E{\1_{\Abar}\min_{t \in \{T_1,\dots,T_n \}} \norm{\nabla f(\overline{x}_t)} \int_0^{T_n}\bpar{\int_0^t \alphaval e^{\alphaval(s-t)}dN_s}^2 dN_t}\\
                  &\le n^{3/7}\bpar{\sqrt{ A_n\bpar{1+\frac{3}{2}}\constalpha } +B_n\constalpha^{-3} \Ccont}L^{\frac{2}{7}}L_2^{\frac{1}{7}} \gapf^{\frac{4}{7}}\\
                  &\overset{(i)}= \bpar{\sqrt{ A_n\constalpha } +B_n\constalpha^{-3} \Big(32\frac{\cupp e}{\cinf^2}(4 + 4C_n + C_n^2)\Big)}L^{\frac{2}{7}}L_2^{\frac{1}{7}} \gapf^{\frac{4}{7}}n^{3/7}\\
                  &\overset{(ii)}\le (1 + C_n + C_n^2/4)\bpar{\sqrt{ A_n\constalpha } + 128B_n\frac{\cupp e}{\cinf^2}\constalpha^{-3}}L^{\frac{2}{7}}L_2^{\frac{1}{7}} \gapf^{\frac{4}{7}}n^{3/7}\\
                  & \overset{(iii)}\le (1 + C_n + C_n^2/4)\Hconst\bpar{\sqrt{ 12\constalpha } + 64\frac{\cupp e}{\cinf^2}\constalpha^{-3}}L^{\frac{2}{7}}L_2^{\frac{1}{7}} \gapf^{\frac{4}{7}}n^{3/7}.
    \end{aligned}
    \end{equation}
    In $(i)$ we use the definition of $\Ccont$, namely $\Ccont = 32(4+4C_n+C_n^2)\frac{\cupp e}{\cinf^2}$. In $(ii)$ we use $(4 + 4C_n + C_n^2)=4(1+C_n+C_n^2/4)$ and $(1+C_n+C_n^2/4) \ge 1$.  In $(iii)$, replace $A_n$ and $B_n$ by their values and we set $\Hconst := \max \{ \sqrt{(1+\frac{3}{2}\alphaval)(1+2\alphaval)}, (1+\alphaval)(1+2\alphaval)(1+\alphaval 8 (1+\alphaval)) \} = (1+\alphaval)(1+2\alphaval)(1+\alphaval 8 (1+\alphaval))$.
    We then optimize on $\constalpha$. Note that we keep implicit the dependence on $\constalpha$ in $\Hconst$, and do not consider it in the optimization of $\constalpha$. Denoting $a =12,$ $b =64\frac{\cupp e}{\cinf^2}$, we want to minimize
    \[\sqrt{a\constalpha} + b\constalpha^{-3}.\]
    The function $\phi: t \mapsto \sqrt{a t} + bt^{-3}$ is derivable on $\R_+^\ast$, and $\phi'(t) = \frac{\sqrt{a}}{2\sqrt{t}} - \frac{3b}{t^4} $. It follows that $\phi$ reaches its minimum on $\R_+^\ast$ at the point
    \[\phi'(t) = 0 \Leftrightarrow t =\bpar{\frac{6b}{\sqrt{a}}}^{2/7}.\]
    Then, the minimum of $\phi$, denoted $\phi_{\min}$, is 
    \begin{align*}
\phi_{\min}
&= \phi\bpar{\bpar{\frac{6b}{\sqrt{a}}}^{2/7}} \\
&= \sqrt{a\left(\frac{6b}{\sqrt{a}}\right)^{2/7}} + \frac{b}{\left(\frac{6b}{\sqrt{a}}\right)^{6/7}} \\
&= a^{1/2}\cdot 6^{1/7} b^{1/7} a^{-1/14}
+ b \cdot a^{3/7} 6^{-6/7} b^{-6/7} \\
&= 6^{1/7} a^{3/7} b^{1/7}
+ 6^{-6/7} a^{3/7} b^{1/7} \\
&= \left(6^{1/7} + 6^{-6/7}\right)a^{3/7} b^{1/7} \\
&= 7\,6^{-6/7}\,a^{3/7} b^{1/7}.\\
\end{align*}
Then, choosing $\constalpha := \bpar{\frac{6b}{\sqrt{a}}}^{2/7}$, and replacing
 $a$ and $b$ by their values, we obtain 
\[\constalpha = \left(\frac{6b}{\sqrt{a}}\right)^{2/7}
= \left(\frac{6 \cdot 64\frac{\cupp e}{\cinf^2}}{\sqrt{12}}\right)^{2/7}
= \left(\frac{\sqrt{3}\cdot 64\cupp e}{\cinf^2}\right)^{2/7},\]
which in turn implies
\[\sqrt{ A_n\constalpha } +64\frac{\cupp e}{\cinf^2}\constalpha^{-3} =7\,6^{-6/7}\,12^{3/7}
64^{1/7}\bpar{\frac{\cupp e}{\cinf^2}}^{1/7}. \]
Finally, we deduce that with the choice
\[\alphaval = \left(\frac{\sqrt{3}\cdot 64\cupp e}{\cinf^2}\right)^{2/7}\bpar{\frac{L_2^2\gapf}{L^3 n}}^{1/7},\]
we obtain
      \begin{equation}
           \begin{aligned}\label{eq:corr_2:3}
            & \E{\min_{t \in \{T_1,\dots,T_n \}} \norm{\nabla f(\overline{x}_t)} \int_0^{T_n}\bpar{\int_0^t \alphaval e^{\alphaval(s-t)}dN_s}^2 dN_t}\\
                  &\le (1 + C_n + C_n^2/4)\Hconst 7\,6^{-6/7}\,12^{3/7}
64^{1/7}\bpar{\frac{\cupp e}{\cinf^2}}^{1/7}\bpar{\frac{\cupp e}{\cinf^2}}^{1/7}L^{\frac{2}{7}}L_2^{\frac{1}{7}} \gapf^{\frac{4}{7}} n^{3/7}\\
&\approx 9.2(1 + C_n + C_n^2/4)\Hconst\bpar{\frac{\cupp}{\cinf^2}}^{1/7}L^{\frac{2}{7}}L_2^{\frac{1}{7}} \gapf^{\frac{4}{7}} n^{3/7}.
    \end{aligned}
    \end{equation}

\subsection{Proof of Lemma~\ref{lem:lower_bound_delta}}\label{app:delta}


We note
\[
\Delta_n
:=
\int_0^{T_n}\left(\int_0^t \alphaval e^{\alphaval(s-t)}\,dN_s\right)^2 dN_t
=
\sum_{i=1}^n \left(\sum_{j=1}^i \alphaval e^{-\alphaval(T_i-T_j)}\right)^2.
\]

By Proposition~\ref{prop:min_gamma}, on the set $\Abar$, we have that for all
\(j\le i\le n\),
\[
T_i-T_j \le \cupp(i-j)+\cupp \Kval .
\]
Hence,
\[\1_{\Abar}
\Delta_n
\ge
\1_{\Abar}\alphaval^2 \sum_{i=1}^n
\left(\sum_{j=1}^i e^{-\cupp\alphaval(i-j+\Kval )}\right)^2
=
\1_{\Abar}\alphaval^2 e^{-2\cupp\alphaval \Kval }
\sum_{i=1}^n
\left(\sum_{j=1}^i e^{-\cupp\alphaval(i-j)}\right)^2.
\]

Set \(r:=e^{-\cupp\alphaval}\). Then
\[
\sum_{j=1}^i e^{-\cupp\alphaval(i-j)}
=
\sum_{\ell=0}^{i-1} r^\ell
=
\frac{1-r^i}{1-r}.
\]
Therefore,
\[
\1_{\Abar}\Delta_n
\ge
\1_{\Abar}\frac{\alphaval^2}{(1-r)^2}e^{-2\cupp\alphaval \Kval }
\sum_{i=1}^n (1-r^i)^2.
\]

Using \(1-e^{-x}\le x\) for all \(x\ge 0\), we get
\[
\frac{\alphaval^2}{(1-r)^2}
=
\frac{\alphaval^2}{(1-e^{-\cupp\alphaval})^2}
\ge
\frac{\alphaval^2}{(\cupp\alphaval)^2}
=
\frac{1}{\cupp^2}.
\]
Moreover,
\[
e^{-2\cupp\alphaval \Kval }
=
e^{-2\frac{\cupp}{\cinf}C_n}.
\]

Hence,
\[
\1_{\Abar}\Delta_n
\ge
\1_{\Abar}\frac{1}{\cupp^2}e^{-2\frac{\cupp}{\cinf}C_n}
\sum_{i=1}^n (1-r^i)^2.
\]

Now,
\[
\sum_{i=1}^n (1-r^i)^2
=
n-2\sum_{i=1}^n r^i+\sum_{i=1}^n r^{2i}
\ge
n-2\sum_{i=1}^\infty r^i
=
n-\frac{2r}{1-r}
=
n-\frac{2}{e^{\cupp\alphaval}-1}.
\]
Therefore,
\[
\1_{\Abar}\Delta_n
\ge
\1_{\Abar}\frac{1}{\cupp^2}e^{-2\frac{\cupp}{\cinf}C_n}
\left(n-\frac{2}{e^{\cupp\alphaval}-1}\right).
\]

Finally, since \(e^x-1\ge x\) for all \(x\ge 0\),
\[
\frac{2}{e^{\cupp\alphaval}-1}
\le
\frac{2}{\cupp\alphaval},
\]
such that 
\[
\1_{\Abar}\Delta_n
\ge
\1_{\Abar}\frac{1}{\cupp^2}e^{-2\frac{\cupp}{\cinf}C_n}
\left(n-\frac{2}{\cupp \alphaval}\right)=\1_{\Abar}\frac{n}{\cupp^2}e^{-2\frac{\cupp}{\cinf}C_n}
\left(1-\frac{2}{\cupp \alphaval n}\right) .
\]

\section{Proof of Theorem~\ref{thm:high_proba_bound}}\label{app:proof_high_proba}
We prove Theorem~\ref{thm:high_proba_bound} in this section, which we restate now.
\maintheorem*

\begin{sproof}
Define 
$$H_0^i :=\sum_{j=1}^i (T_i-T_j)e^{-\alphaval (T_i-T_j)}, \; H_1^i := \sum_{j=1}^i e^{-\alphaval (T_i-T_j)}, \;H_2^i := \sum_{j=1}^i (i-j)e^{-\alphaval (T_i-T_j)}.$$
The proof follows the same five-step strategy for each of the three sums
$H_0^i$, $H_1^i$, $H_2^i$.

\medskip\noindent\textbf{Step 1 --- Choice of block size $\Kval $.}
Since $T_i - T_j = \sum_{k=j}^{i-1}\Delta_k$ with $\Delta_k
\overset{\mathrm{i.i.d.}}{\sim}\mathcal{E}(1)$, the Chernov inequality
controls deviations of a sum of \emph{fixed} length $\Kval $ from its
expectation, whose value is $\Kval $. However, we need a bound valid for all gap lengths $i-j$
simultaneously. The key device is to fix a block size $\Kval $ (whose value
will be determined in Step~2) and decompose each sum into consecutive
blocks of size $\Kval $ plus a residual block.

\medskip\noindent\textbf{Step 2 --- Concentration of $T_i-T_j$
(Proposition~\ref{prop:min_gamma}).}
Applying Chernov to each block of size $\Kval $, combined with a union bound
over the $O(n)$ possible block starting points, gives with probability
$1-\varepsilon$, for all $1 \le j \le i \le n$:
\[
  \cinf(i-j) - \cinf \Kval  \;\le\; T_i - T_j \;\le\; \cupp(i-j) + \cupp \Kval .
\]
The value $\Kval  = \ceil{(\cinf-1-\log\cinf)^{-1}\log(2n/\varepsilon)}$
is chosen to make the failure probability of each block at most
$\varepsilon/(2n)$, so that using a union bound implies a probability over all the blocks of $\varepsilon/2$, that does not grows with $n$.

\medskip\noindent\textbf{Step 3 --- High-probability upper-bounds on
$H_k^i$ (Propositions~\ref{prop:maj_H_1}--\ref{prop:maj_H_2}--\ref{prop:maj_H_0}).}
Substituting the bounds of Step~2 into each sum yields
upper-bounds with high probability on $\{H_k^i\}_{k\in \{0,1,2\}}$. These upper-bound involve terms that depend on $K$.

\medskip\noindent\textbf{Step 4 --- lower-bounds on expectations
(Lemmas~\ref{lem:exp_low_bound_H_0}--\ref{lem:exp_low_bound_H_2}--\ref{lem:exp_low_bound_H_1}).}
Computing the characteristic function of the Gamma distribution,
$\mathbb{E}[H_k^i]$ is computed exactly in closed form, then
lower-bounded 
by quantity whose form matches those of the upper-bound in Step~3.

\medskip\noindent\textbf{Step 5 --- Conclusion
(Section~\ref{app:combine}).}
The upper-bounds from Step 3 involve terms that depend on $K$, while the lower-bounds from Step 4 do not. We define $C_n$ such that $K = \bigO(C_n/\alphaval)$, allowing to combine these bounds. We obtain that the sums $H_k^i$ are upper-bounded by their expectations, up to constants involving $C_n$. As long as $\alphaval = n^{-\beta}$ for some $\beta > 0$, the constant $C_n$ will be uniformly upper-bounded with $n$, and will verify $C_n \to_{n \to +\infty} 0$. 

\end{sproof}

\subsection{Upper-bounding the sums}\label{app:upper}

A key result for our proof is based on a rather classical concentration inequality called the Chernov inequality. We use it in the case of gamma laws, see \textit{e.g.} \citep[Section 2.4]{boucheron2003concentration}.
\begin{lemma}[Chernov inequality]\label{lem:chernov}\notag
    For $k \in \N^\ast$, let $T_k \sim \Gamma(k,1)$, $0<\cinf \le 1 \le \cupp$. Then,
    \begin{enumerate}[label=(\roman*)]
        \item $\P(T_k \le \cinf k)\le e^{-( \cinf-1-\log( \cinf))k}.$
        \item $\P(T_k \ge \cupp k)\le e^{-( \cupp-1-\log( \cupp))k}.$
    \end{enumerate}
\end{lemma}

\begin{proposition}\label{prop:min_gamma}
    Let $\varepsilon,\cinf \in (0,1)$, $n > 0$, and  $\Kval  := \ceil{(\cinf-1-\log\cinf)^{-1}\log(2n/\varepsilon)}$. Let $\cupp \ge 1$ be chosen such that $\cupp$ such that $\cupp- \log(\cupp) = \cinf - \log(\cinf)$. Let the set $\Abar$ such that we have for any  $1 \le j \le i \le n$
    \[ \cinf (i-j) - \cinf \Kval  \le  T_i-T_j \le \cupp (i-j)+\cupp \Kval.\]
    Then, we have $\mathbb{P}(\Abar) \ge 1-\varepsilon$.
\end{proposition}
\begin{proof}
  We note $\Delta_k := T_{k+1}-T_k$ such that
\[T_i-T_j = \sum_{k=j}^{i-1}\Delta_k,\]
where $\Delta_k \overset{i.i.d}{\sim} \mathcal{E}(1)$.
    Let an integer $\Kval  >0$, to be further fixed.

    \noindent
    \textbf{Lower-bound.}
    The idea is to decompose $\sum_{k=j}^{i-1}\Delta_k$ into $\floor{(i-j)/\Kval }$ blocks of size $\Kval $ starting at $j,j+\Kval ,j+2\Kval ,\dots$, plus a
residual block of size $(i-j)\bmod \Kval  < \Kval $.
Chernov is applied to each complete block; the residual block is
non-negative ($\Delta_k \ge 0$ a.s.) and is simply dropped, at the cost
of the term $-\cinf \Kval $ in the final bound.

    We have $\sum_{k=j}^{j-1+\Kval } \Delta_k = T_{j+\Kval } - T_j$ which follows the same law as $T_{\Kval } \sim \Gamma(\Kval ,1)$. Therefore, for any $\cinf \in (0,1)$, applying Lemma~\ref{lem:chernov}~$(i)$,
    \begin{equation}\label{eq:hoeffding_based:1}
        \sum_{k=j}^{j-1+\Kval } \Delta_k  \le \cinf \Kval ,
    \end{equation}
with probability at most $e^{-\Kval (\cinf-1-\log\cinf)}$.
Choosing $\Kval  = \ceil{(\cinf-1-\log\cinf)^{-1}\log(2n/\varepsilon)}$,
this probability is at most $\varepsilon/(2n)$.
By a union bound over $j \in \{1,\dots,n-\Kval \}$, with probability
$1-\varepsilon/2$, \eqref{eq:hoeffding_based:1} fails for all such $j$, namely

        \begin{equation}\label{eq:hoeffding_based:2}
        \sum_{k=j}^{j-1+\Kval } \Delta_k  > \cinf \Kval .
    \end{equation}
Now, consider the following decomposition
\[
 \sum_{k=j}^{i-1}\Delta_k
  = \underbrace{\sum_{s=0}^{\floor{\frac{i-j}{\Kval }}-1}
       \sum_{k=j+s\Kval }^{j+(s+1)\Kval -1}\Delta_k}_{\floor{(i-j)/\Kval }
       \text{ complete blocks}}
   + \underbrace{\sum_{k=j+\floor{\frac{i-j}{\Kval }}\Kval }^{i-1}
     \Delta_k}_{\text{residual block}}.\]

For $s \in \left\{0,\cdots,\floor{\frac{i-j}{\Kval }}-1\right\}$, we have $j+s\Kval  \in\{j,\cdots,i-\Kval \} \subset \{1,\cdots,n-\Kval  \}$. So, \eqref{eq:hoeffding_based:2} gives that with probability $1-\varepsilon/2$, for any such $s$ we have
$\sum_{k=j+s\Kval }^{j+(s+1)\Kval -1}\Delta_k > \cinf \Kval $.
The residual block can be dropped, at it is nonnegative. Then,

\begin{equation}\label{eq:hoeffding_based:3}
\begin{aligned}
  \sum_{k=j}^{i-1}\Delta_k
  &=\sum_{s=0}^{\floor{\frac{i-j}{\Kval }}-1}
       \sum_{k=j+s\Kval }^{j+(s+1)\Kval -1}\Delta_k
   +\sum_{k=j+\floor{\frac{i-j}{\Kval }}\Kval }^{i-1}
     \Delta_k \\
  &\ge \sum_{s=0}^{\floor{\frac{i-j}{\Kval }}-1} \cinf \Kval \\
  &= \cinf \Kval  \floor{\frac{i-j}{\Kval }} \\
  &\ge \cinf \Kval  \bpar{\frac{i-j}{\Kval }-1}
   \qquad\text{($\floor{x}\ge x-1$)}\\
  &= \cinf(i-j) - \cinf \Kval .
\end{aligned}
\end{equation}

    \noindent
    \textbf{Upper-bound.} As for the lower-bound, the strategy for the upper-bound exploits the following decomposition
\[
 \sum_{k=j}^{i-1}\Delta_k
  = \underbrace{\sum_{s=0}^{\floor{\frac{i-j}{\Kval }}-1}
       \sum_{k=j+s\Kval }^{j+(s+1)\Kval -1}\Delta_k}_{\floor{(i-j)/\Kval }
       \text{ complete blocks}}
   + \underbrace{\sum_{k=j+\floor{\frac{i-j}{\Kval }}\Kval }^{i-1}
     \Delta_k}_{\text{residual block}}.\]
However, a difference is that the residual block cannot be dropped by a non-negativity argument, because here we want to obtain an upper-bound. A simple strategy is to bound a residual block by a complete block of size $K$, namely 
\[\sum_{k=j+\floor{\frac{i-j}{\Kval }}\Kval }^{i-1}
     \Delta_k \le \sum_{k=j+\floor{\frac{i-j}{\Kval }}\Kval }^{j+\bpar{\floor{\frac{i-j}{\Kval }}+1}\Kval-1}
     \Delta_k.\]
However, some terms in the sum of the right-hand side are not defined if $j+\bpar{\floor{\frac{i-j}{\Kval }}\Kval+1}-1 \ge i$. So, for the purpose of the analysis, we introduce
$\{\Delta_k\}_{n \le k \le n+\Kval -1}\overset{\mathrm{i.i.d.}}{\sim}
\mathcal{E}(1)$, independent of $\{\Delta_k\}_{k \le n-1}$, such that we can write
\begin{equation}\label{eq:hoeffding_upper:0}
    \begin{aligned}
         \sum_{k=j}^{i-1}\Delta_k
  &=\sum_{s=0}^{\floor{\frac{i-j}{\Kval }}-1}
       \sum_{k=j+s\Kval }^{j+(s+1)\Kval -1}\Delta_k
   +\sum_{k=j+\floor{\frac{i-j}{\Kval }}\Kval }^{i-1}
     \Delta_k \\
     &\le \sum_{s=0}^{\floor{\frac{i-j}{\Kval }}-1}
       \sum_{k=j+s\Kval }^{j+(s+1)\Kval -1}\Delta_k
   +\sum_{k=j+\floor{\frac{i-j}{\Kval }}\Kval }^{j+\bpar{\floor{\frac{i-j}{\Kval }}+1}\Kval-1}
     \Delta_k\\
     &= \sum_{s=0}^{\floor{\frac{i-j}{\Kval }}}
       \sum_{k=j+s\Kval }^{j+(s+1)\Kval -1}\Delta_k
    \end{aligned}
\end{equation}


  Then, applying Lemma~\ref{lem:chernov}~$(ii)$, we have that for any $\cupp > 1$ and for $j \in \{1,\cdots,n-1 \}$
    \begin{equation}\label{eq:hoeffding_upper:1}
        \sum_{k=j}^{j-1+\Kval } \Delta_k  \ge \cupp \Kval ,
    \end{equation}
    with probability at most
    \[e^{-\Kval (\cupp-1-\log(\cupp))}.\] Choosing $\cupp$ such that $\cupp- \log(\cupp) = \cinf - \log(\cinf)$, we obtain
    \[\ceil{(\cupp-1-\log(\cupp))^{-1}\log\bpar{2\frac{n}{\varepsilon}}}  = \ceil{(\cinf -1-\log(\cinf ))^{-1}\log\bpar{2\frac{n}{\varepsilon}}} =: \Kval ,\] and it follows that this probability is at most $\varepsilon/(2n)$.
   By a union bound over $j \in \{1,\dots,n-1 \}$, with probability $1-\frac\varepsilon2$, we have 
           \begin{equation}\label{eq:hoeffding_upper:3}
            \sum_{k=j}^{j-1+\Kval } \Delta_k \le \cupp \Kval
        \end{equation}
So, from \eqref{eq:hoeffding_upper:0}, we deduce that with probability $1-\varepsilon/2$, we have
\begin{equation}\label{eq:hoeffding_upper:4}
\begin{aligned}
  \sum_{k=j}^{i-1}\Delta_k
  &\le  \sum_{s=0}^{\floor{\frac{i-j}{\Kval }}}
       \sum_{k=j+s\Kval }^{j+(s+1)\Kval -1}\Delta_k \le \sum_{s=0}^{\floor{\frac{i-j}{\Kval }}}
      \cupp \Kval
   = \cupp \Kval \bpar{\floor{\frac{i-j}{\Kval }}+1}
   \le \cupp(i-j) + \cupp \Kval .
\end{aligned}
\end{equation}

    \noindent
    \textbf{Conclusion.} 
    The claim follows by combining \eqref{eq:hoeffding_based:3} and \eqref{eq:hoeffding_upper:4}, each holding under probability $1-\frac\varepsilon2$.
\end{proof}

\begin{proposition}\label{prop:maj_H_1}
Let  $n\in \N^\ast$, $\cinf  , \varepsilon \in (0,1)$, $\alphaval \in (0,1]$ and   $\Kval  = \ceil{(\cinf -1-\log(\cinf ))^{-1}\log\bpar{2\frac{n}{\varepsilon}}}$. With probability $1-\varepsilon$, we have
\[\forall i \in \{1,\dots,n \},~\sum_{j=1}^{i} e^{-\alphaval (T_i-T_j)} \le\min\!\left\{
i,\,
\frac{2}{\alphaval \cinf  }+\Kval 
\right\}.\]
\end{proposition}

\begin{proof}
    We bound 
    \[H_1^i := \sum_{j=1}^i e^{-\alphaval (T_i-T_j)}.\]
We note that as long as $i>\Kval $, the $i-\Kval $ first terms of $H_1^i$ can be controlled using the lower-bound on $T_i-T_j$ from
Proposition~\ref{prop:min_gamma}. It remains the last terms, whose amount is at most $\Kval $; these terms can be bounded by $1$.

    Let $i \le \Kval $.  We have
    \begin{equation}
        \sum_{j=1}^i e^{-\alphaval (T_i-T_j)} \le i  \le\min\!\left\{
i,\,
\frac{2}{\alphaval \cinf  }+\Kval 
\right\}.
    \end{equation}

    Let $i > \Kval $. With probability $1-\varepsilon$, we have
\begin{equation}
\begin{aligned}
\sum_{j=1}^{i} e^{-\alphaval (T_i-T_j)}
&= \sum_{j=1}^{i-\Kval } e^{-\alphaval (T_i-T_j)}
 + \sum_{j=i-\Kval +1}^{i} e^{-\alphaval (T_i-T_j)}\\
&\le \sum_{j=1}^{i-\Kval } e^{-\alphaval (T_i-T_j)} + \Kval \\
&\overset{(i)}{\le} \sum_{j=1}^{i-\Kval }
e^{-\alphaval\left(\cinf  (i-j)-\cinf \Kval   \right)}
+ \Kval \\
&\overset{(ii)}{=}
\sum_{\ell=0}^{i-\Kval -1} e^{-\alphaval\cinf\ell}
+ \Kval \\
&\overset{(iii)}{\le}
\min\!\left\{
i-\Kval ,\,
\frac{1}{1-e^{-\alphaval \cinf}}
\right\}
+ \Kval \\
&\overset{(iv)}{\le}
\min\!\left\{
i,\,
\frac{2}{\alphaval \cinf}+\Kval 
\right\}.
\end{aligned}
\end{equation}

For step $(i)$, we apply $T_i-T_j \ge \cinf(i-j)-\cinf \Kval $. This is due to Proposition~\ref{prop:min_gamma}, valid with probability $1-\varepsilon$, for any $j \le i-\Kval $ since
$i-j \ge \Kval $. For step $(ii)$ we set the change of variable $\ell=i-\Kval -j$,
and simplify $e^{\alphaval\cinf \Kval }e^{-\alphaval\cinf \Kval }= 1$.
$(iii)$ uses 
$\sum_{\ell=0}^{i-\Kval -1} e^{-\alphaval\cinf\ell}
= \frac{1- e^{-\alphaval \cinf (i-\Kval)}}{1-e^{-\alphaval \cinf}}
\le \frac{1}{1-e^{-\alphaval \cinf}}$. 
$(iv)$ uses $e^{-t}\le 1-t/2$ for $t=\alphaval\cinf\in(0,1]$,
giving $\frac{1}{1-e^{-\alphaval\cinf}}\le\frac{2}{\alphaval\cinf}$.
\end{proof}

\begin{proposition}\label{prop:maj_H_2}
Let  $n\in \N^\ast$, $\cinf  , \varepsilon \in (0,1)$, $\alphaval \in (0,1]$ and   $\Kval  = \ceil{(\cinf -1-\log(\cinf ))^{-1}\log\bpar{2\frac{n}{\varepsilon}}}$. With probability $1-\varepsilon$, we have

\[
\forall i \in \{1,\dots,n \},~
\sum_{j=1}^i (i-j)e^{-\alphaval (T_i-T_j)}
\le
\min\!\left\{
\frac{i^2}{2},\,
\frac{4}{(\alphaval\cinf)^2}
+
\frac{2\Kval }{\alphaval\cinf}
+
  \frac{\Kval ^2}2
\right\}.
\]
\end{proposition}

\begin{proof}
    We bound 
    \[
    H_2^i := \sum_{j=1}^i (i-j)e^{-\alphaval (T_i-T_j)}.
    \]
Compared with $H_1^i$, the weight $(i-j)$ adds a linear factor. 
When treating $H_2^i$, the main difference is that the sum $\sum_\ell\ell\,q^\ell$ is handled via
$\sum_{\ell=0}^\infty\ell q^\ell = q/(1-q)^2$.
   
    If $i \le \Kval $,
    \[
    H_2^i \le \sum_{j=1}^i (i-j)= \frac{i(i-1)}{2} \le \frac{i^2}{2}\le \min\!\left\{
\frac{i^2}{2},\,
\frac{4}{(\alphaval\cinf)^2}
+
\frac{2\Kval }{\alphaval\cinf}
+
 \frac{\Kval ^2}2
\right\}.
    \]

    If $i > \Kval $, with probability $1-\varepsilon$, we have
\begin{equation}
\begin{aligned}
H_2^i
&=
\sum_{j=1}^{i-\Kval } (i-j)e^{-\alphaval (T_i-T_j)}
+
\sum_{j=i-\Kval +1}^{i}(i-j)e^{-\alphaval (T_i-T_j)}\\
&\overset{(i)}\le
\sum_{j=1}^{i-\Kval }
(i-j)e^{-\alphaval(\cinf(i-j)-\cinf \Kval )}
+
\sum_{j=i-\Kval +1}^{i}(i-j)\\
&\overset{(ii)}{\le}
\sum_{\ell=0}^{i-\Kval -1}
\bigl(\ell+\Kval \bigr)q^\ell
+ \frac{\Kval ^2}2.
\end{aligned}
\end{equation}
Step $(i)$ uses $T_i-T_j \ge \cinf(i-j)-\cinf \Kval $ , holding with probability $1-\varepsilon$, thanks to Proposition~\ref{prop:min_gamma}. It also uses $e^{-\alphaval (T_i-T_j)} \le 1$. In step $(ii)$ we set $q := e^{-\alphaval\cinf}$ and the change of variable $\ell = i-j-K$.

Thus,
\begin{equation}
\begin{aligned}
H_2^i
&\le
\sum_{\ell=0}^{i-\Kval -1}\ell q^\ell
+
\Kval \sum_{\ell=0}^{i-\Kval -1} q^\ell
+ \frac{\Kval ^2}2.
\end{aligned}
\end{equation}

Now,
\[
\sum_{\ell=0}^{m-1}\ell q^\ell
\le
\min\!\left\{
\sum_{\ell=0}^{m-1}\ell,\,
\sum_{\ell=0}^{\infty}\ell q^\ell
\right\}
=
\min\!\left\{
\frac{m(m-1)}{2},\,
\frac{q}{(1-q)^2}
\right\},
\]
and
\[
\sum_{\ell=0}^{m-1} q^\ell
\le
\min\!\left\{
m,\,
\frac{1}{1-q}
\right\}.
\]

Applying this with \(m=i-\Kval \), we obtain
\begin{equation}
\begin{aligned}
H_2^i
&\le
\min\!\left\{
\frac{(i-\Kval )(i-\Kval -1)}{2},\,
\frac{q}{(1-q)^2}
\right\}\\
&\qquad
+
\Kval \min\!\left\{
i-\Kval ,\,
\frac{1}{1-q}
\right\}
+ \frac{\Kval ^2}2.
\end{aligned}
\end{equation}

Using \(\min(a,b)+\min(c,d)\le \min(a+c,b+d)\), we get
\begin{equation}
\begin{aligned}
H_2^i
&\le
\min\!\Bigg\{
\frac{(i-\Kval )(i-\Kval -1)}{2}
+
\Kval (i-\Kval )
+
\frac{\Kval ^2}2,\\
&\hspace{2.5cm}
\frac{q}{(1-q)^2}
+
\frac{\Kval }{1-q}
+
\frac{\Kval ^2}2
\Bigg\}.
\end{aligned}
\end{equation}

 The first term is equal to $\frac{i^2-i+K}{2}$. Since \(i>\Kval \), it follows that $\frac{i^2-i+K}{2} \leq \frac{i^2}{2}$. 
Then,
\begin{equation}
\begin{aligned}
H_2^i
\le
\min\!\left\{
\frac{i^2}{2},\,
\frac{q}{(1-q)^2}
+
\frac{\Kval }{1-q}
+
\frac{\Kval ^2}2
\right\}.
\end{aligned}
\end{equation}

As moreover \(\alphaval\cinf\le 1\), then
\[
1-q=1-e^{-\alphaval\cinf}\ge \frac{\alphaval\cinf}{2},
\]
hence
\[
\frac{1}{1-q}\le \frac{2}{\alphaval\cinf},
\qquad
\frac{q}{(1-q)^2}\le \frac{1}{(1-q)^2}\le \frac{4}{(\alphaval\cinf)^2}.
\]
Therefore,
\begin{equation}
\begin{aligned}
H_2^i
\le
\min\!\left\{
\frac{i^2}{2},\,
\frac{4}{(\alphaval\cinf)^2}
+
\frac{2\Kval }{\alphaval\cinf}
+
\frac{\Kval ^2}2
\right\}.
\end{aligned}
\end{equation}
\end{proof}

The following proposition differs from Propositions~\ref{prop:maj_H_1}
and~\ref{prop:maj_H_2} for two reasons. First, because the weights $(T_i-T_j)$ are random, we need both the upper
bound ($T_i-T_j\le\cupp(i-j)+\cupp \Kval $) and the lower-bound
($T_i-T_j\ge\cinf(i-j)-\cinf \Kval $) from Proposition~\ref{prop:min_gamma}. The upper-bound controls the weight, the lower
bound gives exponential decay. Then, a difficulty is that for small $i$, there is no direct upper-bound of $(T_i-T_j)$ that depends on $i$, which will lead to a crude upper-bound in this case that does not depend on $i$. Fortunately, it will be sufficient to prove our main result (Theorem~\ref{thm:hess_lip}).


\begin{proposition}\label{prop:maj_H_0}
Let $\cinf  , \varepsilon \in (0,1)$, $\cupp$ and $\Kval $ defined as in Proposition~\ref{prop:min_gamma}. With probability $1-\varepsilon$, we have
\[\forall i \in \{1,\dots,\Kval  \},~ \sum_{j=1}^i (T_i-T_j)e^{-\alphaval (T_i-T_j)} \le \cupp \Kval ^2,\]
and

\[\forall i \in \{\Kval +1,\dots,n \},~
\sum_{j=1}^i (T_i-T_j)e^{-\alphaval (T_i-T_j)}
\le
\min\!\left\{
\cupp i^2,\,
\frac{4\cupp}{(\alphaval\cinf)^2}
+
\frac{4\cupp \Kval }{\alphaval\cinf}
+
\cupp \Kval ^2
\right\}\]

\end{proposition}

\begin{proof}
    We bound 
\[
H_0^i :=\sum_{j=1}^i (T_i-T_j)e^{-\alphaval (T_i-T_j)}.
\]
We recall $\cinf < 1 < \cupp$ are chosen such that
\[
\Kval  = \ceil{(\cupp-1-\log(\cupp))^{-1}\log\!\left(\frac{2n}{\varepsilon}\right)}
= \ceil{(\cinf-1-\log(\cinf))^{-1}\log\!\left(\frac{2n}{\varepsilon}\right)}.
\]

If $i \le \Kval $, with probability $1-\varepsilon$ we have
\[
H_0^i \le \sum_{j=1}^i (T_i-T_j) \overset{(i)}{\le} \sum_{j=1}^i \cupp \Kval  \le \cupp i \Kval  \le \cupp \Kval ^2.
\]
$(i)$ uses that as $i-j\le \Kval $, we have $T_i-T_j \le T_{j+\Kval } - T_j \le \cupp \Kval $, the last inequality being due to the Chernov inequality (see \eqref{eq:hoeffding_upper:3}).

If $i > \Kval $, we write
\begin{equation}
\begin{aligned}
\sum_{j=1}^i (T_i-T_j)e^{-\alphaval (T_i-T_j)}
&=
\sum_{j=1}^{i-\Kval } (T_i-T_j)e^{-\alphaval (T_i-T_j)}
+
\sum_{j=i-\Kval +1}^i (T_i-T_j)e^{-\alphaval (T_i-T_j)}.
\end{aligned}
\end{equation}

For the second sum, with probability $1-\varepsilon$,
\[
\sum_{j=i-\Kval +1}^i (T_i-T_j)e^{-\alphaval (T_i-T_j)}
\le
\sum_{j=i-\Kval +1}^i (T_i-T_j)
\overset{(i)}\le
\cupp \Kval ^2.
\]
Again, $(i)$ uses that as $i-j\le \Kval $, we have $T_i-T_j \le T_{j+\Kval } - T_j \le \cupp \Kval $.

Then, with probability $1-\varepsilon$, we have by Proposition~\ref{prop:min_gamma}, 
\[
T_i-T_j \ge \cinf(i-j)-\cinf \Kval ,
\]
and
\[
T_i-T_j \le \cupp(i-j)+\cupp \Kval .
\]
We use both inequalities simultaneously, the upper-bound controls
the factor $(T_i-T_j)$, the lower-bound gives exponential decay in
$e^{-\alphaval(T_i-T_j)}$.
Let $q:=e^{-\alphaval \cinf}$.
Therefore,
\begin{equation}
\begin{aligned}
H_0^i
&\le
\sum_{j=1}^{i-\Kval }
\bigl(\cupp(i-j)+\cupp \Kval \bigr)
e^{-\alphaval(\cinf(i-j)-\cinf \Kval )}
+ \cupp \Kval ^2\\
&\overset{(i)}{\le}
\sum_{\ell=0}^{i-\Kval -1}
\bigl(\cupp \ell+2\cupp \Kval \bigr)q^\ell
+ \cupp \Kval ^2.
\end{aligned}
\end{equation}
In (i) we set the change $\ell=i-j-K$.
Then,
\begin{equation}
\begin{aligned}
H_0^i
&\le
\cupp\sum_{\ell=0}^{i-\Kval -1}\ell q^\ell
+
2\cupp \Kval \sum_{\ell=0}^{i-\Kval -1} q^\ell
+ \cupp \Kval ^2.
\end{aligned}
\end{equation}

Now,
\[
\sum_{\ell=0}^{m-1}\ell q^\ell
\le
\min\!\left\{
\sum_{\ell=0}^{m-1}\ell,\,
\sum_{\ell=0}^{\infty}\ell q^\ell
\right\}
=
\min\!\left\{
\frac{m(m-1)}{2},\,
\frac{q}{(1-q)^2}
\right\},
\]
and
\[
\sum_{\ell=0}^{m-1} q^\ell
\le
\min\!\left\{
m,\,
\frac{1}{1-q}
\right\}.
\]

Applying this with \(m=i-\Kval \), we obtain
\begin{equation}
\begin{aligned}
H_0^i
&\le
\cupp\min\!\left\{
\frac{(i-\Kval )(i-\Kval -1)}{2},\,
\frac{q}{(1-q)^2}
\right\}\\
&\qquad
+
2\cupp \Kval \min\!\left\{
i-\Kval ,\,
\frac{1}{1-q}
\right\}
+
\cupp \Kval ^2.
\end{aligned}
\end{equation}

Using \(\min(a,b)+\min(c,d)\le \min(a+c,b+d)\), we get
\begin{equation}
\begin{aligned}
H_0^i
&\le
\min\!\Bigg\{
\cupp\frac{(i-\Kval )(i-\Kval -1)}{2}
+
2\cupp \Kval (i-\Kval )
+
\cupp \Kval ^2,\\
&\hspace{2.5cm}
\cupp\frac{q}{(1-q)^2}
+
\frac{2\cupp \Kval }{1-q}
+
\cupp \Kval ^2
\Bigg\}.
\end{aligned}
\end{equation}
The first component in the $\min$ is equal to $\cupp \frac{i^2-i+2iK+K-K^2}{2}$. Then,
\begin{align*}
    \cupp \frac{i^2-i+2iK+K-K^2}{2} &= \cupp \frac{i^2-i+(i+1/2)^2 - (K-i-1/2)^2}{2}\\
    &\le  \cupp \frac{i^2-i+(i+1/2)^2 - (1/2)^2}{2}, \textit{ as }K<i\\
    &= \cupp.
\end{align*}


So, we deduce
\begin{equation}
\begin{aligned}
H_0^i
\le
\min\!\left\{
\cupp i^2,\,
\cupp\frac{q}{(1-q)^2}
+
\frac{2\cupp \Kval }{1-q}
+
\cupp \Kval ^2
\right\}.
\end{aligned}
\end{equation}

As moreover \(\alphaval\cinf\le 1\), then
\[
1-q=1-e^{-\alphaval\cinf}\ge \frac{\alphaval\cinf}{2},
\]
hence
\[
\frac{1}{1-q}\le \frac{2}{\alphaval\cinf},
\qquad
\frac{q}{(1-q)^2}\le \frac{1}{(1-q)^2}\le \frac{4}{(\alphaval\cinf)^2}.
\]
Therefore,
\begin{equation}
\begin{aligned}
H_0^i
\le
\min\!\left\{
\cupp i^2,\,
\frac{4\cupp}{(\alphaval\cinf)^2}
+
\frac{4\cupp \Kval }{\alphaval\cinf}
+
\cupp \Kval ^2
\right\}.
\end{aligned}
\end{equation}
\end{proof}

\subsection{lower-bounding expectations}\label{app:lower}
The three lemmas below compute or lower-bound the terms $\mathbb{E}[H_k^i]$, recalling
$$H_0^i :=\sum_{j=1}^i (T_i-T_j)e^{-\alphaval (T_i-T_j)}, \; H_1^i := \sum_{j=1}^i e^{-\alphaval (T_i-T_j)}, \;H_2^i := \sum_{j=1}^i (i-j)e^{-\alphaval (T_i-T_j)}.$$
Since $T_i-T_j\sim\Gamma(i-j,1)$, the Laplace transform of the Gamma
distribution gives exact closed-form expressions, which are then
lower-bounded using elementary inequalities.
\begin{lemma}\label{lem:exact_exp}
    We have
    \begin{itemize}
        \item $\E{H_0^i} =\frac{1-i(1+\alphaval)^{-(i-1)} + (i-1)(1+\alphaval)^{-i}}{\alphaval^2} $
        \item $\E{H_2^i} = (1+\alphaval)\frac{1-i(1+\alphaval)^{-(i-1)} + (i-1)(1+\alphaval)^{-i}}{\alphaval^2}$
    \end{itemize}
\end{lemma}
\begin{proof}
Applying Lemma \ref{lem:exp_stopping_time}, we have
    \[\E{H_2^i} =\sum_{j=1}^i (i-j)\E{e^{-\alphaval (T_i-T_j)}} = \sum_{j=1}^i (i-j)(1+\alphaval)^{-(i-j)} = \sum_{k=1}^{i-1}k(1+\alphaval)^{-k}. \]
    By Lemma~\ref{lem:sum_deriv} applied to $\gamma = (1+\alphaval)^{-1}$, we have
    \begin{align*}
        \E{H_2^i} &=(1+\alphaval)^{-1}\frac{1-i(1+\alphaval)^{-(i-1)} + (i-1)(1+\alphaval)^{-i}}{(1-(1+\alphaval)^{-1})^2} \\
        &= (1+\alphaval)\frac{1-i(1+\alphaval)^{-(i-1)} + (i-1)(1+\alphaval)^{-i}}{\alphaval^2}.
    \end{align*}
     
 By Lemma \ref{lem:gamma_law_deriv}, since $T_i - T_j \sim \Gamma(i-j, 1)$ and setting $k = i-j$:
\begin{align*}
\E{H_i^0} &= \sum_{j=1}^i \E{(T_i-T_j)e^{-\alpha(T_i-T_j)}} = \sum_{k=1}^{i-1} k(1+\alpha)^{-(k+1)} \\
&= (1+\alphaval)^{-1}\sum_{k=1}^{i-1} k(1+\alphaval)^{-k} =(1+\alphaval)^{-1} \E{H_i^2}\\
&= \frac{1-i(1+\alpha)^{-(i-1)}+(i-1)(1+\alpha)^{-i}}{\alpha^2}.
\end{align*}
    
\end{proof}
\begin{lemma}\label{lem:exp_low_bound_H_1}
    \[\E{H_1^i}
\ge
(1-e^{-1})\min\!\left\{i,\frac1\alphaval\right\}.
\]
\end{lemma}
\begin{proof}
Applying Lemma \ref{lem:exp_stopping_time} and bounding the discrete sum by an integral, we have:
\[\E{H_1^i} = \sum_{j=0}^{i-1} (1+\alphaval)^{-j} \ge \sum_{j=0}^{i-1}e^{-\alphaval j} \ge \int_0^i e^{-\alphaval t}dt=\frac{1-e^{-\alphaval i}}{\alphaval}.\]
We also applied the standard inequality $\log(1+x)\le x$ to $x = \alphaval$, which implies 
$(1+\alphaval)^{-1}\ge e^{-\alphaval}$, then $(1+\alphaval)^{-j}\ge e^{-\alphaval j}$. 
Now, we split the proof between the case $i\le\alphaval^{-1}$ and $i>\alphaval^{-1}$.

\noindent
\textbf{Case $i\le \alphaval^{-1}$.} Then $\alphaval i\in[0,1]$.
By concavity of $t\mapsto 1-e^{-t}$ on $[0,1]$, we have $1-e^{-t}\ge(1-e^{-1})t$. Applying it to $t=\alphaval i$ we deduce

\[\frac{1-e^{-\alphaval i}}{\alphaval}\ge(1-e^{-1})i.\]

\noindent\textbf{Case $i>\alphaval^{-1}$.} We have $\alphaval i > 1$, which implies $1-e^{-\alphaval i}>1-e^{-1}$. Hence
\[\frac{1-e^{-\alphaval i}}{\alphaval}>\frac{1-e^{-1}}{\alphaval}\]

\noindent
From the two cases, we deduce
\[\E{H_1^i}\ge(1-e^{-1})\min\!\left\{i,\frac{1}{\alphaval}\right\}.\]

\end{proof}

\begin{lemma}\label{lem:exp_low_bound_H_0}
We have
\[
\mathbb E[H_0^i]
\ge
\frac{e^{-1}}{32}
\min\!\left\{i^2,\frac1{\alphaval^2}\right\}.
\]
\end{lemma}
\begin{proof}
Writing \(r=(1+\alphaval)^{-1}\), Lemma~\ref{lem:exact_exp} gives
\[
\mathbb E[H_0^i]
=
\frac{1-ir^{\,i-1}+(i-1)r^i}{\alphaval^2} = \frac{1-r^{\,i-1}(i-r(i-1))}{\alphaval^2}.
\]
Since
\[
i-(i-1)r= 1 + (i-1)(1-r) = 1+\frac{\alphaval(i-1)}{1+\alphaval},
\]
noting $y:=\frac{\alphaval(i-1)}{1+\alphaval}$, it follows that
\[
\mathbb E[H_0^i]
=
\frac{1-r^{\,i-1}(1+y)}{\alphaval^2}.
\]
We use $\log(1+x)\ge x/(1+x)$ at $x =\alphaval$, which induces

\[
r^{\,i-1}=e^{-(i-1)\log(1+\alphaval)}\le e^{-(i-1)\frac{\alphaval}{1+\alphaval}} =  e^{-y},
\]
and hence
\[
\mathbb E[H_0^i]
\ge
\frac{1-(1+y)e^{-y}}{\alphaval^2}.
\]
Now, we note the following integration by part property: 
$\int_0^y te^{-t}\,dt=[-te^{-t}]_0^y+\int_0^y e^{-t}\,dt
=-ye^{-y}+(1-e^{-y})=1-(1+y)e^{-y}$. We lower-bound this integral form, depending on the value of $y$. \begin{itemize}
  \item $y\le 1$:
  $\int_0^y te^{-t}\,dt\ge e^{-1}\!\int_0^y t\,dt=\frac{e^{-1}}{2}y^2$.
  \item $y\ge 1$:
  $\int_0^y te^{-t}\,dt\ge\int_0^1 te^{-t}\,dt\ge\frac{e^{-1}}{2}$.
\end{itemize}
From this we deduce
\[
\mathbb E[H_0^i]
\ge
\frac{e^{-1}}{2\alphaval^2}\min\{y^2,1\}.
\]
As \(0<\alphaval\le 1\), then \(y\ge \alphaval(i-1)/2\), so
$\min\{y^2,1\}\ge\min\{\alphaval^2(i-1)^2/4,1\}$. This implies
\[
\mathbb E[H_0^i]
\ge
\frac{e^{-1}}8
\min\!\left\{(i-1)^2,\frac1{\alphaval^2}\right\}.
\]
For \(i\ge 2\), we have $(i-1)\ge i/2$ so $(i-1)^2\ge i^2/4$. We then conclude
\[
\mathbb E[H_0^i]
\ge
\frac{e^{-1}}{32}
\min\!\left\{i^2,\frac1{\alphaval^2}\right\}.
\]
\end{proof}
\begin{lemma}\label{lem:exp_low_bound_H_2}
We have
\[
\mathbb E[H_2^i]
\ge
\frac{e^{-1}}{32}
\min\!\left\{i^2,\frac1{\alphaval^2}\right\}.
\]
\end{lemma}
\begin{proof}
    The claim can be deduced from Lemma~\ref{lem:exp_low_bound_H_0}. Indeed, recall
    \[
\mathbb E[H_0^i]
=\sum_{k=1}^{i-1}\frac{k}{(1+\alphaval)^{k+1}} \le \sum_{k=1}^{i-1}\frac{k}{(1+\alphaval)^{k}} = \E{H_2^i}, 
\]
where we used $\alphaval  > 0$.
\end{proof}
\subsection{Combining lower and upper-bounds}\label{app:combine}

We conclude the proof of Theorem~\ref{thm:high_proba_bound} in this section.

We define
\begin{equation}\label{eq:c_n_def}
    C_n :=  \ceil{(\cinf -1-\log(\cinf ))^{-1}\log\bpar{2\frac{n}{\varepsilon}}}\cinf\alphaval = \Kval \cinf\alphaval.
\end{equation}

In the proof, we will replace $\Kval $ with $C_n/(\alphaval \cinf).$Therefore, $C_n$ measures the cost of replacing $\Kval$ by $(\alphaval \cinf)^{-1}$ in the upper-bounds of Section \ref{app:upper}. Note that for our convergence result, our final choice of $\alphaval$ will be of order $1/n^{1/7}$, such that $C_n$ will be of order $\log(\frac{n}{\varepsilon})n^{-1/7}$. Therefore, we will have $C_n \to 0$ as $n \to +\infty$, and that $C_n$ is uniformly bounded on $n \in \N^\ast$.

The proof for the 3 sums are similar, it combines the upper-bounds from Section~\ref{app:upper} and the lower-bounds from Section~\ref{app:lower}.

\noindent
\textbf{(1st step.)} 
  From Proposition~\ref{prop:maj_H_0}, with probability $1-\varepsilon$, we have
  
\[\forall i \in \{\Kval +1,\dots,n \},~
H_0^i
\le
\min\!\left\{
\cupp i^2,\,
\frac{4\cupp}{(\alphaval\cinf)^2}
+
\frac{4\cupp \Kval }{\alphaval\cinf}
+
 \cupp \Kval ^2
\right\}.
\]
Using the definition \eqref{eq:c_n_def} of $C_n$, this becomes
\[\forall i \in \{\Kval +1,\dots,n \},~
H_0^i
\le
\min\!\left\{
\cupp i^2,\,
\bpar{4 + 4 C_n + C_n^2}\frac{\cupp}{(\alphaval\cinf)^2}
\right\}.
\]
As $ (4 + 4C_n + C_n)/\cinf^2 \ge 1$, it implies that with probability $1-\varepsilon$, $\forall i \in \{\Kval +1,\dots,n \}$, we have

   \begin{equation}\label{eq:thm:H_0:1}
 H_0^i \le \frac{\cupp}{c_{inf}^2}\bpar{4 + 4 C_n + C_n^2}\min\left\{ i^2, \frac{1}{\alphaval^2} \right\},
  \end{equation}
  Now, from Lemma~\ref{lem:exp_low_bound_H_0}, we have
 \begin{equation}\label{eq:thm:H_0:2}
 \mathbb E[H_0^i]
\ge
\frac{e^{-1}}{32}
\min\!\left\{i^2,\frac1{\alphaval^2}\right\}.
 \end{equation}
Combining \eqref{eq:thm:H_0:1} and \eqref{eq:thm:H_0:2}, we deduce that with probability $1-\varepsilon$, $\forall i \in \{\Kval +1,\dots,n \}$, we have
\[ \sum_{j=1}^i (T_i-T_j)e^{-\alphaval (T_i-T_j)} \le 32(4+4C_n+C_n^2)\frac{\cupp e}{\cinf^2} \E{\sum_{j=1}^i (T_i-T_j)e^{-\alphaval (T_i-T_j)}}.\]

\noindent
\textbf{(2nd step.)} 
    From Proposition~\ref{prop:maj_H_1}, with probability $1-\varepsilon$
   \[\forall i \in \{1,\dots,n \},~H_1^i\le\min\!\left\{
i,\,
\frac{2}{\alphaval \cinf  }+\Kval 
\right\}.\]
Using the definition \eqref{eq:c_n_def} of $C_n$, we get that with probability $1-\varepsilon$, $\forall i \in \{1,\dots,n \},$
\begin{equation}\label{eq:thm:H_1:1}
\forall i \in \{1,\dots,n \},~H_1^i\le\min\!\left\{
i,\,
\frac{2+C_n}{\alphaval \cinf  }
\right\}.
\end{equation}
Now, from Lemma~\ref{lem:exp_low_bound_H_1}, for any $i \in \{1,\dots,n \},$
\begin{equation}\label{eq:thm:H_1:2}
    (1-e^{-1})^{-1}\E{H_1^i}
\ge
\min\!\left\{i,\frac1\alphaval\right\}.
\end{equation}
Combining \eqref{eq:thm:H_1:1} and \eqref{eq:thm:H_1:2}, we deduce that with probability $1-\varepsilon$, 
\[\forall i \in \{1,\dots,n \},~ H_1^i \le \frac{2+C_n}{\cinf}(1-e^{-1})^{-1}\E{H_1^i}.\]

\noindent
\textbf{(3d step.)} By Proposition~\ref{prop:maj_H_2}, with probability $1-\varepsilon$,
\[
\forall i \in \{1,\dots,n \},~
H_2^i
\le
\min\!\left\{
\frac{i^2}{2},\,
\frac{4}{(\alphaval\cinf)^2}
+
\frac{2\Kval }{\alphaval\cinf}
+
 \Kval ^2
\right\}.
\]
Using the definition \eqref{eq:c_n_def} of $C_n$, we get that with probability $1-\varepsilon$, $\forall i \in \{1,\dots,n \},$

\[
H_2^i
\le
\min\!\left\{
\frac{i^2}{2},\,
\frac{4+2C_n+C_n^2}{(\alphaval\cinf)^2}
\right\}.
\]
Because $1/2 < 4+2C_n+C_n^2$, we deduce that with probability $1-\varepsilon$, $\forall i \in \{1,\dots,n \},$
\begin{equation}\label{eq:thm:H_2:1}
    H_2^i \le \frac{(4+2C_n+C_n^2)}{c_{inf}^2}\min \left\{i^2,\frac{1}{\alphaval^2}\right\}.
\end{equation}
Now, from Lemma~\ref{lem:exp_low_bound_H_2}, for any $i \in \{1,\dots,n\}$,
\begin{equation}\label{eq:thm:H_2:2}
\mathbb E[H_2^i]
\ge
\frac{e^{-1}}{32}
\min\!\left\{i^2,\frac1{\alphaval^2}\right\}.
\end{equation}
Combining \eqref{eq:thm:H_2:1} and \eqref{eq:thm:H_2:2}, we deduce that with probability $1-\varepsilon$, 
\[\forall i \in \{1,\dots,n \},~ H_2^i \le 32(4+2C_n+C_n^2)\frac{e}{\cinf^2}\E{H_2^i}.  \]

\noindent
\textbf{Conclusion.}
We conclude by defining
\[\Ccont :=\max \left\{ 32(4+4C_n+C_n^2)\frac{\cupp e}{\cinf^2}, \frac{2+C_n}{\cinf}(1-e^{-1})^{-1},  32(4+2C_n+C_n^2)\frac{e}{\cinf^2}\right\}.\]
It follows that
\[\Ccont = 32(4+4C_n+C_n^2)\frac{\cupp e}{\cinf^2}. \]

\section{Additional technical Lemmas used to prove Theorem~\ref{thm:hess_lip}}
In this section, we state and prove lemmas used to prove Theorem~\ref{thm:hess_lip}, but not explicitly stated in Section~\ref{sec:proof}. We start with some elementary results.
\begin{lemma}\label{lem:sum_geometric}
    Let $\alphaval > 0$, $i \ge 1$. We have for $c>0$
    \begin{equation*}
        \sum_{j=1}^i (1+c\alphaval)^{-(i-j)} = \frac{1+c\alphaval - (1+c\alphaval)^{-i+1}}{c\alphaval}, 
    \end{equation*}
\end{lemma}
\begin{proof}
    \begin{align*}
         \sum_{j=1}^i (1+c\alphaval)^{-(i-j)}  &= \sum_{k=0}^{i-1} (1+c\alphaval)^{-k}\\
         &= \frac{1-(1+c\alphaval)^{-i}}{1 - (1+c\alphaval)^{-1}}\\
         &= \frac{1+c\alphaval - (1+c\alphaval)^{-i+1}}{c\alphaval}
    \end{align*}
\end{proof}

\begin{lemma}\label{lem:sum_deriv}
Let $\gamma \neq 1$, $n \in \mathbb{N}^\ast$. We have
    \begin{equation}\label{eq:lem_sum_deriv}
     \sum_{k=1}^n k\gamma^k = \gamma\frac{1-(n+1)\gamma^{n} + n\gamma^{n+1}}{(1-\gamma)^2}.   
    \end{equation}
\end{lemma}
\begin{proof}
    Define $\varphi(\gamma) := \sum_{k=1}^n \gamma^k = \frac{\gamma - \gamma^{n+1}}{1-\gamma}$. We have
    \begin{equation*}
        \varphi'(y) = \sum_{k=1}^n k\gamma^{k-1} = \frac{(1-(n+1)\gamma^{n})(1-\gamma) + \gamma - \gamma^{n+1}}{(1-\gamma)^2} = \frac{1-(n+1)\gamma^{n} + n\gamma^{n+1}}{(1-\gamma)^2}.
    \end{equation*}
    Multiply by $\gamma$ allows to get \eqref{eq:lem_sum_deriv}. 

    
\end{proof}


\begin{lemma}\label{lem:exp_stopping_time}
    If $i\le j$, we have
    \begin{equation*}
        \mathbb{E}\left[ e^{-\alphaval(T_j - T_i)} \right] = (1+\alphaval)^{-j+i}.
    \end{equation*}
\end{lemma}
\begin{proof}
It is trivially true if $i = j$. If $i<j$, we have
    \begin{align*}
        e^{-\alphaval(T_j-T_i)} &=\prod_{k=0}^{j-i-1} e^{-\alphaval(T_{k+i+1}-T_{k+i})} 
    \end{align*}
    By independence of the $T_{k+1}-T_k$ for all $k$, we have
        \begin{align*}
        \mathbb{E}\left[\prod_{k=0}^{j-i-1} e^{-\alphaval(T_{k+i+1}-T_{k+i})} \right] &=\prod_{k=0}^{j-i-1} \mathbb{E}\left[e^{-\alphaval(T_{i+j+1}-T_{i+j})}\right].
    \end{align*}
Recall $T_{k+1}-T_k$ for all $k$ follows an exponential law with parameter $1$, such that we have $\mathbb{E}\left[e^{-\alphaval(T_{i+j+1}-T_{i+j})}\right] = (1+\alphaval)^{-1}$. So,
    \begin{align*}
    \prod_{k=0}^{j-i-1} \mathbb{E}\left[e^{-\alphaval(T_{i+j+1}-T_{i+j})}\right]&=\prod_{k=0}^{j-i-1} (1+\alphaval)^{-1}\\
         &=(1+\alphaval)^{-j+i}.
    \end{align*}
\end{proof}

\begin{lemma}\label{prop:calc_sto_comput_1}
For $n \ge 1$, $\alphaval > 0$, we have
\begin{enumerate}[label=(\roman*)]
    \item The following upper-bound \begin{equation*}
        \mathbb{E}\left[ \int_0^{T_n} \bpar{\int_0^t \alphaval e^{\alphaval(s-t)}dN_s}^2 dN_t\right]       \le     \bpar{1+\frac{3}{2}\alphaval}(1+2\alphaval)n.
    \end{equation*}
    \item  The following equality
    \begin{equation*}
        \begin{aligned}
             & \mathbb{E}\left[ \int_0^{T_n} \bpar{\int_0^t \alphaval e^{\alphaval(s-t)}dN_s}^2 dN_t\right] \\
             &=\frac{(2+\alphaval)(1+2\alphaval)}{2}n 
         + \frac{(1+2\alphaval)(1+\frac{3}{2}\alphaval)}{2\alphaval}(1- (1+2\alphaval)^{-n})\\
         &- \frac{2(1+\alphaval)(1+2\alphaval)}{\alphaval}(1-(1+\alphaval)^{-n})
        \end{aligned}
    \end{equation*}
\end{enumerate}

\end{lemma}
\begin{proof}
\textbf{Point (i) \:}
    We have 
    \begin{align}
        \int_0^{T_n} \bpar{\int_0^t \alphaval e^{\alphaval(s-t)}dN_s}^2 dN_t &= \sum_{i=1}^n \bpar{\sum_{j=1}^i \alphaval e^{\alphaval(T_j-T_i)}}^2\nonumber\\
        &= \alphaval^2\sum_{i=1}^n \sum_{j=1}^i  e^{-2\alphaval(T_i-T_j)} + 2 \alphaval^2 \sum_{i=1}^n \sum_{1 \le j < l \le i}  e^{-\alphaval(T_i-T_j)}e^{-\alphaval(T_i-T_l)}\label{eq:calc_sto_poiss_1}.
    \end{align}
    We focus on the first term in \eqref{eq:calc_sto_poiss_1}. We have by using Lemma~\ref{lem:exp_stopping_time} and Lemma~\ref{lem:sum_geometric} ($c =2$)
    \begin{align*}
         \mathbb{E}\left[ \alphaval^2\sum_{i=1}^n \sum_{j=1}^i  e^{-2\alphaval(T_i-T_j)}\right]&=  \alphaval^2\sum_{i=1}^n \sum_{j=1}^i   \mathbb{E}\left[ e^{-2\alphaval(T_i-T_j)}\right]
   \\
   &= \alphaval^2\sum_{i=1}^n \sum_{j=1}^i (1+2\alphaval)^{-(i-j)}  \\
   &= \alphaval^2 \sum_{i=1}^n \frac{1+2\alphaval - (1+2\alphaval)^{-i+1}}{2\alphaval}
    \end{align*}
    Omitting the non positive term, we have
    \begin{equation*}
         \mathbb{E}\left[ \alphaval^2\sum_{i=1}^n \sum_{j=1}^i  e^{-2\alphaval(T_i-T_j)}\right] \le \alphaval n \frac{1+2\alphaval}{2}.
    \end{equation*}
    We focus on the second term of \eqref{eq:calc_sto_poiss_1}. We have
    \begin{align*}
        2 \alphaval^2 \sum_{i=1}^n \sum_{1 \le j < l \le i}  e^{-\alphaval(T_i-T_j)}e^{-\alphaval(T_i-T_l)} &= 2\alphaval^2 \sum_{i=1}^n \sum_{1 \le j < l \le i}  e^{-2\alphaval(T_i-T_l)}e^{-\alphaval(T_l-T_j)}
    \end{align*}
    We take expectation and we use Lemma~\ref{lem:exp_stopping_time}
     \begin{align*}
        2 \alphaval^2 \mathbb{E}\left[\sum_i^n \sum_{1 \le j < l \le i}  e^{-\alphaval(T_i-T_j)}e^{-\alphaval(T_i-T_l)}\right]
        &=2\alphaval^2 \sum_{i=1}^n \sum_{1 \le j < l \le i}  \mathbb{E}\left[ e^{-2\alphaval(T_i-T_l)}e^{-\alphaval(T_l-T_j)} \right]\\
        &=2\alphaval^2 \sum_{i=1}^n \sum_{1 \le j < l \le i}  \mathbb{E}\left[ e^{-2\alphaval(T_i-T_l)}\right] \mathbb{E}\left[e^{-\alphaval(T_l-T_j)} \right]\\
        &=2\alphaval^2  \sum_{i=1}^n \sum_{1 \le j < l \le i} (1+2\alphaval)^{-(i-l)}(1+\alphaval)^{-(l-j)}\\
        &= 2\alphaval^2  \sum_{i=1}^n \sum_{l = 1}^i \sum_{j = 1}^{l-1} (1+2\alphaval)^{-(l-j)}  (1+\alphaval)^{-(i-j)}
        .
    \end{align*}
    Now, using twice Lemma~\ref{lem:sum_geometric} with $c =1$ and $c=2$, and omitting non negative terms
    \begin{align*}
         2\alphaval^2  \sum_{i=1}^n \sum_{l = 1}^i \sum_{j = 1}^{l-1} (1+2\alphaval)^{-(i-l)}  (1+\alphaval)^{-(i-j)} &= 2\alphaval^2  \sum_{i=1}^n \sum_{l = 1}^i (1+2\alphaval)^{-(i-l)} \sum_{j = 1}^{l-1}  (1+\alphaval)^{-(l-j)}\\
         &= 2\alphaval^2  \sum_{i=1}^n \sum_{l = 1}^i (1+2\alphaval)^{-(i-l)}\bpar{\frac{1+\alphaval - (1+\alphaval)^{-i+1}}{\alphaval}}\\
         &\le 2\alphaval(1+\alphaval)\sum_{i=1}^n \sum_{l = 1}^i (1+2\alphaval)^{-(i-l)}\\
         &=2\alphaval(1+\alphaval)\sum_{i=1}^n \bpar{\frac{1+2\alphaval - (1+2\alphaval)^{-i+1}}{2\alphaval}}\\
         &\le (1+\alphaval)(1+2\alphaval)n
    \end{align*}
    We then deduce that
    \begin{align*}
         2 \alphaval^2 \mathbb{E}\left[\sum_i^n \sum_{1 \le j < l \le i}  e^{-\alphaval(T_i-T_j)}e^{-\alphaval(T_i-T_l)}\right]  \le  (1+\alphaval)(1+2\alphaval)n.
    \end{align*}
    Finally, we have the bound
     \begin{align*}
        \mathbb{E}\left[ \int_0^{T_n} \bpar{\int_0^t \alphaval e^{\alphaval(s-t)}dN_s}^2 dN_t\right]&\le   \alphaval  \frac{1+2\alphaval}{2}n + (1+\alphaval)(1+2\alphaval)n \\
        &=\bpar{1+\frac{3}{2}\alphaval}(1+2\alphaval)n.
    \end{align*}
\textbf{Point (ii) \:}
    Let $Y_k = \bpar{\int_{[0,T_k]}\alphaval e^{\alphaval(s-T_k)}dN_s}^2$. We have
    \begin{align*}
        Y_{k+1} &= \bpar{\int_{[0,T_{k+1}]}\alphaval e^{\alphaval(s-T_{k+1})}dN_s}^2\\
        &=\bpar{e^{\alphaval(T_k-T_{k+1})}\int_{[0,T_{k}]}\alphaval e^{\alphaval(s-T_{k})}dN_s + \int_{]T_k,T_{k+1}]}\alphaval e^{\alphaval(s-T_{k+1})}dN_s}^2.
    \end{align*}
    We have 
    $\int_{]T_k,T_{k+1}]}\alphaval e^{\alphaval(s-T_{k+1})}dN_s = \alphaval$. So
    \begin{align*}
        Y_{k+1} &=e^{2\alphaval(T_k-T_{k+1})}Y_k + 2\alphaval e^{\alphaval(T_k-T_{k+1})}\int_{[0,T_{k}]}\alphaval e^{\alphaval(s-T_{k})}dN_s +  \alphaval^2.
    \end{align*}
    We take the expectation
        \begin{align*}
        \E{Y_{k+1}} &=\E{e^{2\alphaval(T_k-T_{k+1})}Y_k + 2\alphaval e^{\alphaval(T_k-T_{k+1})}\int_{[0,T_{k}]}\alphaval e^{\alphaval(s-T_{k})}dN_s +  \alphaval^2}\\
        &=\E{e^{2\alphaval(T_k-T_{k+1})}Y_k} + 2\alphaval \E{e^{\alphaval(T_k-T_{k+1})}\int_{[0,T_{k}]}\alphaval e^{\alphaval(s-T_{k})}dN_s} +  \alphaval^2\\
        &=\E{e^{2\alphaval(T_k-T_{k+1})}}\E{Y_k} + 2\alphaval \E{e^{\alphaval(T_k-T_{k+1})}}\E{\int_{[0,T_{k}]}\alphaval e^{\alphaval(s-T_{k})}dN_s} +  \alphaval^2\\
        &=(1+2\alphaval)^{-1}\E{Y_k} + 2\alphaval(1+\alphaval)^{-1}(1+\alphaval - (1+\alphaval)^{-k+1}) +  \alphaval^2\\
         &=(1+2\alphaval)^{-1}\E{Y_k} + 2\alphaval(1 - (1+\alphaval)^{-k}) +  \alphaval^2\\
         &=(1+2\alphaval)^{-1}\E{Y_k} -2\alphaval(1+\alphaval)^{-k} +\alphaval(2+\alphaval).
    \end{align*}
    Noting $u_k := \E{Y_k}$, we get a recursive formula of the form
    \begin{equation*}
        u_{k+1} = a u_k + b_k + c,
    \end{equation*}
    whose solution is given by
    \begin{equation}\label{eq:last_calc_0}
        u_k = a^k\sum_{i=0}^{k-1}\frac{b_i + c}{a^{i+1}} = \sum_{i=0}^{k-1}(b_i +c) a^{k-i-1}.
    \end{equation}
    We compute \eqref{eq:last_calc_0} in two steps. We start with
    \begin{equation}
        \begin{aligned}\label{eq:last_calc_1}
              \sum_{i=0}^{k-1}c a^{k-i-1}
        &= c\sum_{j=0}^{k-1}a^j \\
        &= c\frac{1-a^k}{1-a}\\
        &= \alphaval(2+\alphaval)\frac{1-(1+2\alphaval)^{-k}}{1-(1+2\alphaval)^{-1}}\\& = \frac{(2+\alphaval)(1+2\alphaval)}{2} - \frac{2+\alphaval}{2}(1+2\alphaval)^{-k+1}.
        \end{aligned}
    \end{equation}
  
    Then, we have
    \begin{equation}
        \begin{aligned}\label{eq:last_calc_2}
                      a^{k-1}\sum_{i=0}^{k-1}\frac{b_i}{a^{i}} &= -2\alphaval a^{k-1}\sum_{i=0}^{k-1} \bpar{\frac{1+2\alphaval}{1+\alphaval}}^i\\
        &=-2\alphaval a^{k-1}\bpar{1-\bpar{\frac{1+2\alphaval}{1+\alphaval}}^k}\bpar{1-\bpar{\frac{1+2\alphaval}{1+\alphaval}}}^{-1}\\
        &=2(1+\alphaval)\bpar{ a^{k-1} - a^{-1}(1+\alphaval)^{-k}}\\
        &=2(1+\alphaval)(1+2\alphaval)( (1+2\alphaval)^{-k} -(1+\alphaval)^{-k})\\
        &=2(1+\alphaval)(1+2\alphaval) (1+2\alphaval)^{-k} -2(1+\alphaval)(1+2\alphaval)(1+\alphaval)^{-k}.
        \end{aligned}
    \end{equation}
   
    Injecting \eqref{eq:last_calc_1} and \eqref{eq:last_calc_2} into \eqref{eq:last_calc_0}, we get the following expression for $u_k$
    \begin{align*}
        u_k &=\frac{(2+\alphaval)(1+2\alphaval)}{2} - \frac{2+\alphaval}{2}(1+2\alphaval)^{-k+1} +2(1+\alphaval)(1+2\alphaval) (1+2\alphaval)^{-k} \\
        &-2(1+\alphaval)(1+2\alphaval)(1+\alphaval)^{-k}\\
        &=\frac{(2+\alphaval)(1+2\alphaval)}{2} + \bpar{-\frac{(2+\alphaval)(1+2\alphaval)}{2}+2(1+\alphaval)(1+2\alphaval)} (1+2\alphaval)^{-k} \\
        &-2(1+\alphaval)(1+2\alphaval)(1+\alphaval)^{-k}\\
        &=\frac{(2+\alphaval)(1+2\alphaval)}{2} + (1+2\alphaval)(1+\frac{3}{2}\alphaval)(1+2\alphaval)^{-k} -2(1+\alphaval)(1+2\alphaval)(1+\alphaval)^{-k}.
    \end{align*}
    It remains to sum over $k=0,\dots,n$, to get
    \begin{align*}
        \sum_{k=1}^n u_k &= \frac{(2+\alphaval)(1+2\alphaval)}{2}n + \frac{(1+2\alphaval)(1+\frac{3}{2}\alphaval)}{2\alphaval}(1- (1+2\alphaval)^{-n}) \\
        &- \frac{2(1+\alphaval)(1+2\alphaval)}{\alphaval}(1-(1+\alphaval)^{-n}).
    \end{align*}
    To conclude, remark that 
    \begin{equation*}
        \mathbb{E}\left[ \int_0^{T_n} \bpar{\int_0^t \alphaval e^{\alphaval(s-t)}dN_s}^2 dN_t\right]      = \sum_{k=1}^n \E{\bpar{\int_{[0,T_k]} \alphaval e^{\alphaval(s-T_k)}dN_s}^2 } = \sum_{k=1}^n \E{Y_k} = \sum_{k=1}^n u_k.
    \end{equation*}
    \end{proof}
    \begin{lemma}\label{lem:cupp}
Let $\tilde \varepsilon \in (0,1/2]$, and $\cupp \ge 1$ such that it satisfies $1-\tilde \varepsilon-\log(1-\tilde \varepsilon) = \cupp - \log(\cupp)$. Then, we have $\cupp \le 1+2\tilde \varepsilon$.    
\end{lemma}
\begin{proof}
     We denote $h(t) = t-\log(t)$. By assumption, we have $h(\cupp) = h(1- \tilde \tilde \varepsilon)$. We have that $h$ is non-decreasing on $[1,+\infty)$. Then, if we show $h(1+c_0\tilde \varepsilon) \ge h(1- \tilde \varepsilon)$ for some $c_0 \ge 1$, it will induce $\cupp \le 1+c_0 \tilde \varepsilon$. First,
\[
h(1+c_0\tilde \varepsilon)-1
=
c_0\tilde \varepsilon-\log(1+c_0\tilde \varepsilon)
=
\int_0^{c_0\tilde \varepsilon}\frac{t}{1+t}\,dt.
\]
Since $t\mapsto (1+t)^{-1}$ is decreasing, for $t\in[0,c_0\tilde \varepsilon]$,
$
\frac{1}{1+t}\ge \frac{1}{1+c_0\tilde \varepsilon}$,
therefore,
\[
h(1+c_0\tilde \varepsilon)-1
\ge
\frac{1}{1+c_0\tilde \varepsilon}\int_0^{c_0\tilde \varepsilon} t\,dt
=
\frac{c_0^2\tilde \varepsilon^2}{2(1+c_0\tilde \varepsilon)}.
\]

On the other hand,
\[
h(1- \tilde \varepsilon)-1
=
-\tilde \varepsilon-\log(1- \tilde \varepsilon)
=
\int_0^{\tilde \varepsilon} \frac{t}{1-t}\,dt.
\]
Since $t\mapsto (1-t)^{-1}$ is increasing on $[0,\tilde \varepsilon]$,
$\frac{1}{1-t}\le \frac{1}{1- \tilde \varepsilon}$,
thus
\[
h(1- \tilde \varepsilon)-1
\le
\frac{1}{1- \tilde \varepsilon}\int_0^{\tilde \varepsilon} t\,dt
=
\frac{\tilde \varepsilon^2}{2(1- \tilde \varepsilon)}.
\]

Consequently, it is enough that
$\frac{c_0^2\tilde \varepsilon^2}{2(1+c_0\tilde \varepsilon)}
\ge
\frac{\tilde \varepsilon^2}{2(1- \tilde \varepsilon)}$,
or equivalently $
c_0^2(1- \tilde \varepsilon)\ge 1+c_0\tilde \varepsilon$.
This is equivalent to
\[
\tilde \varepsilon\le \frac{c_0^2-1}{c_0^2+c_0}
=
1-\frac{1}{c_0}.
\]
Therefore, for any $c_0>1$, choosing
$a:=1-\frac1{c_0}$
gives, for all $\tilde \varepsilon\in(0,a]$, $
h(1+c_0\tilde \varepsilon)\ge h(1- \tilde \varepsilon)=h(\cupp)$,
from which we deduce
\[\cupp\le 1+c_0\tilde \varepsilon.\]
The result is the particular case $c_0 = 2 \Rightarrow \tilde \varepsilon \in \Big(0,\frac{1}{2}\Big]$. 
\end{proof}


\end{document}